\documentclass[a4paper]{article}
\usepackage{amsmath,amsthm,amsfonts,amssymb,amscd,bm,makeidx,indentfirst,tocloft}
\renewcommand\cftsecfont{\normalsize}
\newcommand{\e}{\ensuremath{\epsilon}}
\newcommand{\ve}{\ensuremath{\varepsilon}}
\newcommand{\kk}{\ensuremath{\tilde{k}}}
\newcommand{\rto}{\ensuremath{\rightarrow}}
\newcommand{\lem}{\ensuremath{\lesssim}}
\newcommand{\W}{\ensuremath{\mathcal{W}}}

\newtheorem{theorem}{Theorem}[section]

\newtheorem{lemma}[theorem]{Lemma}

\newtheorem{remark}[theorem]{Remark}

\numberwithin{equation}{section}

\title{\Large Full Regularity and Well-Posedness of the Nonlinear Unsteady Prandtl Equations with Robin or Dirichlet Boundary Condition}
\author{\normalsize Fuzhou Wu\thanks{E-mail: michael8723@gmail.com; fuzhou.wu@yahoo.com; wufz12@mails.tsinghua.edu.cn} \\
\small\it  Yau Mathematical Sciences Center, Tsinghua University\\
\small\it  Beijing 100084, China \\[3pt]
\small\it  Center of Mathematical Sciences and Applications, Harvard University\\
\small\it  Cambridge, Massachusetts 02138, USA
}
\date{}

\begin{document}
\maketitle
\setlength\parindent{2em}
\setlength\parskip{5pt}

\begin{abstract}
\normalsize{
In this paper, we study the full regularity and well-posedness of classical solutions to the nonlinear unsteady Prandtl equations with Robin or Dirichlet boundary condition in half space. Under Oleinik's monotonicity assumption, we prove the large time existence of classical solutions to the nonlinear Prandtl equations with Robin or Dirichlet boundary condition, when both initial vorticity and the general Euler flow are sufficiently small. For the general Euler flow, the vertical velocity of the Prandtl flow is unbounded. We prove that the Prandtl solutions preserve the full regularities in our solution spaces. The uniqueness and stability are also proved in the weighted Sobolev spaces.
}
\end{abstract}

%\newpage
\tableofcontents

%%% find 1
\section{Introduction}
In this paper, we study the full regularities and well-posedness of classical solutions to the following Prandtl system in half space:
\begin{equation}\label{Sect1_PrandtlEq}
\left\{\begin{array}{ll}
u_t + u u_x + v u_y + p_x = u_{yy},\quad (x,y)\in\mathbb{R}_{+}^2, \ t>0, \\[9pt]
u_x + v_y =0, \\[9pt]
(u_y - \beta u)|_{y=0} = 0,\quad v|_{y=0} =0,\\[9pt]
\lim\limits_{y\rto +\infty}u = U(t,x), \\[10pt]
u|_{t=0} = u_0(x,y),
\end{array}\right.
\end{equation}
where $u,v$ denote the tangential and normal velocities of the boundary layer, with $y$ being the scaled normal variable to the boundary, the parameter $\beta>0$. $\omega=u_y$ is the vorticity. $U,p$ denote the values on the boundary of the tangential velocity and pressure of the Euler flow which satisfies the Bernoulli's law:
\begin{equation}\label{Sect1_Bernoulli_Law}
\begin{array}{ll}
U_t + U U_x + p_x = 0.
\end{array}
\end{equation}

For $(u_y - \beta u)|_{y=0} = 0$ in $(\ref{Sect1_PrandtlEq})$, $0<\beta<+\infty$ corresponds to Robin boundary condition.
$\beta=+\infty$ corresponds to the following Prandtl equations with Dirichlet boundary condition:
\begin{equation}\label{Sect1_PrandtlEq_Dirichlet}
\left\{\begin{array}{ll}
u_t + u u_x + v u_y + p_x = u_{yy},\quad (x,y)\in\mathbb{R}_{+}^2, \ t>0, \\[9pt]
u_x + v_y =0, \\[9pt]
u|_{y=0} = v|_{y=0} =0,\\[9pt]
\lim\limits_{y\rto +\infty}u = U(t,x), \\[10pt]
u|_{t=0} = u_0(x,y).
\end{array}\right.
\end{equation}

While $\beta=0$ corresponds to Neumann boundary condition.
To our best knowledge, the Prandtl equations with Neumann boundary condition have no physical background, and their well-posedness is unknown in mathematical viewpoint. In this paper, the parameter $\beta\rto 0+$ is not allowed.

$(\ref{Sect1_PrandtlEq_Dirichlet})$ was proposed by L. Prandtl (see \cite{Prandtl_1905}), while $(\ref{Sect1_PrandtlEq})$ was proposed in
$\cite{Wang_Wang_Xin_2010}$, which studied the asymptotic behaviors of the solutions to incompressible Navier-Stokes equations with Navier-slip boundary condition in which the slip length depends on the viscosity:
\begin{equation}\label{Sect1_NS_Equation}
\left\{\begin{array}{ll}
u_t + u u_x + v u_y + p_x = \e\triangle u, \\[7pt]
v_t + u v_x + v v_y + p_y = \e\triangle v, \\[7pt]
u_x + v_y =0, \\[7pt]
(\e^{\gamma}u_y - \beta u)|_{y=0} = 0,\quad v|_{y=0} =0, \\[7pt]
(u,v)|_{t=0} = (u_0,v_0)(x,y).
\end{array}\right.
\end{equation}
When $\gamma>\frac{1}{2}$ (super-critical), the leading boundary layer profile satisfies $(\ref{Sect1_PrandtlEq_Dirichlet})$.
When $\gamma=\frac{1}{2}$ (critical), the leading boundary layer profile satisfies $(\ref{Sect1_PrandtlEq})$ where $\beta\in (0,+\infty)$.
When $\gamma<\frac{1}{2}$ (sub-critical), the leading boundary layer profile appears in the $O(\e^{1-2\gamma})$ order terms of the solutions and satisfies the linearized Prandtl equations.

\subsection*{1.1\ Motivation of This Work}
The motivations for this work are as follows:

1. The establishment of the well-posedness of the Prandtl systems $(\ref{Sect1_PrandtlEq})$ and $(\ref{Sect1_PrandtlEq_Dirichlet})$ is the first step to study the inviscid limit of Navier-Stokes equations $(\ref{Sect1_NS_Equation})$. For $(\ref{Sect1_NS_Equation})$, One of interesting problems is to investigate convergence rate of the inviscid limit of $(\ref{Sect1_NS_Equation})$ which depends on the parameters $\e,\beta,\gamma$.
By now, we have no Sobolev well-posedness theory about Robin boundary problem $(\ref{Sect1_PrandtlEq})$ which exists objectively in physics.
Without developing boundary estimates, the well-posedness of the Robin boundary problem $(\ref{Sect1_PrandtlEq})$
can not be established.

When we develop a priori estimates which are uniform with respect to $\beta$,
the Robin boundary condition $\partial_{t,x}^{\alpha}u_y =\beta\partial_{t,x}^{\alpha}u$ can not simplify our boundary estimates.
In this paper, $\partial_{t,x}^{\alpha}u_y =\beta\partial_{t,x}^{\alpha}u$ is used to derive a new evolution equations on the boundary such that a priori estimates can be closed by coupling the evolution equations in the interior of the domain and the evolution equations on the boundary.

2. Without using  Crocco transformation, Nash-Moser-H$\ddot{o}$rmander iteration, uniform regularity approach, nonlinear cancellation method, mollification or regularization, we introduce new transformation of equations, new boundary conditions, new a priori estimates to prove the full regularities and the well-posedness
of classical solutions to the Prandtl equations $(\ref{Sect1_PrandtlEq})$ and $(\ref{Sect1_PrandtlEq_Dirichlet})$.

We want to know whether the Prandtl solutions preserve the full regularities and decay rates in our solution spaces, since the Prandtl equations $(\ref{Sect1_PrandtlEq})$ and $(\ref{Sect1_PrandtlEq_Dirichlet})$ have the vertical viscous terms but lack the horizontal viscous term, $v$ and its derivatives bring difficulties into each order estimates.
Though the Dirichlet boundary problem $(\ref{Sect1_PrandtlEq_Dirichlet})$ admits global weak solutions (see \cite{Xin_Zhang_2004}), classical solutions (see \cite{Masmoudi_Wong_2012,Alexandre_Yang_2014,Liu_Wang_Yang_2014,Wang_Xie_Yang_2014,Xu_Zhang_2015,Xu_Zhang_2015_NonMonotone}), we prove in this paper that our classical solutions preserve the full regularities
which are more than \cite{Masmoudi_Wong_2012,Alexandre_Yang_2014,Liu_Wang_Yang_2014,Wang_Xie_Yang_2014,Xu_Zhang_2015,Xu_Zhang_2015_NonMonotone})
and exist in the large time interval, we determine exactly the relationship between the regularities of Prandtl solutions and those of initial vorticity.

We want to know the relationship between the lifespan of Prandtl solutions and the size of the initial data, which needs elaborate estimates of the growth of weighted Sobolev norms by using differential inequalities or comparison principles.
 Note that the local existence is easier to prove by applying Gr$\ddot{o}$nwall's inequality to a priori estimates,
the global existence of classical Prandtl solutions is open.

3. We treat the general Euler flow, where $\partial_x U(t,x)\neq 0$ in general. As $y\rto +\infty$, $u_x(t,x,y)$ converges to $U_x(t,x)$ rather than
decay to zero, thus $v= -\int_0^y u_x(t,x,\tilde{y})\,\mathrm{d}\tilde{y}$ may diverge as $y\rto +\infty$.
$v$ grows with the order $O(y)$, thus we have to control $(1+y)^{-1} v$, where the weight $(1+y)^{-1}$ emerges in the a priori estimates due to the other faster-decaying terms.
Therefore, $v$ brings difficulties into our a priori estimates.
Note that when the Euler flow is constant, $v=O(1)$, thus the space weights and a priori estimates are simpler.

It is reasonable to assume the far field of the Euler flow is static, namely $U(t,x) \rto 0$ as $|x|\rto +\infty$,
$\|U(t,x)\|_{H^{|\alpha|+1}(\mathbb{R})}$ is bounded.
Obviously, we can simply assume the domain is periodic in x-direction, that is $\mathbb{T}\times\mathbb{R}_{+}$ (see \cite{Masmoudi_Wong_2012,Xin_Zhang_2004}).

4. Our a priori estimates are uniform with respect to $\beta$. As $\beta\rto +\infty$, the solutions of the Robin boundary problem $(\ref{Sect1_PrandtlEq})$ converge uniformly to those of the Dirichlet problem $(\ref{Sect1_PrandtlEq_Dirichlet})$ not only in the interior of the domain but also on its boundary.
Then we want to know the boundary behaviors of the solutions and their derivatives as $\beta\rto +\infty$.

\subsection*{1.2\ Survey of Previous Results}

For the nonlinear Prandtl equations, the known results are mainly about the Dirichlet boundary case $(\ref{Sect1_PrandtlEq_Dirichlet})$, where the solutions vanish on the boundary, namely $u|_{y=0}=v|_{y=0}=0$, we survey there some results:

After L. Prandtl (see \cite{Prandtl_1905}) proposed the Prandtl equations with Dirichlet boundary condition, their well-posedness theories attract much attention. Under Oleinik's monotonicity assumption $u_y>0$, the Prandtl equations with Dirichlet boundary condition can be reduced to a single quasilinear equation of $u_y$ via Crocco transformation, then O. A. Oleinik and V. N. Samokhin (see \cite{Oleinik_Samokhin_1999}) proved the local in time well-posedness. Under Oleinik's monotonicity assumption $u_y>0$ and favorable pressure condition $p_x\leq 0$, Xin and Zhang (see \cite{Xin_Zhang_2004}) proved the global existence of BV weak solutions via splitting viscosity method and Crocco transformation. This results are extended to three dimensional setting (see \cite{Liu_Wang_Yang_2015}).

Under Oleinik's monotonicity assumption, N. Masmoudi and T. K. Wong (see \cite{Masmoudi_Wong_2012}) proved local existence and uniqueness for 2D Prandtl equations in periodic domain $\mathbb{T}\times\mathbb{R}_{+}$ by using uniform regularity approach and the nonlinear cancelation of the vertical velocity. The vertical velocity is canceled by coupling the velocity equations and the vorticity equations.

Shear flow means the vertical velocity vanishes and the horizontal velocity approaches a constant as $y\rto +\infty$.
When the initial data is a small perturbation of a monotonic shear flow and Oleinik's monotonicity assumption is satisfied,
Alexander, Wang, Xu and Yang (see \cite{Alexandre_Yang_2014}) proved the well-posedness of the 2D Prandtl equations
by applying the energy method and Nash-Moser-H$\ddot{o}$rmander iteration. This framework makes solutions lose some regularities.
By using this framework, \cite{Liu_Wang_Yang_2014} proved the well-posedness of the 3D Prandtl equations under constraints on its flow structure,
\cite{Wang_Xie_Yang_2014} proved the well-posedness of 2D compressible flow.

By using uniform regularity approach (see \cite{Masmoudi_Wong_2012}),
the long time well-posedness was proved in \cite{Xu_Zhang_2015} when the initial data is a sufficient small perturbation of a monotonic shear flow,
\cite{Xu_Zhang_2015_NonMonotone} constructed a local in time solution as a perturbation of a non-monotonic shear flow.
However, there are a loss of regularities and a loss of decay rates.

However, we prove in this paper that the Prandlt solutions exist in the large time interval and have full regularities, our solution spaces are different from the above works.

When Oleinik's monotonicity assumption is violated, E and Engquist (see \cite{E_Engquist_1997}) proved the unsteady Prandtl equations do not have global strong solutions, namely, local solutions either do not exist or blow up; Grenier (see \cite{Grenier_2000}), Hong and Hunter (see \cite{Hong_Hunter_2003}) proved the nonlinear instability of the unsteady Prandtl equations; \cite{Varet_Dormy_2010,Varet_Nguyen_2010,Guo_Nguyen_2011} proved ill-posedness of Prandtl equations in Sobolev spaces for some data or in some weak sense.

Additionally, as to the nonlinear steady Prandtl equations with Dirichlet boundary condition, O. A. Oleinik (see \cite{Oleinik_1963}) used von Mise transformation to prove strong solutions are global in space for favorable pressure $p_x\leq 0$. While for adverse pressure $p_x>0$, boundary layer separation may happen (see \cite{Caffarelli_E_1995}).

Without Oleinik's monotonicity assumption, the data and solutions are required to be in the analytic or Gevrey classes.
For the data that are analytic in both $x$ and $y$ variables, the abstract Cauchy-Kowalewski theorem (see \cite{Safonov_1995}) can be applied, then the local existence of analytic solutions is proved in \cite{Sammartino_Caflisch_1998,Lombardo_Cannone_Sammartino_2003} for the Dirichlet boundary case, and in \cite{Ding_Jiang_2014} for the Robin boundary case.
For the data that are analytic in $x$ variable and have Sobolev regularity in $y$ variable, the existence is proved in \cite{Kukavica_Vicol_2013,Zhang_Zhang_2014} by using the energy method. For the data that belong to the Gevrey class $\frac{7}{4}$ in $x$ variable,
D. G$\acute{e}$rard-Varet and N. Masmoudi (see \cite{Varet_Masmoudi_2013}) proved local well-posedness. As to the Gevrey class regularity, see \cite{Li_Wu_Xu_2015}.

\subsection*{1.3\ Main Results of This Paper and Strategies of the Proofs}
\vspace{0.15cm}

The prandtl equations match the Bernoulli's law $(\ref{Sect1_Bernoulli_Law})$,
the term $v u_y$ appears in the Prandtl equations, but it disappears in the Bernoulli's law $(\ref{Sect1_Bernoulli_Law})$.
Since $v= O(y)$, $u_y$ must decay faster than $u_t + u u_x + p_x = u_t + u u_x - U_t -U U_x$.
If we assume $u_t + u u_x + p_x  = o(y^{-\zeta})$, $u_y = o(y^{-\zeta_1})$, then $\zeta_1-\zeta\geq 1$.
Similarly, $u_{yy}$ decays faster than $u_y$ due to $v$, etc. It is reasonable that the more y-derivatives the solution has,
the faster it decays. So the solutions must have algebraic decay rates. Thus, we introduce the functional norms
\begin{equation*}
\begin{array}{ll}
\|f\|_{\mathcal{H}_{\ell}^k(\Omega)} = \sum\limits_{|\alpha|+\sigma\leq k}\|(1+y)^{\ell+\sigma}
\partial_{t,x}^{\alpha}\partial_y^{\sigma} f\|_{L_{\ell+\sigma}^2(\Omega)},
\end{array}
\end{equation*}
where $\ell>1$, the weight $(1+y)^{\ell+\sigma}$ was introduced by \cite{Masmoudi_Wong_2012}.

In this paper, the time derivatives of initial data $u_0$ can be expressed in terms of the space derivatives of $u_0,v_0$ by solving the Prandtl equations.
The time derivatives, space derivatives of the initial data must satisfy the Prandtl equations, we say the initial data are compatible.

Under Oleinik's monotonicity assumption, we have
the following results for the Robin boundary problem $(\ref{Sect1_PrandtlEq})$:
\begin{theorem}\label{Sect1_Main_Thm}
Considering the nonlinear unsteady Prandtl equations with Robin boundary condition $(\ref{Sect1_PrandtlEq})$
under Oleinik's monotonicity assumption $\omega=u_y>0$.
Giving any integer $k\geq 6$, $U(t,x)\in C^{k+1}([0,+\infty)\times\mathbb{R})$ and $U(t,x)>0$,
we have the following existence, uniqueness and stability results:

$1$. For any fixed finite number $T\in (0,+\infty)$, there exist sufficiently small real number $0<\ve_1 = o(T^{-1})$ and
suitably large real numbers $\ell_0>1$, $\delta_{\beta}>0$ such that if $\ell\geq\ell_0$, $\beta\in [\delta_{\beta},+\infty)$, the compatible initial data satisfies
\begin{equation}\label{Sect1_Data_Conditions1}
\left\{\begin{array}{ll}
\omega_0>0, \quad u_0|_{y=0}>0, \quad
(\partial_y u_0 - \beta u_0)|_{y=0} =0, \quad \lim\limits_{y\rto +\infty} u_0 =U|_{t=0}, \\[6pt]

\|\omega_0\|_{\mathcal{H}_{\ell}^k(\mathbb{R}_{+}^2)} + \frac{1}{\sqrt{\beta}}\big\|\omega_0|_{y=0}\big\|_{\mathcal{H}_{\ell}^k(\mathbb{R})} \leq \ve_1,
\quad \|U(t,x)\|_{H^{k+1}([0,T]\times\mathbb{R})} \leq C_0\ve_1, \\[6pt]

0< c_1 (1+y)^{-\theta} \leq \omega_0 \leq c_2 (1+y)^{-\theta}, \ \theta>\frac{\ell+1}{2},
\end{array}\right.
\end{equation}
 then the Prandtl system $(\ref{Sect1_PrandtlEq})$ admits a unique classical solution $(\omega,u,v)$ in $[0,T]$ satisfying
\begin{equation}\label{Sect1_Solution_Regularity}
\begin{array}{ll}
\omega \in \mathcal{H}_{\ell}^k(\mathbb{R}_{+}^2), \hspace{1.3cm}
\omega,\ \omega_y \in \mathcal{H}_{\ell}^k([0,T]\times\mathbb{R}_{+}^2),
\\[8pt]

u-U \in \mathcal{H}_{\ell-1}^k(\mathbb{R}_{+}^2)\cap \mathcal{H}_{\ell-1}^k([0,T]\times\mathbb{R}_{+}^2),
\\[8pt]

\partial_y^{j} u|_{y=0} \in H^{k-j}(\mathbb{R})\cap H^{k-j}([0,T]\times\mathbb{R}), \hspace{0.3cm} 0\leq j\leq k, \\[8pt]

\partial_{t,x}^{\alpha} v + y\cdot \partial_{t,x}^{\alpha}\partial_x U
\in L_{y, \ell-1}^{\infty}(L_{t,x}^2), \hspace{0.3cm} |\alpha|\leq k-1.
\end{array}
\end{equation}

$2$. The classical solution to $(\ref{Sect1_PrandtlEq})$ is stable with respect to the initial data in the following sense: for any given two initial
data satisfying $(\ref{Sect1_Data_Conditions1})$, then for all $p\leq k-1$, the corresponding solutions of the Prandtl system $(\ref{Sect1_PrandtlEq})$ satisfy
\begin{equation}\label{Sect1_Stability}
\begin{array}{ll}
\|u^1-u^2\|_{\mathcal{H}_{\ell-1}^p(\mathbb{R}_{+}^2)}
+ \sum\limits_{j=0}^{p}\big\|\partial_y^j u^1|_{y=0} - \partial_y^j u^2|_{y=0} \big\|_{H^{p-j}(\mathbb{R})}
\\[10pt]\quad

+ \|\omega^1-\omega^2\|_{\mathcal{H}_{\ell}^p(\mathbb{R}_{+}^2)}
+\sum\limits_{|\alpha|\leq p-1} \|\partial_{t,x}^{\alpha} v^1 - \partial_{t,x}^{\alpha} v^2\|_{L_{y, \ell-1}^{\infty}(L_{t,x}^2)} \\[10pt]

\leq C(\ve_1,T)\big[\|\omega^1_0-\omega^2_0\|_{\mathcal{H}_{\ell}^p(\mathbb{R}_{+}^2)}
+ \frac{1}{\beta-\delta_{\beta}} \big\|\omega_0^1|_{y=0} -\omega_0^2|_{y=0}\big\|_{H^p(\mathbb{R})}^2 \big].
\end{array}
\end{equation}

$3$. As $\beta\rto +\infty$, $\big\|u|_{y=0}\big\|_{H^k(\mathbb{R})} =O(\frac{1}{\sqrt{\beta}})$,
$\big\|\omega|_{y=0}\big\|_{H^k(\mathbb{R})} =O(\sqrt{\beta})$ and $(\omega,u,v)$ satisfy
the regularities $(\ref{Sect1_Solution_Regularity})$ uniformly.
\end{theorem}

Next, we give some remarks on the results in Theorem $\ref{Sect1_Main_Thm}$:
\begin{remark}\label{Sect1_Remark}
(i) If $U(t,x)$ and $u|_{y=0}$ are large, we only have local existence of classical solutions.
In order to have the large time existence of $(\ref{Sect1_PrandtlEq})$, $U(t,x)$ must be small.
Since $\|\omega_0\|_{\mathcal{H}_{\ell}^k(\mathbb{R}_{+}^2)}\leq \ve_1$ implies $\frac{1}{\beta}\big\|\omega_0|_{y=0}\big\|_{H^k(\mathbb{R})}\lem \ve_1$
due to the trace theorem and Robin boundary condition,
we do not need $\frac{1}{\sqrt{\beta}}\big\|\omega_0|_{y=0}\big\|_{H^k(\mathbb{R})}\lem \ve_1$ if $\beta$ is bounded above.
However, in order to develop uniform estimates as $\beta\rto +\infty$, we needs
$\frac{1}{\sqrt{\beta}}\big\|\omega_0|_{y=0}\big\|_{H^k(\mathbb{R})}\lem \ve_1$.

(ii). $\ell_0>1$ is suitably large, the solutions decay very fast in the y-direction, then
some boundary terms of boundary estimates can be absorbed by the viscous terms of interior estimates.
We do not need the favorable pressure condition $p_x\leq 0$.
Without suitable largeness of $\ell_0$, $p_x\leq 0$ does not suffice to close higher order boundary estimates
in this paper.

(iii). When $\beta<+\infty$, due to the Robin boundary condition, $\omega|_{y=0}\in H^k(\mathbb{R})\cap H^k([0,T]\times\mathbb{R})$ and
$\big\|\omega|_{y=0}\big\|_{H^p(\mathbb{R})}$ is stable. In this paper, $\beta\rto 0+$ is not allowed, actually we need that $\beta\geq\delta_{\beta}$
is suitably large, such that $\beta-\frac{u_{yy}}{u_y}|_{y=0}$ is positive and away from zero, then no degeneracy arises on the boundary, which is necessary for the boundary estimates.

(iv). $U(t,x)\rto 0$ as $|x|\rto +\infty$, but $U(t,x)>0$ when $|x|\neq +\infty$, then Oleinik's monotonicity assumption $u_y>0$ makes sense.
If $U(t,x)$ approaches some positive constant as $|x|\rto +\infty$, then we assume the Prandtl flow is a small perturbation of a shear flow $(u^s,0)$, where $u^s$ satisfies the heat equation with Robin boundary condition. $u^s|_{y=0}$ provides a non-zero reference for $u|_{y=0}$ such that $\|u|_{y=0} -u^s|_{y=0}\|_{H^s(\mathbb{R})}$ can be bounded.
\end{remark}

The a priori estimates for the Robin boundary problem $(\ref{Sect1_PrandtlEq})$ are uniform with respect to the parameter $\beta$. By passing to the limit
$\beta\rto +\infty$, we have the following results for the Dirichlet boundary problem $(\ref{Sect1_PrandtlEq_Dirichlet})$ under Oleinik's monotonicity assumption.
\begin{theorem}\label{Sect1_Main_Thm_Dirichlet}
Considering the nonlinear unsteady Prandtl equations with Dirichlet boundary condition $(\ref{Sect1_PrandtlEq_Dirichlet})$
under Oleinik's monotonicity assumption $\omega=u_y>0$.
Giving any integer $k\geq 6$, $U(t,x)\in C^{k+1}([0,+\infty)\times\mathbb{R})$ and $U(t,x)>0$, we have the following existence, uniqueness and stability results:

$1$. For any fixed finite number $T\in (0,+\infty)$, there exist sufficiently small real number $0<\ve_2 =o(T^{-1})$
and suitably large real numbers $\ell_0>1$ such that if $\ell\geq\ell_0$, the compatible initial data and $U(t,x)$ satisfy
\begin{equation}\label{Sect1_Data_Conditions2}
\left\{\begin{array}{ll}
\omega_0>0, \quad u_0|_{y=0}=0, \quad \lim\limits_{y\rto +\infty} u_0 =U|_{t=0}, \\[9pt]

\|\omega_0\|_{\mathcal{H}_{\ell}^k(\mathbb{R}_{+}^2)} \leq \ve_2,
\quad \|U(t,x)\|_{H^{k+1}([0,T]\times\mathbb{R})} \leq C_0\ve_2,\\[8pt]

0< c_1 (1+y)^{-\theta} \leq \omega_0 \leq c_2 (1+y)^{-\theta}, \ \theta>\frac{\ell+1}{2},
\end{array}\right.
\end{equation}
then the Prandtl system $(\ref{Sect1_PrandtlEq_Dirichlet})$ admits a unique classical solution $(\omega,u,v)$ in $[0,T]$ satisfying
\begin{equation}\label{Sect1_Solution_Regularity_Dirichlet}
\begin{array}{ll}
\omega \in \mathcal{H}_{\ell}^k(\mathbb{R}_{+}^2), \hspace{1.3cm}
\omega,\ \omega_y \in \mathcal{H}_{\ell}^k([0,T]\times\mathbb{R}_{+}^2),
\\[8pt]

u-U \in \mathcal{H}_{\ell-1}^k(\mathbb{R}_{+}^2)\cap \mathcal{H}_{\ell-1}^k([0,T]\times\mathbb{R}_{+}^2),
\\[8pt]

\partial_y^{j} u|_{y=0} \in H^{k-j}(\mathbb{R})\cap H^{k-j}([0,T]\times\mathbb{R}), \hspace{0.3cm} 0\leq j\leq k, \\[8pt]

\partial_{t,x}^{\alpha} v + y\cdot \partial_{t,x}^{\alpha}\partial_x U
\in L_{y, \ell-1}^{\infty}(L_{t,x}^2), \hspace{0.3cm} |\alpha|\leq k-1.
\end{array}
\end{equation}

$2$. The classical solution to $(\ref{Sect1_PrandtlEq})$ is stable with respect to the initial data in the following sense: for any given two initial data satisfy $(\ref{Sect1_Data_Conditions2})$, then for all $p\leq k-1$, the corresponding solutions of the Prandtl system $(\ref{Sect1_PrandtlEq_Dirichlet})$ satisfy
\begin{equation}\label{Sect1_Solution_Stable_Dirichlet}
\begin{array}{ll}
\|u^1-u^2\|_{\mathcal{H}_{\ell-1}^p(\mathbb{R}_{+}^2)}
+ \sum\limits_{j=0}^{p}\big\|\partial_y^j u^1|_{y=0} - \partial_y^j u^2|_{y=0} \big\|_{H^{p-j}(\mathbb{R})}

\\[10pt]\quad
+\|\omega^1-\omega^2\|_{\mathcal{H}_{\ell}^p(\mathbb{R}_{+}^2)}
+\sum\limits_{|\alpha|\leq p-1} \|\partial_{t,x}^{\alpha} v^1 - \partial_{t,x}^{\alpha} v^2\|_{L_{y, \ell-1}^{\infty}(L_{t,x}^2)} \\[10pt]

\leq  C(\ve_2,T)\|\omega^1_0-\omega^2_0\|_{\mathcal{H}_{\ell}^p(\mathbb{R}_{+}^2)}.
\end{array}
\end{equation}
\end{theorem}

\begin{remark}\label{Sect1_Remark_Dirichlet}
(i) Since $u|_{y=0} =0$ is fixed and $\|\omega\|_{\mathcal{H}_{\ell}^k}$ is small, $\|U(t,x)\|_{H^{k}}$ can not be arbitrarily large.
$\partial_y^{j} u|_{y=0} \in H^{k-j}(\mathbb{R})\cap H^{k-j}([0,T]\times\mathbb{R}), 0\leq j\leq k$ in $(\ref{Sect1_Solution_Regularity_Dirichlet})$
and the stability of $\sum_{j=0}^{p}\big\|\partial_y^j u^1|_{y=0} - \partial_y^j u^2|_{y=0} \big\|_{H^{p-j}(\mathbb{R})}$
in $(\ref{Sect1_Solution_Stable_Dirichlet})$
are derived by using the trace theorem rather than direct a priori estimates.

(ii). When $\ell_0>1$ is suitably large, we do not need the favorable pressure condition $p_x\leq 0$.
Without suitable largeness of $\ell_0$, $p_x\leq 0$ does not suffice to close higher order boundary estimates
in this paper. In \cite{Xin_Zhang_2004}, $p_x\leq 0$ is one of necessary conditions for the global existence of BV weak solutions of
the Dirichlet boundary problem $(\ref{Sect1_PrandtlEq_Dirichlet})$.

(iii). For the Dirichlet boundary problem $(\ref{Sect1_PrandtlEq_Dirichlet})$, our methods, estimates, iteration scheme and solution spaces are different from those of \cite{Alexandre_Yang_2014,Masmoudi_Wong_2012}.

(iv). $U(t,x)\rto 0$ as $|x|\rto +\infty$, but $U(t,x)>0$ when $|x|\neq +\infty$, then Oleinik's monotonicity assumption $u_y>0$ makes sense.
If $U(t,x)$ approaches some positive constant $\bar{U}$ as $|x|\rto +\infty$, we do not need using the shear flow $(u^s,0)$, since $u|_{y=0} =0$
and we can bound $\|U(t,x)-\bar{U}\|_{H^s(\mathbb{R})}$ for $0\leq s\leq k+1$.
\end{remark}

\vspace{0.25cm}
Finally, we show our strategies of our proofs.

Step 1. A priori estimates and full regularities of classical solutions.

Before constructing classical Prandtl solutions, we want to know the regularities and space decay rates of classical Prandtl solutions, find out suitable solution spaces, determine the relationship between the lifespan of Prandtl solutions and the size of the initial data.
Note that the space decay rates play an important role in proving the full regularities of classical solutions.

Denote $\tilde{u}(t,x,y) =u(t,x,y)-U(t,x)$, $\partial_{t,x}^{\alpha} = \partial_t^{\sigma_1}\partial_x^{\sigma_2}$ where $\alpha=(\sigma_1,\sigma_2)$,
$W_{\alpha} = u_y\partial_y\big(\frac{\partial_{t,x}^{\alpha}(u-U)}{u_y}\big)$ and
$W_{\alpha,\sigma} =u_y\partial_y\big(\frac{\partial_{t,x}^{\alpha}\partial_y^{\sigma}(u-U)}{u_y}\big)$. Note that $u_y>0$, we have the following system for
$W_{\alpha}$:
\begin{equation}\label{Sect1_Existence_VorticityEq_1}
\left\{\begin{array}{ll}
\partial_t W_{\alpha} + u\partial_x W_{\alpha} - \partial_{yy} W_{\alpha} + Q_1 W_{\alpha} \\[7pt]\quad
+ u_y\partial_y(\frac{\partial_{t,x}^{\alpha} \tilde{u}}{u_y}\frac{u_{yt}}{u_y})
+ u u_y\partial_y(\frac{\partial_{t,x}^{\alpha} \tilde{u}}{u_y}\frac{u_{yx}}{u_y})
- u_y\partial_y(\frac{\partial_{t,x}^{\alpha} \tilde{u}}{u_y}\frac{u_{yyy}}{u_y})\\[8pt]\quad

= - u_y\partial_y\frac{[\partial_{t,x}^{\alpha},\, u_y]v}{u_y}
- u_y\partial_y\frac{[\partial_{t,x}^{\alpha},\, u\partial_x]\tilde{u}}{u_y}
- u_y\partial_y\frac{[\partial_{t,x}^{\alpha},\, \tilde{u}\partial_x]U}{u_y}
-\tilde{u} u_y\partial_y\frac{\partial_x\partial_{t,x}^{\alpha} U}{u_y}, \\[14pt]

\frac{1}{\beta - \frac{u_{yy}}{u_y}}(\partial_t W_{\alpha} + u\partial_x W_{\alpha})
- \partial_y W_{\alpha} - \frac{u_{yy}}{u_y}W_{\alpha}
+ Q_2 \cdot\frac{1}{\beta - \frac{u_{yy}}{u_y}}W_{\alpha}

\\[8pt]\quad
= -[\partial_{t,x}^{\alpha},\,u\partial_x]\tilde{u} -[\partial_{t,x}^{\alpha},\,\tilde{u}\partial_x]U
+ Q_3, \quad y=0, \\[12pt]

W_{\alpha}|_{t=0} = \partial_y u_0(x,y)
\partial_y\big(\frac{\partial_{t,x}^{\alpha}u_0(x,y)-\partial_{t,x}^{\alpha} U_0(x)}{\partial_y u_0(x,y)}\big),
\end{array}\right.
\end{equation}
where the lower order terms $Q_1, Q_2, Q_3$ are defined as
\begin{equation}\label{Sect1_Quantity_Definition_1}
\begin{array}{ll}
Q_1 := - \frac{u_{yt}}{u_y} - \frac{u u_{yx}}{u_y} - \frac{u_{yyy}}{u_y} +2(\frac{u_{yy}}{u_y})^2, \\[8pt]

Q_2 := \frac{(\frac{u_{yy}}{u_y})_t}{\beta - \frac{u_{yy}}{u_y}}
+ \frac{u(\frac{u_{yy}}{u_y})_x}{\beta - \frac{u_{yy}}{u_y}}
 - \frac{u_{yyy}}{u_y}, \quad y=0,
\\[10pt]

Q_3 :=-\tilde{u}\partial_x\partial_{t,x}^{\alpha}U
+ \partial_t\big[\frac{\beta}{\beta - \frac{u_{yy}}{u_y}} \partial_{t,x}^{\alpha} U\big]
+ u\partial_x\big[\frac{\beta}{\beta - \frac{u_{yy}}{u_y}} \partial_{t,x}^{\alpha} U\big] \\[6pt]\hspace{1cm}
- \frac{u_{yyy}}{u_y}\cdot
\frac{\beta}{\beta - \frac{u_{yy}}{u_y}} \partial_{t,x}^{\alpha} U.
\end{array}
\end{equation}

\vspace{-0.2cm}
The equations $(\ref{Sect1_Existence_VorticityEq_1})$ produce the estimates for $W_{\alpha}$ as follows:
\begin{equation}\label{Sect1_Existence_Estimate1}
\begin{array}{ll}
\frac{\mathrm{d}}{\mathrm{d}t} \|W_{\alpha}\|_{L_{\ell}^2(\mathbb{R}_{+}^2)}^2
+\frac{\mathrm{d}}{\mathrm{d}t} \int\limits_{\mathbb{R}} \frac{1}{\beta - \frac{u_{yy}}{u_y}|_{y=0}}
 (W_{\alpha}|_{y=0})^2 \,\mathrm{d}x
+ \|\partial_y W_{\alpha}\|_{L_{\ell}^2(\mathbb{R}_{+}^2)}^2 \\[10pt]

\lem \mathcal{L}\sum\limits_{|\alpha^{\prime}|\leq |\alpha|}\|W_{\alpha^{\prime}}\|_{L_{\ell}^2(\mathbb{R}_{+}^2)}^2
+ \mathcal{L}\sum\limits_{|\alpha^{\prime}|\leq |\alpha|-1}\|W_{\alpha^{\prime},1}\|_{L_{\ell+1}^2(\mathbb{R}_{+}^2)}^2

+ \mathcal{L}\|\omega_y\|_{L^2_{\ell+1}(\mathbb{R}_{+}^2)}^2 \\[10pt]\quad
+ \mathcal{L}\sum\limits_{|\alpha^{\prime}|\leq |\alpha|}\int\limits_{\mathbb{R}} \frac{1}{\beta -\frac{u_{yy}}{u_y}} (W_{\alpha^{\prime}})^2
\,\mathrm{d}x
+ \|U\|_{H^{|\alpha|+1}(\mathbb{R})}^2,\ |\alpha|\leq k,
\end{array}
\end{equation}
where $\mathcal{L}$ represents the lower order terms relating to $\tilde{u}$ or $\omega$, which is bounded by
$(\sum\limits_{|\alpha|+\sigma\leq 6}\|W_{\alpha,\sigma}\|_{L_{\ell+\sigma}^2}^2)^{\frac{s}{2}}$, where $s\geq 1$.
$W_{0,1}\equiv 0$, but $\|\omega_y\|_{L^2_{\ell+1}(\mathbb{R}_{+}^2)}^2$ appears in the estimates.
$f\lem g$ means there exists a constant $C>0$ such that $f\leq Cg$.

When $\sigma>0$, we have the following equations for $W_{\alpha,\sigma}$:
\begin{equation}\label{Sect1_Existence_VorticityEq_2}
\left\{\begin{array}{ll}
\partial_t W_{\alpha,\sigma} + u\partial_x W_{\alpha,\sigma}
- \partial_{yy} W_{\alpha,\sigma} +Q_1 W_{\alpha,\sigma} \\[6pt]\quad

+ u_y\partial_y[\frac{\partial_{t,x}^{\alpha}\partial_y^{\sigma} u}{u_y}\frac{u_{yt}}{u_y}]
+ u u_y\partial_y[\frac{\partial_{t,x}^{\alpha}\partial_y^{\sigma} u}{u_y}\frac{u_{yx}}{u_y}]
- u_y\partial_y\big[\frac{\partial_{t,x}^{\alpha}\partial_y^{\sigma} u}{u_y} \frac{u_{yyy}}{u_y} \big] \\[9pt]\quad

=  - u_y\partial_y\frac{[\partial_{t,x}^{\alpha}\partial_y^{\sigma},\, u\partial_x]u}{u_y}
- u_y\partial_y\frac{[\partial_{t,x}^{\alpha}\partial_y^{\sigma},\, u_y]v}{u_y}, \\[7pt]

-\partial_y W_{\alpha,\sigma} - \frac{u_{yy}}{u_y}W_{\alpha,\sigma}

= \mathcal{P}_1\Big(\sum\limits_{m=0}^{\sigma_1-1} (W_{\alpha_1+(1,0),\sigma_1-1-m}) (\frac{u_{yy}}{u_y})^m \\[7pt]\quad
 + \frac{1}{\beta - \frac{u_{yy}}{u_y}}(W_{\alpha_1+(1,0)} - \beta \partial_x\partial_{t,x}^{\alpha_1} U)(\frac{u_{yy}}{u_y})^{\sigma_1},

 \sum\limits_{m=0}^{\sigma_2-1} (W_{\alpha_2+(0,1),\sigma_2-1-m} (\frac{u_{yy}}{u_y})^m \\[7pt]\quad
 + \frac{1}{\beta - \frac{u_{yy}}{u_y}}(W_{\alpha_2+(0,1)} - \beta \partial_{t,x}^{\alpha_2} U)(\frac{u_{yy}}{u_y})^{\sigma_2} \Big),
 \quad y=0, \\[9pt]

W_{\alpha,\sigma}|_{t=0} = \partial_y u_0(x,y) \partial_y(\frac{\partial_{t,x}^{\alpha}\partial_y^{\sigma} u_0(x,y)}{\partial_y u_0(x,y)}).
\end{array}\right.
\end{equation}
where $\alpha_1\leq\alpha, \alpha_2\leq\alpha, \sigma_1\leq\sigma,\sigma_2\leq\sigma$,
$\mathcal{P}_1$ is a polynomial whose degree is less than or equal to $2$, the explicit form of $\mathcal{P}_1$
is given by the right hand side of $(\ref{Appendix_2_Sigma1_Existence_BC_5})$.

$(\ref{Sect1_Existence_VorticityEq_2})$ produces the following estimates:
\begin{equation*}
\begin{array}{ll}
\frac{\mathrm{d}}{\mathrm{d}t} \|W_{\alpha,\sigma}\|_{L_{\ell+\sigma}^2(\mathbb{R}_{+}^2)}^2
+ \|\partial_y W_{\alpha,\sigma}\|_{L_{\ell+\sigma}^2(\mathbb{R}_{+}^2)}^2  \\[11pt]\quad

\leq C\mathcal{L}\sum\limits_{|\alpha^{\prime}|\leq |\alpha|-1}\|W_{\alpha^{\prime},\sigma+1}\|_{L^2_{\ell+\sigma+1}(\mathbb{R}_{+}^2)}^2
+ C\mathcal{L}\sum\limits_{|\alpha^{\prime}|\leq |\alpha|,\sigma^{\prime}\leq \sigma}\|W_{\alpha^{\prime},\sigma^{\prime}}\|_{L^2_{\ell+\sigma^{\prime}}(\mathbb{R}_{+}^2)}^2 \\[12pt]\quad

+ q \sum\limits_{|\alpha^{\prime}|\leq |\alpha|+1,\sigma^{\prime}\leq \sigma-1}\|\partial_y W_{\alpha^{\prime},\sigma^{\prime}}\|_{L_{\ell+\sigma^{\prime}}^2(\mathbb{R}_{+}^2)}^2
+ C\|U\|_{H^{|\alpha|+1}(\mathbb{R})}^2  \\[12pt]\quad

 + C\mathcal{L}\int\limits_{\mathbb{R}}\frac{1}{\beta - \frac{u_{yy}}{u_y}}(W_{\alpha+(1,0)}^2 + W_{\alpha+(0,1)}^2) \,\mathrm{d}x,

 \ 0<\sigma\leq k-1, \ 0<|\alpha|\leq k-\sigma,
\end{array}
\end{equation*}

\begin{equation}\label{Sect1_Existence_Estimates2}
\begin{array}{ll}
\frac{\mathrm{d}}{\mathrm{d}t} \|W_{\alpha,\sigma}\|_{L_{\ell+\sigma}^2(\mathbb{R}_{+}^2)}^2
+ \|\partial_y W_{\alpha,\sigma}\|_{L_{\ell+\sigma}^2(\mathbb{R}_{+}^2)}^2

\leq  C\mathcal{L}\sum\limits_{|\alpha^{\prime}|\leq |\alpha|,\sigma^{\prime}\leq \sigma}\|W_{\alpha^{\prime},\sigma^{\prime}}\|_{L^2_{\ell+\sigma^{\prime}}(\mathbb{R}_{+}^2)}^2  \\[12pt]\quad

+ q \sum\limits_{|\alpha^{\prime}|\leq |\alpha|+1,\sigma^{\prime}\leq \sigma-1}\|\partial_y W_{\alpha^{\prime},\sigma^{\prime}}\|_{L_{\ell+\sigma^{\prime}}^2(\mathbb{R}_{+}^2)}^2

+ C\|U\|_{H^{|\alpha|+1}(\mathbb{R})}^2  \\[12pt]\quad
 + C\mathcal{L}\int\limits_{\mathbb{R}}\frac{1}{\beta - \frac{u_{yy}}{u_y}}(W_{\alpha+(1,0)}^2 + W_{\alpha+(0,1)}^2) \,\mathrm{d}x,

 \ 0<\sigma\leq k,\ |\alpha|=0,
\end{array}
\end{equation}
where $q\in (0,1)$ is sufficiently small when $\ell_0$ is suitably large.

Note that by using Hardy's inequality (see \cite{Masmoudi_Wong_2012}, Lemma B.1),
the terms like $u_y\partial_y(\frac{\partial_{t,x}^{\alpha}\tilde{u}}{u_y}F)$ are bounded by $W_{\alpha}$,
the terms like $u_y\partial_y(\frac{\partial_{t,x}^{\alpha}\partial_y^{\sigma}\tilde{u}}{u_y}F)$ are bounded by $W_{\alpha,\sigma}$.
Then the a priori estimates for $(\ref{Sect1_Existence_VorticityEq_1})$ and $(\ref{Sect1_Existence_VorticityEq_2})$ are easy to close.

Since only $W_{0,1}\equiv 0$, we have to estimate $\|\omega_y\|_{L^2_{\ell+1}}$ directly.
Denote $\tilde{W} =\omega_y$, we have the following equations:
\begin{equation}\label{Sect1_VorticityY_Eq}
\left\{\begin{array}{ll}
\tilde{W}_t + u\tilde{W}_x +v\tilde{W}_y -u_x\tilde{W} - \omega_x\int\limits_y^{+\infty} \tilde{W}\,\mathrm{d}\tilde{y} =\tilde{W}_{yy},\quad (x,y)\in\mathbb{R}_{+}^2, \ t>0, \\[5pt]
\tilde{W}_y = u_{yt} + u u_{yx}, \quad y=0,\\[5pt]
\tilde{W}|_{t=0} = \partial_{yy} u_0(x,y).
\end{array}\right.
\end{equation}
Then $(\ref{Sect1_VorticityY_Eq})$ has the estimate of $\|\omega_y\|_{L_{\ell+1}^2} = \|\tilde{W}\|_{L_{\ell+1}^2}$:
\begin{equation}\label{Sect1_VorticityY_Estimate}
\begin{array}{ll}
\frac{\mathrm{d}}{\mathrm{d}t}\|\omega_y\|_{L_{\ell+1}^2(\mathbb{R}_{+}^2)}^2
+ \|\omega_{yy}\|_{L_{\ell+1}^2(\mathbb{R}_{+}^2)}^2 \lem \mathcal{L}\|\omega_y\|_{L_{\ell+1}^2(\mathbb{R}_{+}^2)}^2 + \|U\|_{H^1(\mathbb{R})}^2
+ \mathcal{L} \\[8pt]\hspace{5.4cm}

+ q\|\partial_y W_{(1,0)}\|_{L_{\ell}^2(\mathbb{R}_{+}^2)}^2 + q\|\partial_y W_{(0,1)}\|_{L_{\ell}^2(\mathbb{R}_{+}^2)}^2.
\end{array}
\end{equation}

\vspace{-0.2cm}
Our a priori estimates are uniform with respect to $\beta$.
Note that let $\beta\rto +\infty$, the limit of $(\ref{Sect1_Existence_VorticityEq_1})_2$ is the following equation, which is exactly
the boundary condition for the Dirichlet boundary problem $(\ref{Sect1_PrandtlEq_Dirichlet})$:
\begin{equation}\label{Sect1_BC_Dirichlet_1}
\begin{array}{ll}
-\partial_y W_{\alpha} - \frac{u_{yy}}{u_y}W_{\alpha}
= \partial_t \partial_{t,x}^{\alpha} U +U\partial_x \partial_{t,x}^{\alpha} U
+ [\partial_{t,x}^{\alpha},\, U\partial_x]U - \frac{u_{yyy}}{u_y} \partial_{t,x}^{\alpha} U.
\end{array}
\end{equation}

Without coupling the estimates $(\ref{Sect1_Existence_Estimate1}),
(\ref{Sect1_Existence_Estimates2}),
(\ref{Sect1_VorticityY_Estimate})$ together, a priori estimates can not be closed.
By summing the estimates $(\ref{Sect1_Existence_Estimate1}),
(\ref{Sect1_Existence_Estimates2}),
(\ref{Sect1_VorticityY_Estimate})$ together and apply differential inequalities to the sum, we can
control the growth of Prandtl solutions, determine the relationship between the lifespan of Prandtl solutions and the size of the initial data.
Thus, giving any fixed finite $T$, we can find out a class of sufficiently small data such that there exist a Prandtl solution in $[0,T]$.

\vspace{0.25cm}
Step 2. The iteration scheme and the existence of the Prandtl systems.

We construct the Prandtl solutions to the Robin boundary problem $(\ref{Sect1_PrandtlEq})$. Note that the Prandtl solutions to the Dirichlet boundary problem $(\ref{Sect1_PrandtlEq_Dirichlet})$ are constructed similarly.

Assume we have a sequence of approximate solutions $\{(u^n, v^n)\}$ of the Prandtl system $(\ref{Sect1_PrandtlEq})$, where the zero-th order
approximate solution is chosen as the initial data, i. e.,
$(u^0,v^0) \equiv (u_0,v_0)$,
then we define the equations of the $(n+1)$-order approximate solution $\{(u^{n+1}, v^{n+1})\}$ as follows:
\begin{equation}\label{Sect1_PrandtlEq_Iteration}
\left\{\begin{array}{ll}
\partial_t u^{n+1} + u^n \partial_x u^{n+1} + v^{n+1} \partial_y u^n + p_x = \partial_{yy} u^{n+1},\quad (x,y)\in\mathbb{R}_{+}^2, \ t>0, \\[6pt]
\partial_x u^{n+1} + \partial_y v^{n+1} =0, \\[6pt]
(\partial_y u^{n+1} - \beta u^{n+1})|_{y=0} = 0,\quad v^{n+1}|_{y=0} =0,\\[6pt]
\lim\limits_{y\rto +\infty}u^{n+1} = U(t,x), \\[7pt]
u^{n+1}|_{t=0} = u_0(x,y).
\end{array}\right.
\end{equation}
Note that for the Dirichlet boundary problem $(\ref{Sect1_PrandtlEq_Dirichlet})$, the boundary conditions in $(\ref{Sect1_PrandtlEq_Iteration})$
need to be replaced with $u^{n+1}|_{y=0} = v^{n+1}|_{y=0} =0$. Our iteration scheme $(\ref{Sect1_PrandtlEq_Iteration})$ is different from \cite{Masmoudi_Wong_2012,Alexandre_Yang_2014,Liu_Wang_Yang_2014,Wang_Xie_Yang_2014,Xu_Zhang_2015,Xu_Zhang_2015_NonMonotone}.

Similar to $(\ref{Sect1_Existence_VorticityEq_1})$, we have the equations of
$W_{\alpha}^{n+1} = \partial_y u^n \partial_y(\frac{\partial_{t,x}^{\alpha}(u^{n+1}-U)}{\partial_y u^n})$,
see $(\ref{Sect3_Existence_VorticityEq_1})$. When $\sigma\geq 1$, we have the equation of
$W_{\alpha,\sigma}^{n+1} = \partial_y u^n \partial_y(\frac{\partial_{t,x}^{\alpha}\partial_y^{\sigma}u^{n+1}}{\partial_y u^n})$,
see $(\ref{Sect3_Existence_VorticityEq_2})$.
Since the equations $(\ref{Sect1_PrandtlEq_Iteration})$ are linear,
it is much easier to prove the estimates of approximate solutions
$\{(\omega^{n+1},u^{n+1},v^{n+1})|n\geq 0\}$ than the a priori estimates of the Prandtl solutions developed in Step 1.
Note that $W_{0,1}^{n+1} \neq 0$ here,
$u^n$ satisfies Oleinik's monotonicity assumption $\omega^n>0$.

The equations of approximate solutions $W_{\alpha,\sigma}^{n+1}$ are linear, but $(u^{n},v^{n})$ also grows. However, similar to Step 1, we can control the growth of approximate solutions, determine the relationship between the lifespan of approximate solutions and the size of the initial data. The limits of the approximate solutions lie in our solution spaces which are complete.
Thus, the Prandtl solutions can be constructed.

\vspace{0.3cm}
Step 3. A priori estimates and the stability of the Prandtl systems.

Suppose $(u^1,v^1)$ and $(u^2,v^2)$ are two classical Prandtl solutions with data $u_0^1(x,y)$ and $u_0^2(x,y)$ respectively,
denote
\begin{equation}\label{Sect1_Define_Variables}
\begin{array}{ll}
\delta u=u^1-u^2,\quad \delta v=v^1-v^2,\quad \delta \omega = \partial_y \delta u, \\[5pt]

\bar{u}=\frac{u^1+u^2}{2}, \quad \bar{v}=\frac{v^1+v^2}{2}, \quad \bar{\omega} = \frac{\omega^1+\omega^2}{2},
\end{array}
\end{equation}
then $(\delta u,\delta v)$ satisfies the following system:
\begin{equation}\label{Sect1_SolDifference_Eq}
\left\{\begin{array}{ll}
(\delta u)_t + \bar{u} (\delta u)_x + (\delta u) \bar{u}_x + \bar{v} (\delta u)_y + (\delta v) \bar{u}_y - (\delta u)_{yy} = 0, \\[6pt]
(\delta u)_x + (\delta v)_y =0, \\[6pt]
((\delta u)_y - \beta (\delta u))|_{y=0} =0,\quad  \delta v|_{y=0} =0,\\[6pt]
 \lim\limits_{y\rto +\infty} \delta u(t,x,y) =0, \\[6pt]
\delta u|_{t= 0} = u_0^1 - u_0^2.
\end{array}\right.
\end{equation}
Note that when $\beta=+\infty$, the boundary conditions in $(\ref{Sect1_SolDifference_Eq})$ should be replaced with $\delta u|_{y=0} =\delta v|_{y=0}=0$.
The first equation $(\ref{Sect1_SolDifference_Eq})_1$ was used in \cite{Alexandre_Yang_2014}.

Denote $\W_{\alpha} = \bar{u}_y\partial_y(\frac{\partial_{t,x}^{\alpha}\delta u}{\bar{u}_y})$ and
$\W_{\alpha,\sigma} = \bar{u}_y\partial_y(\frac{\partial_{t,x}^{\alpha}\partial_y^{\sigma}\delta u}{\bar{u}_y})$. Note that $\bar{u}_y>0$, we have the following system for $\W_{\alpha,\sigma}$:
\begin{equation}\label{Sect1_Stability_VorticityEq_1}
\left\{\begin{array}{ll}
\partial_t \W_{\alpha,\sigma} + \bar{u}\partial_x \W_{\alpha,\sigma} + \bar{v}\partial_y \W_{\alpha,\sigma} - \partial_{yy} \W_{\alpha,\sigma}
+ Q_4\W_{\alpha,\sigma} + \bar{u}_y\partial_y[\frac{\partial_{t,x}^{\alpha}\partial_y^{\sigma}\delta u}{\bar{u}_y} Q_5] \\[8pt]\quad

= -\bar{u}_y\partial_y\big(\frac{[\partial_{t,x}^{\alpha}\partial_y^{\sigma},\, \bar{u} \partial_x]\delta u}{\bar{u}_y}\big)
-\bar{u}_y\partial_y\big(\frac{[\partial_{t,x}^{\alpha}\partial_y^{\sigma},\, \bar{u}_x]\delta u}{\bar{u}_y}\big)
-\bar{u}_y\partial_y\big(\frac{[\partial_{t,x}^{\alpha}\partial_y^{\sigma},\, \bar{v}\partial_y]\delta u}{\bar{u}_y}\big) \\[8pt]\quad
-\bar{u}_y\partial_y\big(\frac{[\partial_{t,x}^{\alpha}\partial_y^{\sigma},\, \bar{u}_y]\delta v}{\bar{u}_y}\big),
\quad \sigma\geq 0, \\[14pt]

\frac{1}{\beta - \frac{\bar{u}_{yy}}{\bar{u}_y}}(\partial_t \W_{\alpha} + \bar{u}\partial_x \W_{\alpha})
- \partial_y \W_{\alpha} - \frac{\bar{u}_{yy}}{\bar{u}_y}\W_{\alpha}
+ \frac{1}{\beta - \frac{\bar{u}_{yy}}{\bar{u}_y}}\W_{\alpha}Q_6 \\[10pt]\quad

 = -\sum\limits_{\alpha_1>0}\partial_{t,x}^{\alpha_1}\bar{u}\cdot \frac{\W_{\alpha_2+(0,1)}}{\beta-\frac{\bar{u}_{yy}}{\bar{u}_y} }
 -\sum\limits_{\alpha_1>0}\partial_{t,x}^{\alpha_1}\partial_x\bar{u} \cdot \frac{\W_{\alpha_2}}{\beta-\frac{\bar{u}_{yy}}{\bar{u}_y} },

\quad y=0,\ \sigma=0, \\[15pt]

-\partial_y \W_{\alpha,\sigma} - \frac{\bar{u}_{yy}}{\bar{u}_y}\W_{\alpha,\sigma}

= \mathcal{P}_2\Big(\mathcal{L}+ \sum\limits_{|\alpha^{\prime}|\leq |\alpha|+1}\sum\limits_{m=0}^{\sigma} (W_{\alpha^{\prime},\sigma-m}) (\frac{\bar{u}_{yy}}{\bar{u}_y})^m \\[10pt]\quad
 + \partial_x \partial_{t,x}^{\alpha^{\prime}} (\bar{u}-U)(\frac{\bar{u}_{yy}}{\bar{u}_y})^{\sigma},

\sum\limits_{|\alpha^{\prime}|\leq |\alpha|+1}\sum\limits_{m=0}^{\sigma-1} (\W_{\alpha^{\prime},\sigma-1-m}) (\frac{\bar{u}_{yy}}{\bar{u}_y})^m
 \\[10pt]\quad
 + \sum\limits_{|\alpha^{\prime}|\leq |\alpha|+1}\frac{1}{\beta - \frac{\bar{u}_{yy}}{\bar{u}_y}}\W_{\alpha^{\prime}}(\frac{\bar{u}_{yy}}{\bar{u}_y})^{\sigma}\Big), \quad y=0,\ \sigma\geq 1, \\[14pt]

\W_{\alpha,\sigma}|_{t=0} = (\partial_y u_0^1 + \partial_y u_0^2)\partial_y(\frac{\partial_{t,x}^{\alpha}\partial_y^{\sigma}\delta u}
{\partial_y u_0^1 + \partial_y u_0^2}),
\end{array}\right.
\end{equation}
where the terms $Q_4, Q_5, Q_6$ are defined as
\begin{equation}\label{Sect1_Quantity_Definition_2}
\begin{array}{ll}
Q_4 = -\frac{\bar{u}_{yt} +\bar{u}\bar{u}_{yx} +\bar{v}\bar{u}_{yy} +\bar{u}_{yyy}}{\bar{u}_y} + 2(\frac{\bar{u}_{yy}}{\bar{u}_y})^2, \\[11pt]

Q_5 =  \frac{\bar{u}_{yt} + \bar{u}\bar{u}_{yx} + \bar{v}\bar{u}_{yy} - \bar{u}_{yyy}}{\bar{u}_y}, \\[10pt]

Q_6 = \frac{(\frac{\bar{u}_{yy}}{\bar{u}_y})_t}{\beta - \frac{\bar{u}_{yy}}{\bar{u}_y}}
+ \frac{\bar{u}(\frac{\bar{u}_{yy}}{\bar{u}_y})_x}{\beta - \frac{\bar{u}_{yy}}{\bar{u}_y}}
- \frac{\bar{u}_{yyy}}{\bar{u}_y} + \bar{u}_x,
\end{array}
\end{equation}
and $\mathcal{P}_2$ is a polynomial with lower order degrees, whose explicit form is determined by $(\ref{Appendix_4_Stability1_BC_7})$
and $(\ref{Appendix_4_Stability1_BC_4})$.

Letting $\beta\rto +\infty$, the limit of $(\ref{Sect1_Stability_VorticityEq_1})_2$ is the following equation, which is exactly the boundary condition for the Dirichlet boundary case.
\begin{equation}\label{Sect1_BC_Dirichlet_2}
\begin{array}{ll}
\partial_y \W_{\alpha} + 2\frac{\bar{u}_{yy}}{\bar{u}_y} \W_{\alpha} =0, \quad y=0.
\end{array}
\end{equation}
This boundary condition $(\ref{Sect1_BC_Dirichlet_2})$ appeared in \cite{Alexandre_Yang_2014}.

When $\sigma=0,\ |\alpha|\leq p$, we have the following estimate:
\begin{equation}\label{Sect1_Estimates_4}
\begin{array}{ll}
\frac{\mathrm{d}}{\mathrm{d}t} \|\W_{\alpha}\|_{L_{\ell}^2}^2
+\frac{\mathrm{d}}{\mathrm{d}t} \int\limits_{\mathbb{R}} \frac{1}{\beta - \frac{\bar{u}_{yy}}{\bar{u}_y}|_{y=0}}
 (\W_{\alpha}|_{y=0})^2 \,\mathrm{d}x
+ \|\partial_y \W_{\alpha}\|_{L_{\ell}^2}^2
 \\[12pt]\quad

\lem \big(\mathcal{L}\sum\limits_{|\alpha^{\prime}|\leq |\alpha|}\|W_{\alpha^{\prime},1}\|_{L_{\ell+1}^2}
+ \mathcal{L}\sum\limits_{|\alpha^{\prime}|\leq |\alpha|+1}\|W_{\alpha^{\prime}}\|_{L_{\ell}^2} + \|U\|_{H^{|\alpha|+1}} +\mathcal{L}\big) \\[13pt]\quad
\cdot \Big[\sum\limits_{|\alpha^{\prime}|\leq |\alpha|-1}\|\W_{\alpha^{\prime},1}\|_{L_{\ell+1}^2}^2
+ \sum\limits_{|\alpha^{\prime}|\leq |\alpha|} \big(\|\W_{\alpha^{\prime}}\|_{L_{\ell}^2}^2
+ \int\limits_{\mathbb{R}}
 \frac{1}{\beta - \frac{\bar{u}_{yy}}{\bar{u}_y}|_{y=0}} (\W_{\alpha^{\prime}}|_{y=0})^2 \,\mathrm{d}x \big) \Big].
\end{array}
\end{equation}

When $\sigma>0$, we have the following estimate:
\begin{equation}\label{Sect1_Estimates_5}
\begin{array}{ll}
\frac{\mathrm{d}}{\mathrm{d}t} \|\W_{\alpha,\sigma}\|_{L_{\ell+\sigma}^2}^2
+\frac{\mathrm{d}}{\mathrm{d}t} \int\limits_{\mathbb{R}} \frac{1}{\beta - \frac{\bar{u}_{yy}}{\bar{u}_y}|_{y=0}}
 (\W_{\alpha}|_{y=0})^2 \,\mathrm{d}x
+ \|\partial_y \W_{\alpha,\sigma}\|_{L_{\ell+\sigma}^2}^2  \\[13pt]\quad

\leq CQ_7 \big(\sum\limits_{|\alpha^{\prime}|\leq |\alpha|,\sigma^{\prime}\leq \sigma}\|\W_{\alpha^{\prime},\sigma^{\prime}}\|_{L_{\ell+\sigma^{\prime}}^2}^2
+ \sum\limits_{|\alpha^{\prime}|\leq |\alpha|+1,\sigma^{\prime}\leq \sigma-1}\|\W_{\alpha^{\prime},\sigma^{\prime}}\|_{L_{\ell+\sigma^{\prime}}^2}^2
\\[15pt]\quad
+ \sum\limits_{|\alpha^{\prime}|\leq |\alpha|-1}\|\W_{\alpha^{\prime},\sigma+1}\|_{L_{\ell+\sigma+1}^2}^2
+ \sum\limits_{|\alpha^{\prime}|\leq |\alpha|+1}\int\limits_{\mathbb{R}} \frac{1}{\beta - \frac{\bar{u}_{yy}}{\bar{u}_y}|_{y=0}} (\W_{\alpha}|_{y=0})^2 \,\mathrm{d}x \big) \\[14pt]\quad

+ q\sum\limits_{|\alpha^{\prime}| \leq |\alpha|+1, \sigma^{\prime}\leq \sigma-1}\|\partial_y \W_{\alpha^{\prime},\sigma^{\prime}}\|_{L_{\ell+\sigma}^2}^2,
\ 0< \sigma \leq p-1,\ |\alpha|\leq p-\sigma, \\[18pt]

\frac{\mathrm{d}}{\mathrm{d}t} \|\W_{\alpha,\sigma}\|_{L_{\ell+\sigma}^2}^2
+\frac{\mathrm{d}}{\mathrm{d}t} \int\limits_{\mathbb{R}} \frac{1}{\beta - \frac{\bar{u}_{yy}}{\bar{u}_y}|_{y=0}}
 (\W_{\alpha}|_{y=0})^2 \,\mathrm{d}x
+ \|\partial_y \W_{\alpha,\sigma}\|_{L_{\ell+\sigma}^2}^2 \\[10pt]

\leq C Q_7 \big(\sum\limits_{|\alpha^{\prime}|\leq |\alpha|,\sigma^{\prime}\leq \sigma}\|\W_{\alpha^{\prime},\sigma^{\prime}}\|_{L_{\ell+\sigma^{\prime}}^2}^2
+ \sum\limits_{|\alpha^{\prime}|\leq |\alpha|+1,\sigma^{\prime}\leq \sigma-1}\|\W_{\alpha^{\prime},\sigma^{\prime}}\|_{L_{\ell+\sigma^{\prime}}^2}^2
\\[12pt]\quad
+ \sum\limits_{|\alpha^{\prime}|\leq |\alpha|+1}\int\limits_{\mathbb{R}} \frac{1}{\beta - \frac{\bar{u}_{yy}}{\bar{u}_y}|_{y=0}} (\W_{\alpha}|_{y=0})^2 \,\mathrm{d}x \big)

+ q\sum\limits_{|\alpha^{\prime}| \leq |\alpha|+1, \sigma^{\prime}\leq \sigma-1}\|\partial_y \W_{\alpha^{\prime},\sigma^{\prime}}\|_{L_{\ell+\sigma}^2}^2,

\\[12pt]\quad
\ 0<\sigma\leq p,\ |\alpha|=0,
\end{array}
\end{equation}
where
\begin{equation}\label{Sect1_Estimates_6}
\begin{array}{ll}
Q_7(\alpha,\sigma) := \mathcal{L}\sum\limits_{|\alpha^{\prime}| \leq |\alpha|-1, \sigma^{\prime}\leq \sigma+1}\|W_{\alpha^{\prime},\sigma^{\prime}}\|_{L^2(\mathbb{R}_{+}^2)}

\\[10pt]\quad
+ \mathcal{L}\sum\limits_{|\alpha^{\prime}| \leq |\alpha|+1, \sigma^{\prime}\leq \sigma-1}(\|W_{\alpha^{\prime},\sigma^{\prime}}\|_{L^2(\mathbb{R}_{+}^2)}
+ \|W_{\alpha^{\prime},\sigma^{\prime}}|_{y=0}\|_{L^2(\mathbb{R})})

\\[10pt]\quad
+ \mathcal{L}\sum\limits_{|\alpha^{\prime}| \leq |\alpha|, \sigma^{\prime}\leq \sigma}(\|W_{\alpha^{\prime},\sigma^{\prime}}\|_{L^2(\mathbb{R}_{+}^2)}
+ \|W_{\alpha^{\prime},\sigma^{\prime}}|_{y=0}\|_{L^2(\mathbb{R})})

\\[10pt]\quad
+ \mathcal{L}\sum\limits_{|\alpha^{\prime}|\leq |\alpha|+1}\int\limits_{\mathbb{R}}
 \frac{1}{\beta - \frac{\bar{u}_{yy}}{\bar{u}_y}|_{y=0}} (W_{\alpha^{\prime}}|_{y=0})^2 \,\mathrm{d}x
 +\mathcal{L} <+\infty,
\end{array}
\end{equation}
In $(\ref{Sect1_Stability_VorticityEq_1}), (\ref{Sect1_Estimates_4}), (\ref{Sect1_Estimates_6})$, the lower order term
$\mathcal{L}$ is related to $\bar{u},\bar{v},\bar{\omega}$, and
 $\mathcal{L} \lem (\sum\limits_{|\alpha|+\sigma\leq 6} \|W_{\alpha,\sigma}\|_{L_{\ell+\sigma}^2}^2)^{\frac{s}{2}}$, where $s\geq 1$.

By summing $(\ref{Sect1_Estimates_4}),
(\ref{Sect1_Estimates_5}),
(\ref{Sect1_Estimates_6})$, and apply differential inequalities to the sum, it is easy to close a priori estimates and
prove the stability results. The stability argument implies the uniqueness of the Prandtl systems.
The a priori estimates are uniform with respect to $\beta$, then we also have the stability and uniqueness of $(\ref{Sect1_PrandtlEq_Dirichlet})$.

\begin{remark}\label{Sect1_Estimates_Remark}
(i). When $\sigma\geq 1$,
$\partial_y(\frac{\partial_{t,x}^{\alpha}\partial_y^{\sigma}(u-U)}{u_y})$
and $\partial_y(\frac{\partial_{t,x}^{\alpha}\partial_y^{\sigma}\delta u}{\bar{u}_y})$
decay to zero, but
$\partial_y(\frac{\partial_{t,x}^{\alpha}(u-U)}{u_y})$
and $\partial_y(\frac{\partial_{t,x}^{\alpha}\delta u}{\bar{u}_y})$ may not vanish at $\{y=+\infty\}$, then they
may not be integrable.
so we define and estimate the variables $W_{\alpha} = u_y\partial_y(\frac{\partial_{t,x}^{\alpha}(u-U)}{u_y})$
and $\W_{\alpha} = \bar{u}_y\partial_y(\frac{\partial_{t,x}^{\alpha}\delta u}{\bar{u}_y})$.

(ii). Though $(\ref{Sect1_PrandtlEq})$ and $(\ref{Sect1_PrandtlEq_Dirichlet})$ contain $v$ which loses $\partial_x$-regularity and decay rate, we
use the divergence free condition to
eliminate $\partial_{t,x}^{\alpha}\partial_y^{\sigma} v$ and $\partial_{t,x}^{\alpha}\partial_y^{\sigma} \partial_x u$,
 $\partial_{t,x}^{\alpha}\partial_y^{\sigma}\delta v$ and $\partial_{t,x}^{\alpha}\partial_y^{\sigma}\partial_x \delta u$ in each order estimate
to recover the full regularities of classical Prandtl solutions.
\end{remark}

The rest of the paper is organized as follows: In Section 2, we develop a priori estimates for the existence of Prandtl solutions.
In Section 3, we prove the existence theorems of the Prandtl equations. In Section 4, we develop a priori estimates for the stability
of Prandtl solutions. In Section 5, we prove the stability and uniqueness of the Prandtl equations. In the Appendix, we derive the equations and their boundary conditions.

%%% find 2
\section{A Priori Estimates and Full Regularities of Prandtl Solutions}
In this section, we develop a priori estimates for the existence of the Robin boundary problem $(\ref{Sect1_PrandtlEq})$. These estimates are uniform with respect to $\beta$, then we have a priori estimates for the Dirichlet boundary problem $(\ref{Sect1_PrandtlEq_Dirichlet})$ by passing to the limit.

Assume $[0,T]$ is the lifespan of the Prandtl solutions, thus $T$ is not a blow-up time, the weighted Sobolev norms of the Prandtl solutions remain bounded
at $t=T$.
In order to control the growth of the weighted Sobolev norms in $[0,T]$, we fix a large constant $1\ll M<+\infty$ such that the Sobolev norms of the solutions at $t=T$ are $M$ times the norms of the initial data.

\subsection{Preliminary Estimates}
In this subsection, we prove some preliminary lemmas which are useful for a priori estimates.
In this subsection, $\Omega = [0,T]\times\mathbb{R}_{+}^2$ or $\Omega = \mathbb{R}_{+}^2$.
\begin{lemma}\label{Sect2_Preliminary_Lemma0}
Assume if $0< c_1 (1+y)^{-\theta} \leq \omega_0 \leq c_2 (1+y)^{-\theta}$, where $\theta>\frac{\ell+1}{2}$, we have the properties:
\begin{equation}\label{Sect2_Preliminary_Estimates0}
\begin{array}{ll}
0< \tilde{c}_1(t)\cdot (1+y)^{-\theta} \leq \omega \leq \tilde{c}_2(t)\cdot (1+y)^{-\theta},\quad t\geq 0,\ y\gg 1, \\[7pt]

|u_{yy}| \leq c_3(t)\cdot (1+y)^{-\theta-1}, \quad t\geq 0,\ y\gg 1, \\[7pt]

|(1+y)\frac{u_{yy}}{u_y}|_{\infty} \leq c_4M|(1+y)\frac{\partial_{yy} u_0}{\partial_y u_0}|_{\infty}, \quad t\geq 0,
\end{array}
\end{equation}
where $\tilde{c}_1(t)>0,\tilde{c}_2(t)>0$ are different from $c_1,c_2$.
\end{lemma}

\begin{proof}
The proof of $(\ref{Sect2_Preliminary_Estimates0})$ can follow from Taylor series expansion argument and pointwise interpolation argument (see \cite{Masmoudi_Wong_2012}).

$\omega_t = \omega_{yy} - u\omega_x - v\omega_y$, whose right hand side only contains higher order space decay terms than $\omega$ as long as $\ell>1$.
Thus,
\begin{equation*}
\begin{array}{ll}
\omega = \omega_0 + \int_0^t \text{ [higher order space decay terms] } \,\mathrm{d}t.
\end{array}
\end{equation*}
Namely,
except for the coefficients evolving in time, $\omega$ and $\omega_0$ have the same leading order of space decay, by which
higher order space decay terms are dominated.
That is $0< \tilde{c}_1(t)\cdot (1+y)^{-\theta} \leq \omega \leq \tilde{c}_2(t)\cdot (1+y)^{-\theta},\ y\gg 1$.

By the definition of the limit, for any small $h>0,\ y\gg 1$,
\begin{equation}\label{Sect2_Preliminary_Estimates_0}
\begin{array}{ll}
u_{yy}(t,x,y) \leq \frac{\tilde{c}_2(1+y+h)^{-\theta} - \tilde{c}_1(1+y)^{-\theta}}{h}
\rto -\theta \tilde{c}_2 (1+y)^{-\theta-1}, \quad \text{as}\ h\rto 0, \\[7pt]

u_{yy}(t,x,y) \geq \frac{\tilde{c}_1(1+y+h)^{-\theta} - \tilde{c}_2(1+y)^{-\theta}}{h}
\rto -\theta \tilde{c}_1 (1+y)^{-\theta-1}, \quad \text{as}\ h\rto 0,
\end{array}
\end{equation}
Take $c_3(t) = \max\{\theta \tilde{c}_1(t), \theta \tilde{c}_2(t)\}$.

$|(1+y)\frac{u_{yy}}{u_y}|_{\infty} \leq c_4M|(1+y)\frac{\partial_{yy} u_0}{\partial_y u_0}|_{\infty}$, due to the estimates of $(1+y)^2 u_{yy}$ and $(1+y)u_y$.
Thus, Lemma $\ref{Sect2_Preliminary_Lemma0}$ is proved.
\end{proof}

\begin{remark}\label{Sect2_Preliminary_Remark0}
For the shear flow which is the heat equation, \cite{Xu_Zhang_2015} proved the shear flow has the same lower and upper orders of algebraic decay rates with those of the initial data by using Peetre's inequality.
\end{remark}

\begin{lemma}\label{Sect2_Preliminary_Lemma1}
If $|(1+y)\frac{u_{yy}}{u_y}|_{\infty} \leq c_4M|(1+y)\frac{\partial_{yy} u_0}{\partial_y u_0}|_{\infty}$,
then there exist two positive constants $c_5,c_6$ such that
\begin{equation}\label{Sect2_Preliminary_Estimates1}
\begin{array}{ll}
c_5\|W_{\alpha,\sigma}\|_{L^2_{\ell+\sigma}(\Omega)}\leq \|\partial_{t,x}^{\alpha}\partial_y^{\sigma}\omega\|_{L^2_{\ell+\sigma}(\Omega)}
\leq c_6\|W_{\alpha,\sigma}\|_{L^2_{\ell+\sigma}(\Omega)}, \ \sigma\geq 0,\\[8pt]

\|\partial_{t,x}^{\alpha}\tilde{u}\|_{L^2_{\ell-1}(\Omega)}
\lem \|\partial_{t,x}^{\alpha}\omega\|_{L^2_{\ell}(\Omega)}
\lem \|W_{\alpha,0}\|_{L^2_{\ell}(\Omega)}, \ \sigma= 0, \\[8pt]

\|\partial_{t,x}^{\alpha}\partial_y^{\sigma}u\|_{L^2_{\ell+\sigma-1}(\Omega)}
\lem \|\partial_{t,x}^{\alpha}\partial_y^{\sigma}\omega\|_{L^2_{\ell+\sigma}(\Omega)}
\lem \|W_{\alpha,\sigma}\|_{L^2_{\ell+\sigma}(\Omega)}, \ \sigma\geq 1.
\end{array}
\end{equation}
\end{lemma}

\begin{proof}
Apply Hardy's inequality (see \cite{Masmoudi_Wong_2012}, Lemma B.1) to $W_{\alpha,\sigma} = \partial_{t,x}^{\alpha}\partial_y^{\sigma}\omega - \partial_{t,x}^{\alpha}\partial_y^{\sigma} \tilde{u}\frac{u_{yy}}{u_y}$,
we get the estimate:
\begin{equation}\label{Sect2_Preliminary_Estimates1_1}
\begin{array}{ll}
\|W_{\alpha,\sigma}\|_{L^2_{\ell+\sigma}(\Omega)} = \|\partial_{t,x}^{\alpha}\partial_y^{\sigma}\omega\|_{L^2_{\ell}(\Omega)}
 + \|(1+y)^{\ell+\sigma}\partial_{t,x}^{\alpha}\partial_y^{\sigma} \tilde{u}\frac{u_{yy}}{u_y}\|_{L^2(\Omega)} \\[7pt]

\leq \|\partial_{t,x}^{\alpha}\partial_y^{\sigma}\omega\|_{L^2_{\ell+\sigma}(\Omega)}
 + |(1+y)\frac{u_{yy}}{u_y}|_{\infty}\|(1+y)^{\ell+\sigma-1}\partial_{t,x}^{\alpha}\partial_y^{\sigma} \tilde{u}\|_{L^2(\Omega)} \\[7pt]

\leq \|\partial_{t,x}^{\alpha}\partial_y^{\sigma}\omega\|_{L^2_{\ell+\sigma}(\Omega)}
 + |(1+y)\frac{u_{yy}}{u_y}|_{\infty}\cdot\frac{2}{2\ell+2\sigma-1}\|(1+y)^{\ell+\sigma}\partial_{t,x}^{\alpha}\partial_y^{\sigma} \omega\|_{L^2(\Omega)} \\[7pt]

\leq \frac{1}{c_5} \|\partial_{t,x}^{\alpha}\partial_y^{\sigma}\omega\|_{L^2_{\ell+\sigma}(\Omega)}.
\end{array}
\end{equation}
where $c_5: = (1+\frac{2}{2\ell_0-1}c_4M|(1+y)\frac{\partial_{yy} u_0}{\partial_y u_0}|_{\infty})^{-1}$ is independent of $T$.

\vspace{0.2cm}
Since $\partial_{t,x}^{\alpha}\partial_y^{\sigma}\omega = W_{\alpha,\sigma} + \partial_{t,x}^{\alpha}\partial_y^{\sigma} \tilde{u}\frac{u_{yy}}{u_y}$,
\begin{equation}\label{Sect2_Preliminary_Estimates1_2}
\begin{array}{ll}
\|\partial_{t,x}^{\alpha}\partial_y^{\sigma}\omega\|_{L^2_{\ell+\sigma}(\Omega)}
\leq \|W_{\alpha,\sigma}\|_{L^2_{\ell+\sigma}(\Omega)}
+ \|\partial_{t,x}^{\alpha}\partial_y^{\sigma} \tilde{u}\frac{u_{yy}}{u_y}\|_{L^2_{\ell+\sigma}(\Omega)} \\[7pt]

\leq \|W_{\alpha,\sigma}\|_{L^2_{\ell+\sigma}(\Omega)}
+ |(1+y)\frac{u_{yy}}{u_y}|_{\infty}\|(1+y)^{\ell+\sigma-1}\partial_{t,x}^{\alpha}\partial_y^{\sigma} \tilde{u}\|_{L^2(\Omega)} \\[7pt]

\leq \|W_{\alpha,\sigma}\|_{L^2_{\ell+\sigma}(\Omega)}
+ |(1+y)\frac{u_{yy}}{u_y}|_{\infty}\cdot\frac{2}{2\ell+2\sigma-1}\|(1+y)^{\ell+\sigma}\partial_{t,x}^{\alpha}\partial_y^{\sigma} \omega\|_{L^2(\Omega)}\\[7pt]

\leq \|W_{\alpha,\sigma}\|_{L^2_{\ell+\sigma}(\Omega)}
+ c_4M|(1+y)\frac{\partial_{yy} u_0}{\partial_y u_0}|_{\infty}\cdot\frac{2}{2\ell+2\sigma-1}\|(1+y)^{\ell+\sigma}\partial_{t,x}^{\alpha}\partial_y^{\sigma} \omega\|_{L^2(\Omega)}.
\end{array}
\end{equation}

Let $\ell_0 \geq 2c_4M|(1+y)\frac{\partial_{yy} u_0}{\partial_y u_0}|_{\infty} +\frac{1}{2}$ which is independent of $T$, then it follows from $(\ref{Sect2_Preliminary_Estimates1_2})$ that
\begin{equation}\label{Sect2_Preliminary_Estimates1_3}
\begin{array}{ll}
\|\partial_{t,x}^{\alpha}\partial_y^{\sigma}\omega\|_{L^2_{\ell+\sigma}(\Omega)}
\leq 2\|W_{\alpha,\sigma}\|_{L^2_{\ell+\sigma}(\Omega)}.
\end{array}
\end{equation}
Take $c_6=2$. Thus, Lemma $\ref{Sect2_Preliminary_Lemma1}$ is proved.
\end{proof}

\begin{lemma}\label{Sect2_Preliminary_Lemma2}
Assume $|F|_{\infty} + |(1+y)\partial_y F|_{\infty} \leq c_7\mathcal{L}$, there exists a positive constant $c_8$ such that
\begin{equation}\label{Sect2_Preliminary_Estimates2}
\begin{array}{ll}
\|u_y\partial_y(\frac{\partial_{t,x}^{\alpha}\partial_y^{\sigma}(u-U)}{u_y}F)\|_{L^2_{\ell+\sigma}(\Omega)}
\leq c_8\mathcal{L}\|W_{\alpha,\sigma}\|_{L^2_{\ell+\sigma}(\Omega)}, \ \sigma\geq 0.
\end{array}
\end{equation}
\end{lemma}

\begin{proof}
Since $u_y\partial_y(\frac{\partial_{t,x}^{\alpha}\partial_y^{\sigma}(u-U)}{u_y}F) = F\cdot W_{\alpha,\sigma}
+ \partial_y F\cdot \partial_{t,x}^{\alpha}\partial_y^{\sigma}\tilde{u}$,
\begin{equation}\label{Sect2_Preliminary_Estimates2_1}
\begin{array}{ll}
\|u_y\partial_y(\frac{\partial_{t,x}^{\alpha}(u-U)}{u_y}F)\|_{L^2_{\ell+\sigma}(\Omega)} \\[7pt]

\leq |F|_{\infty}\|W_{\alpha,\sigma}\|_{L^2_{\ell+\sigma}(\Omega)}
+ |(1+y)\partial_y F|_{\infty}\|(1+y)^{\ell+\sigma-1}\partial_{t,x}^{\alpha}\partial_y^{\sigma}\tilde{u}\|_{L^2(\Omega)} \\[7pt]

\leq |F|_{\infty}\|W_{\alpha,\sigma}\|_{L^2_{\ell+\sigma}(\Omega)}
+ |(1+y)\partial_y F|_{\infty}\cdot\frac{2}{2\ell+2\sigma-1}\|(1+y)^{\ell+\sigma}\partial_{t,x}^{\alpha}\partial_y^{\sigma} \omega\|_{L^2(\Omega)} \\[7pt]

\leq |F|_{\infty}\|W_{\alpha,\sigma}\|_{L^2_{\ell+\sigma}(\Omega)}
+ |(1+y)\partial_y F|_{\infty}\cdot\frac{2}{2\ell+2\sigma-1}c_6 \mathcal{L}\|W_{\alpha,\sigma}\|_{L^2_{\ell+\sigma}(\Omega)} \\[7pt]

= c_8\mathcal{L} \|W_{\alpha,\sigma}\|_{L^2_{\ell+\sigma}(\Omega)},
\end{array}
\end{equation}
where $c_8 = c_7\max\{1, \frac{2}{c_6(2\ell_0-1)}\}<+\infty$.
Thus, Lemma $\ref{Sect2_Preliminary_Lemma2}$ is proved.
\end{proof}

\begin{lemma}\label{Sect2_Preliminary_Lemma3}
Assume $|U_x|_{\infty}$, $|F|_{\infty}$ and $\|F\|_{L^2_y(\mathbb{R}_{+})}$ are bounded, then
\begin{equation}\label{Sect2_Preliminary_Estimates3}
\begin{array}{ll}
\|(1+y)^{-1} v\|_{L^{\infty}(\Omega)}
\leq c_9\|\tilde{u}_x\|_{H^2_{\ell}(\Omega)} + |U_x|_{\infty}, \\[8pt]

\|F\cdot(1+y)^{-1} \partial_{t,x}^{\alpha^{\prime}} v\|_{L^2(\Omega)}
\leq 2|F|_{\infty}\|\partial_{t,x}^{\alpha^{\prime}}\partial_x\tilde{u} \|_{L^2(\Omega)} \\[3pt]\hspace{4.3cm}

+ \|\partial_{t,x}^{\alpha^{\prime}}\partial_x U\|_{\Omega\cap{\{y=+\infty\}}}\|F\|_{L^2_y(\mathbb{R}_{+})}.
\end{array}
\end{equation}
\end{lemma}

\begin{proof}
We estimate $(\ref{Sect2_Preliminary_Estimates3})_1$:
\begin{equation}\label{Sect2_Preliminary_Estimates3_1}
\begin{array}{ll}
\|(1+y)^{-1} v\|_{L^{\infty}}
\leq \|(1+y)^{-1} \int_0^y \tilde{u}_x \,\mathrm{d}\tilde{y}\|_{L^{\infty}}
+ \|(1+y)^{-1} \int_0^y U_x \,\mathrm{d}\tilde{y}\|_{L^{\infty}} \\[9pt]\hspace{2.55cm}
\leq \|(1+y)^{-1} \int_0^y \tilde{u}_x \,\mathrm{d}\tilde{y}\|_{L^{\infty}}
+ |U_x|_{\infty}.
\end{array}
\end{equation}
At first, we investigate $L^2$-norm $\|(1+y)^{-1} \int_0^y \tilde{u}_x \,\mathrm{d}\tilde{y}\|_{L^2(\Omega)}$. Apply Hardy's inequality (see \cite{Masmoudi_Wong_2012}, Lemma B.1) to $\|(1+y)^{-1} \int_0^y \tilde{u}_x \,\mathrm{d}\tilde{y}\|_{L^2(\Omega)}$, we get
\begin{equation}\label{Sect2_Preliminary_Estimates3_2}
\begin{array}{ll}
\|(1+y)^{-1} \int_0^y \tilde{u}_x \,\mathrm{d}\tilde{y}\|_{L^2(\Omega)} \\[8pt]
\leq \|(1+y)^{-1} \int_0^y \tilde{u}_x \,\mathrm{d}\tilde{y}\big|_{y=0}\|_{L^2(\partial\Omega)}
+2\|\partial_y\int_0^y \tilde{u}_x \,\mathrm{d}\tilde{y}\|_{L^2(\Omega)}
\leq 2\|\tilde{u}_x\|_{L^2(\Omega)}.
\end{array}
\end{equation}
Continuing using Hardy's inequality,
it is easy to show $\|(1+y)^{-1} v\|_{L^{\infty}(\Omega)}\leq c_9\|\tilde{u}_x\|_{H^2_{\ell}(\Omega)} + |U_x|_{\infty}$ for some $c_9>0$.

Next, we estimate $(\ref{Sect2_Preliminary_Estimates3})_2$:
\begin{equation}\label{Sect2_Preliminary_Estimates3_3}
\begin{array}{ll}
\|F\cdot(1+y)^{-1} \partial_{t,x}^{\alpha^{\prime}} v\|_{L^2(\Omega)}
\\[6pt]
\leq \|F\cdot(1+y)^{-1} \int\limits_0^y \partial_{t,x}^{\alpha^{\prime}}\partial_x\tilde{u} \,\mathrm{d}\tilde{y}\|_{L^2(\Omega)}
+ \|F\cdot(1+y)^{-1} \int\limits_0^y \partial_{t,x}^{\alpha^{\prime}}\partial_x U \,\mathrm{d}\tilde{y}\|_{L^2(\Omega)} \\[9pt]
\leq 2|F|_{\infty}\|\partial_{t,x}^{\alpha^{\prime}}\partial_x\tilde{u} \|_{L^2(\Omega)}
+ \|\partial_{t,x}^{\alpha^{\prime}}\partial_x U\|_{\Omega\cap{\{y=+\infty\}}}\|F\|_{L^2_y(\mathbb{R}_{+})}.
\end{array}
\end{equation}
Thus, Lemma $\ref{Sect2_Preliminary_Lemma3}$ is proved.
\end{proof}

\subsection{Estimates for Tangential Derivatives}
We have the interior estimates and boundary estimates for the tangential derivatives as the following lemma stated:
\begin{lemma}\label{Sect2_1_Existence_Estimate_Lemma}
Assume $\delta_{\beta}\leq\beta<+\infty$, $\ell\geq\ell_0$,
$W_{\alpha}$ satisfies $(\ref{Sect1_Existence_VorticityEq_1})$, then $W_{\alpha}$ satisfies the a priori estimate $(\ref{Sect1_Existence_Estimate1})$.
\end{lemma}

\begin{proof}
Multiple $(\ref{Sect1_Existence_VorticityEq_1})_1$ with $ (1+y)^{2\ell} W_{\alpha}$, integrate in $\mathbb{R}_{+}^2$,
we get
\begin{equation}\label{Sect2_1_Estimates_1}
\begin{array}{ll}
\frac{1}{2}\frac{\mathrm{d}}{\mathrm{d}t} \|W_{\alpha}\|_{L_{\ell}^2}^2
+ \|\partial_y W_{\alpha}\|_{L_{\ell}^2}^2

+ \int\limits_{\mathbb{R}} \partial_y W_{\alpha}|_{y=0} \cdot W_{\alpha}|_{y=0} \,\mathrm{d}x
:= I_1 + I_2,
\end{array}
\end{equation}
where
\begin{equation*}
\begin{array}{ll}
I_1 := \iint\limits_{\mathbb{R}_{+}^2} \big[- u\partial_x W_{\alpha}
- Q_1 W_{\alpha} \big]
\cdot (1+y)^{2\ell} W_{\alpha} \mathrm{d}x\mathrm{d}y \\[8pt]\hspace{0.85cm}
+ \iint\limits_{\mathbb{R}_{+}^2} \big[ - u_y\partial_y(\frac{\partial_{t,x}^{\alpha} \tilde{u}}{u_y}\frac{u_{yt}}{u_y})
- u u_y\partial_y(\frac{\partial_{t,x}^{\alpha} \tilde{u}}{u_y}\frac{u_{yx}}{u_y})
+ u_y\partial_y(\frac{\partial_{t,x}^{\alpha} \tilde{u}}{u_y}\frac{u_{yyy}}{u_y}) \big] \\[12pt]\hspace{0.85cm}
\cdot (1+y)^{2\ell} W_{\alpha} \mathrm{d}x\mathrm{d}y

- 2\ell\iint\limits_{\mathbb{R}_{+}^2} \partial_y W_{\alpha}
\cdot (1+y)^{2\ell-1} W_{\alpha} \mathrm{d}x\mathrm{d}y, \\[14pt]

I_2 := - \iint\limits_{\mathbb{R}_{+}^2} \big[u_y\partial_y\frac{[\partial_{t,x}^{\alpha},\, u_y]v}{u_y}
+ u_y\partial_y\frac{[\partial_{t,x}^{\alpha},\, u\partial_x]\tilde{u}}{u_y}
+ u_y\partial_y\frac{[\partial_{t,x}^{\alpha},\, \tilde{u}\partial_x]U}{u_y} \\[9pt]\hspace{0.85cm}
+ \tilde{u} u_y\partial_y\frac{\partial_x\partial_{t,x}^{\alpha} U}{u_y} \big]
\cdot (1+y)^{2\ell} W_{\alpha} \mathrm{d}x\mathrm{d}y.
\end{array}
\end{equation*}

\noindent
Multiple $(\ref{Sect1_Existence_VorticityEq_1})_2$ with $  W_{\alpha}$, integrate in $\mathbb{R}$,
we get
\begin{equation}\label{Sect2_1_Estimates_2}
\begin{array}{ll}
\frac{1}{2}\frac{\mathrm{d}}{\mathrm{d}t} \int\limits_{\mathbb{R}} \frac{1}{\beta - \frac{u_{yy}}{u_y}|_{y=0}}
 (W_{\alpha}|_{y=0})^2 \,\mathrm{d}x

- \int\limits_{\mathbb{R}} \partial_y W_{\alpha}|_{y=0}
\cdot W_{\alpha}|_{y=0} \,\mathrm{d}x

:= I_3 + I_4,
\end{array}
\end{equation}
where
\begin{equation*}
\begin{array}{ll}
I_3 := \frac{1}{2} \int\limits_{\mathbb{R}} \partial_t(\frac{u_{yy}}{u_y})
 \frac{1}{(\beta - \frac{u_{yy}}{u_y})^2} (W_{\alpha})^2 \,\mathrm{d}x
+ \int\limits_{\mathbb{R}} \frac{u_{yy}}{u_y}
 (W_{\alpha})^2 \,\mathrm{d}x\\[11pt]\hspace{0.88cm}

- \int\limits_{\mathbb{R}}  \big(\frac{1}{\beta - \frac{u_{yy}}{u_y}}
 u\partial_x W_{\alpha} \big) \cdot W_{\alpha} \,\mathrm{d}x
- \int\limits_{\mathbb{R}}  Q_2\cdot\frac{1}{\beta - \frac{u_{yy}}{u_y}} W_{\alpha} \cdot W_{\alpha} \,\mathrm{d}x \\[10pt]\hspace{0.88cm}

+ \int\limits_{\mathbb{R}}  Q_3|_{|\alpha|=0} \cdot W_{\alpha} \,\mathrm{d}x,
\\[12pt]

I_4 := - \int\limits_{\mathbb{R}} [\partial_{t,x}^{\alpha},\,u\partial_x]\tilde{u} \cdot W_{\alpha} \,\mathrm{d}x

- \int\limits_{\mathbb{R}} [\partial_{t,x}^{\alpha},\,\tilde{u}\partial_x]U \cdot W_{\alpha} \,\mathrm{d}x .
\end{array}
\end{equation*}

\vspace{-0.2cm}
By $(\ref{Sect2_1_Estimates_1}) + (\ref{Sect2_1_Estimates_2})$, we get
\begin{equation}\label{Sect2_1_Estimates_3}
\begin{array}{ll}
\frac{\mathrm{d}}{\mathrm{d}t} \|W_{\alpha}\|_{L_{\ell}^2}^2
+\frac{\mathrm{d}}{\mathrm{d}t} \int\limits_{\mathbb{R}} \frac{1}{\beta - \frac{u_{yy}}{u_y}|_{y=0}}
 (W_{\alpha}|_{y=0})^2 \,\mathrm{d}x

+ 2\|\partial_y W_{\alpha}\|_{L_{\ell}^2}^2 \\[8pt]

=2(I_1 + I_2 + I_3 + I_4).
\end{array}
\end{equation}

\vspace{-0.3cm}
By Lemma $\ref{Sect2_Preliminary_Lemma2}$, it is easy to obtain the estimate of $I_1$:
\begin{equation}\label{Sect2_1_Estimates_4}
\begin{array}{ll}
I_1 \leq q\| \partial_y W_{\alpha}\|_{L_{\ell}^2}^2 + C_1 \mathcal{L}\|W_{\alpha}\|_{L_{\ell}^2}^2.
\end{array}
\end{equation}

\vspace{-0.2cm}
Next we estimate $I_2$:

By applying the integration by parts to the first term of $I_2$, we have
\begin{equation*}
\begin{array}{ll}
-\iint\limits_{\mathbb{R}_{+}^2} u_y\partial_y\frac{[\partial_{t,x}^{\alpha},\, u_y]v}{u_y} \cdot (1+y)^{2\ell} W_{\alpha} \mathrm{d}x\mathrm{d}y \\[9pt]

= \iint\limits_{\mathbb{R}_{+}^2} \sum\limits_{\alpha_1>0}[(u_y\partial_y\frac{\partial_{t,x}^{\alpha_1}\tilde{u}
}{u_y})+ \partial_{t,x}^{\alpha_1}\tilde{u}\frac{u_{yy}}{u_y}]\partial_{t,x}^{\alpha_2}v
\cdot \big[\frac{u_{yy}}{u_y}(1+y)^{2\ell} W_{\alpha} \\[10pt]\quad
+ 2\ell(1+y)^{2\ell-1}W_{\alpha} +(1+y)^{2\ell}\partial_y W_{\alpha}\big] \mathrm{d}x\mathrm{d}y
\end{array}
\end{equation*}

\begin{equation}\label{Sect2_1_Estimates_5}
\begin{array}{ll}
= \iint\limits_{\mathbb{R}_{+}^2} \sum\limits_{\alpha_1>0}[(1+y)^{\ell}W_{\alpha_1}
+ (1+y)^{\ell-1}\partial_{t,x}^{\alpha_1}\tilde{u}\cdot (1+y)\frac{u_{yy}}{u_y}]\cdot(1+y)^{-1}\partial_{t,x}^{\alpha_2}v \\[11pt]\quad
\cdot [(1+y)\frac{u_{yy}}{u_y}\cdot(1+y)^{\ell} W_{\alpha}
+ 2\ell(1+y)^{\ell}W_{\alpha}] \mathrm{d}x\mathrm{d}y \\[7pt]\quad

+\iint\limits_{\mathbb{R}_{+}^2} \sum\limits_{\alpha_1>0}
\partial_y\partial_{t,x}^{\alpha_1}\tilde{u}\cdot \partial_{t,x}^{\alpha_2}v
\cdot [(1+y)^{2\ell}\partial_y W_{\alpha}] \mathrm{d}x\mathrm{d}y \\[12pt]

\lem \iint\limits_{\mathbb{R}_{+}^2} \sum\limits_{\alpha_1>0}
\partial_y\partial_{t,x}^{\alpha_1}\tilde{u}\cdot \partial_{t,x}^{\alpha_2}v
\cdot [(1+y)^{2\ell}\partial_y W_{\alpha}] \mathrm{d}x\mathrm{d}y

+ \mathcal{L}\sum\limits_{|\alpha^{\prime}|\leq |\alpha|}\|W_{\alpha^{\prime}}\|_{L^2_{\ell}(\mathbb{R}_{+}^2)}^2,
\end{array}
\end{equation}
thus we need to estimate
$I_{2,1}:=\iint\limits_{\mathbb{R}_{+}^2} \sum\limits_{\alpha_1>0}
\partial_y\partial_{t,x}^{\alpha_1}\tilde{u}\cdot \partial_{t,x}^{\alpha_2}v
\cdot [(1+y)^{2\ell}\partial_y W_{\alpha}] \mathrm{d}x\mathrm{d}y$.

When $\alpha_1=\alpha$, then we estimate by using the integration by parts:
\begin{equation}\label{Sect2_1_Estimates_5_2}
\begin{array}{ll}
I_{2,1} =\iint\limits_{\mathbb{R}_{+}^2} \partial_y\partial_{t,x}^{\alpha}\tilde{u}\cdot v
\cdot [(1+y)^{2\ell}\partial_y W_{\alpha}] \mathrm{d}x\mathrm{d}y \\[9pt]

= -\iint\limits_{\mathbb{R}_{+}^2}
\partial_y\partial_{t,x}^{\alpha}\omega \cdot v
\cdot [(1+y)^{2\ell} (\partial_{t,x}^{\alpha}\omega -\partial_{t,x}^{\alpha}\tilde{u}\frac{u_{yy}}{u_y})
] \mathrm{d}x\mathrm{d}y \\[8pt]\quad
- 2\ell\iint\limits_{\mathbb{R}_{+}^2}
(1+y)^{\ell}\partial_{t,x}^{\alpha}\omega \cdot (1+y)^{-1}v
\cdot [(1+y)^{\ell} W_{\alpha}] \mathrm{d}x\mathrm{d}y \\[8pt]\quad
- \iint\limits_{\mathbb{R}_{+}^2}
(1+y)^{\ell}\partial_{t,x}^{\alpha}\omega \cdot \partial_y v
\cdot [(1+y)^{\ell} W_{\alpha}] \mathrm{d}x\mathrm{d}y \\[8pt]

\leq \iint\limits_{\mathbb{R}_{+}^2}
(1+y)^{2\ell} (\partial_{t,x}^{\alpha}\omega)^2  \cdot [\partial_y v
+ 2\ell (1+y)^{-1}v] \mathrm{d}x\mathrm{d}y \\[4pt]\quad

- \iint\limits_{\mathbb{R}_{+}^2}
(1+y)^{\ell} \partial_{t,x}^{\alpha}\omega \cdot \Big[\partial_y v \big((1+y)^{\ell-1}\partial_{t,x}^{\alpha}\tilde{u}\cdot(1+y)\frac{u_{yy}}{u_y}\big) \\[7pt]\quad
+ 2\ell(1+y)^{-1}v \cdot\big((1+y)^{\ell-1}\partial_{t,x}^{\alpha}\tilde{u}\cdot(1+y)\frac{u_{yy}}{u_y}\big) \\[7pt]\quad
+ (1+y)^{-1}v \cdot \big((1+y)^{\ell}\partial_{t,x}^{\alpha}\omega (1+y)\frac{u_{yy}}{u_y} \\[7pt]\quad
+ (1+y)^{\ell-1}\partial_{t,x}^{\alpha}\tilde{u}
\cdot(1+y)^2(\partial_y\frac{u_{yy}}{u_y}) \big)
\Big] \mathrm{d}x\mathrm{d}y + C\|W_{\alpha}\|_{L^2_{\ell}(\mathbb{R}_{+}^2)}^2 \\[7pt]

\lem \mathcal{L}\|W_{\alpha}\|_{L^2_{\ell}(\mathbb{R}_{+}^2)}^2.
\end{array}
\end{equation}

\vspace{-0.2cm}
When $\alpha_1\neq\alpha$, obviously $|\alpha_1|\leq |\alpha|-1$, then we estimate by using the integration by parts:
\begin{equation}\label{Sect2_1_Estimates_5_3}
\begin{array}{ll}
I_{2,1} =\iint\limits_{\mathbb{R}_{+}^2} \sum\limits_{0<\alpha_1<\alpha}
\partial_y\partial_{t,x}^{\alpha_1}\tilde{u}\cdot \partial_{t,x}^{\alpha_2}v
\cdot [(1+y)^{2\ell}\partial_y W_{\alpha}] \mathrm{d}x\mathrm{d}y \\[10pt]

= -\iint\limits_{\mathbb{R}_{+}^2} \sum\limits_{0<\alpha_1<\alpha}
(1+y)^{\ell+1}\partial_y\partial_{t,x}^{\alpha_1}\omega\cdot (1+y)^{-1}\partial_{t,x}^{\alpha_2}v
\cdot [(1+y)^{\ell}W_{\alpha}] \mathrm{d}x\mathrm{d}y\\[10pt]\quad
- 2\ell\iint\limits_{\mathbb{R}_{+}^2} \sum\limits_{0<\alpha_1<\alpha}
(1+y)^{\ell}\partial_{t,x}^{\alpha_1}\omega\cdot (1+y)^{-1}\partial_{t,x}^{\alpha_2}v
\cdot [(1+y)^{\ell}W_{\alpha}] \mathrm{d}x\mathrm{d}y\\[10pt]\quad
- \iint\limits_{\mathbb{R}_{+}^2} \sum\limits_{0<\alpha_1<\alpha}
(1+y)^{\ell}\partial_{t,x}^{\alpha_1}\omega\cdot \partial_y\partial_{t,x}^{\alpha_2}v
\cdot [(1+y)^{\ell}W_{\alpha}] \mathrm{d}x\mathrm{d}y \\[10pt]

\lem \mathcal{L}\sum\limits_{|\alpha^{\prime}|\leq |\alpha|-1}\|W_{\alpha^{\prime},1}\|_{L^2_{\ell+1}(\mathbb{R}_{+}^2)}^2
+ \mathcal{L}\|\omega_y\|_{L^2_{\ell+1}(\mathbb{R}_{+}^2)}^2
+ \mathcal{L}\sum\limits_{|\alpha^{\prime}|\leq |\alpha|}\|W_{\alpha^{\prime}}\|_{L^2_{\ell}(\mathbb{R}_{+}^2)}^2.
\end{array}
\end{equation}

By $(\ref{Sect2_1_Estimates_5}),(\ref{Sect2_1_Estimates_5_2}),(\ref{Sect2_1_Estimates_5_3})$, we have estimated the first term of $I_2$:
 \begin{equation}\label{Sect2_1_Estimates_5_4}
\begin{array}{ll}
\iint\limits_{\mathbb{R}_{+}^2} u_y\partial_y\frac{[\partial_{t,x}^{\alpha},\, u_y]v}{u_y} \cdot (1+y)^{2\ell} W_{\alpha} \mathrm{d}x\mathrm{d}y \\[8pt]
\lem \mathcal{L}\sum\limits_{|\alpha^{\prime}|\leq |\alpha|-1}\|W_{\alpha^{\prime},1}\|_{L^2_{\ell+1}(\mathbb{R}_{+}^2)}^2
+ \mathcal{L}\sum\limits_{|\alpha^{\prime}|\leq |\alpha|}\|W_{\alpha^{\prime}}\|_{L^2_{\ell}(\mathbb{R}_{+}^2)}^2.
\end{array}
\end{equation}

We estimate the second term of $I_2$:
\begin{equation}\label{Sect2_1_Estimates_6}
\begin{array}{ll}
- \iint\limits_{\mathbb{R}_{+}^2} u_y\partial_y\frac{[\partial_{t,x}^{\alpha},\, u\partial_x]\tilde{u}}{u_y} \cdot (1+y)^{2\ell} W_{\alpha} \mathrm{d}x\mathrm{d}y \\[8pt]

=- \iint\limits_{\mathbb{R}_{+}^2} \sum\limits_{\alpha_1>0}u_y\partial_y\frac{(\partial_{t,x}^{\alpha_1}\tilde{u} + \partial_{t,x}^{\alpha_1}U) \partial_{t,x}^{\alpha_2}\partial_x\tilde{u}
}{u_y} \cdot (1+y)^{2\ell} W_{\alpha} \mathrm{d}x\mathrm{d}y \\[11pt]

\lem \mathcal{L}\sum\limits_{|\alpha^{\prime}|\leq |\alpha|}\|W_{\alpha^{\prime}}\|_{L_{\ell}^2(\mathbb{R}_{+}^2)}^2
+ \mathcal{L}\sum\limits_{|\alpha^{\prime}|\leq |\alpha|}\|W_{\alpha^{\prime}+(-1,1)}\|_{L_{\ell}^2(\mathbb{R}_{+}^2)}^2
+ \sum\limits_{|\alpha_1|\leq |\alpha|}\|\partial_{t,x}^{\alpha_1}U\|_{L^2(\mathbb{R})}^2.
\end{array}
\end{equation}

We estimate the third term of $I_2$:
\begin{equation}\label{Sect2_1_Estimates_7}
\begin{array}{ll}
- \iint\limits_{\mathbb{R}_{+}^2} u_y\partial_y\frac{[\partial_{t,x}^{\alpha},\, \tilde{u}\partial_x]U}{u_y} \cdot (1+y)^{2\ell} W_{\alpha} \mathrm{d}x\mathrm{d}y \\[10pt]

- \iint\limits_{\mathbb{R}_{+}^2} \sum\limits_{\alpha_1>0} u_y\partial_y\frac{\partial_{t,x}^{\alpha_1}\tilde{u}\partial_{t,x}^{\alpha_2}
\partial_x U}{u_y} \cdot (1+y)^{2\ell} W_{\alpha} \mathrm{d}x\mathrm{d}y \\[10pt]

\lem \mathcal{L}\sum\limits_{|\alpha^{\prime}|\leq |\alpha|} \|W_{\alpha^{\prime}}\|_{L_{\ell}^2(\mathbb{R}_{+}^2)}^2
+ \sum\limits_{|\alpha^{\prime}|\leq |\alpha|}\|\partial_{t,x}^{\alpha^{\prime}}U\|_{L^2(\mathbb{R})}^2.
\end{array}
\end{equation}

\vspace{-0.2cm}
Thus,
\begin{equation}\label{Sect2_1_Estimates_8}
\begin{array}{ll}
I_2 \leq
\mathcal{L}\sum\limits_{|\alpha^{\prime}|\leq |\alpha|-1}\|W_{\alpha^{\prime},1}\|_{L^2_{\ell+1}(\mathbb{R}_{+}^2)}^2
+ \mathcal{L}\|\omega_y\|_{L^2_{\ell+1}(\mathbb{R}_{+}^2)}^2 \\[8pt]\qquad
+ \mathcal{L}\sum\limits_{|\alpha^{\prime}|\leq |\alpha|}\|W_{\alpha^{\prime}}\|_{L^2_{\ell}(\mathbb{R}_{+}^2)}^2
+ \|U\|_{H^{|\alpha|+1}(\mathbb{R})}^2.
\end{array}
\end{equation}

Next we estimate $I_3$:
\begin{equation}\label{Sect2_1_Estimates_9}
\begin{array}{ll}
- \int\limits_{\mathbb{R}}  \big(\frac{1}{\beta - \frac{u_{yy}}{u_y}}
 u\partial_x W_{\alpha} \big) \cdot W_{\alpha} \,\mathrm{d}x

\lem \mathcal{L}\int\limits_{\mathbb{R}}
\frac{1}{\beta - \frac{u_{yy}}{u_y}}
 ( W_{\alpha})^2 \,\mathrm{d}x, \\[15pt]

- \int\limits_{\mathbb{R}}  Q_2\cdot\frac{1}{\beta - \frac{u_{yy}}{u_y}} W_{\alpha} \cdot W_{\alpha} \,\mathrm{d}x

\lem \mathcal{L}\int\limits_{\mathbb{R}}
\frac{1}{\beta - \frac{u_{yy}}{u_y}}
 ( W_{\alpha})^2 \,\mathrm{d}x.
\end{array}
\end{equation}

Since $\mathcal{L}\leq M\mathcal{L}|_{t=0}$, let $\ell_0$ is suitably large, we have the following two estimates:
\begin{equation}\label{Sect2_1_Estimates_10}
\begin{array}{ll}
\int\limits_{\mathbb{R}}  Q_3 \cdot W_{\alpha} \,\mathrm{d}x

\leq q \|\partial_y W_{\alpha}\|_{L_{\ell}^2(\mathbb{R}_{+}^2)}^2
 + \|U\|_{H^{|\alpha|+1}(\mathbb{R})}^2, \\[15pt]

\int\limits_{\mathbb{R}} \frac{u_{yy}}{u_y}
 (W_{\alpha})^2 \,\mathrm{d}x
\leq q\| \partial_y W_{\alpha}\|_{L_{\ell}^2(\mathbb{R}_{+}^2)}^2,
\end{array}
\end{equation}
where $q$ is small.

Thus,
\begin{equation}\label{Sect2_1_Estimates_11}
\begin{array}{ll}
I_3 \leq C_2 \mathcal{L}\int\limits_{\mathbb{R}}
\frac{1}{\beta - \frac{u_{yy}}{u_y}}
 ( W_{\alpha})^2 \,\mathrm{d}x + q \|\partial_y W_{\alpha}\|_{L_{\ell}^2(\mathbb{R}_{+}^2)}^2
 + C_2\|U\|_{H^{|\alpha|+1}(\mathbb{R})}^2.
\end{array}
\end{equation}

Next we estimate $I_4$ which has two terms:
\begin{equation}\label{Sect2_1_Estimates_12}
\begin{array}{ll}
- \int\limits_{\mathbb{R}}  [\partial_{t,x}^{\alpha},\,u\partial_x]\tilde{u} \cdot W_{\alpha} \,\mathrm{d}x \\[10pt]

= - \int\limits_{\mathbb{R}}  [\partial_{t,x}^{\alpha},\,\tilde{u}\partial_x]\tilde{u} \cdot W_{\alpha} \,\mathrm{d}x
- \int\limits_{\mathbb{R}}  [\partial_{t,x}^{\alpha},\,U\partial_x]\tilde{u} \cdot W_{\alpha} \,\mathrm{d}x\\[10pt]

\lem \mathcal{L}\sum\limits_{|\alpha^{\prime}|\leq |\alpha|}\int\limits_{\mathbb{R}} \frac{1}{\beta -\frac{u_{yy}}{u_y}} (W_{\alpha^{\prime}})^2
\,\mathrm{d}x
+ \sum\limits_{|\alpha^{\prime}|\leq |\alpha|} \|\partial_{t,x}^{\alpha^{\prime}} U\|_{L^2(\mathbb{R})}^2.
\end{array}
\end{equation}
\begin{equation}\label{Sect2_1_Estimates_13}
\begin{array}{ll}
- \int\limits_{\mathbb{R}}  [\partial_{t,x}^{\alpha},\,\tilde{u}\partial_x]U \cdot W_{\alpha} \,\mathrm{d}x
\\[7pt]
\lem \mathcal{L}\sum\limits_{|\alpha^{\prime}|\leq |\alpha|}\int\limits_{\mathbb{R}} \frac{1}{\beta -\frac{u_{yy}}{u_y}} (W_{\alpha^{\prime}})^2
\,\mathrm{d}x
+ \sum\limits_{|\alpha^{\prime}|\leq |\alpha|} \|\partial_{t,x}^{\alpha^{\prime}} U\|_{L^2(\mathbb{R})}^2.
\end{array}
\end{equation}

\vspace{-0.2cm}
Thus,
\begin{equation}\label{Sect2_1_Estimates_14}
\begin{array}{ll}
I_4 \leq \mathcal{L}\sum\limits_{|\alpha^{\prime}|\leq |\alpha|}\int\limits_{\mathbb{R}} \frac{1}{\beta -\frac{u_{yy}}{u_y}} (W_{\alpha^{\prime}})^2
\,\mathrm{d}x
+ \|U\|_{H^{|\alpha|}(\mathbb{R})}^2.
\end{array}
\end{equation}

Plug $(\ref{Sect2_1_Estimates_4}), (\ref{Sect2_1_Estimates_8})$, $(\ref{Sect2_1_Estimates_11}), (\ref{Sect2_1_Estimates_14})$ into $(\ref{Sect2_1_Estimates_3})$, we have
\begin{equation}\label{Sect2_1_Estimates_15}
\begin{array}{ll}
\frac{\mathrm{d}}{\mathrm{d}t}\|W_{\alpha}\|_{L_{\ell}^2(\mathbb{R}_{+}^2)}^2
+ \frac{\mathrm{d}}{\mathrm{d}t}\int\limits_{\mathbb{R}} \frac{1}{\beta - \frac{u_{yy}}{u_y}}
 (W_{\alpha})^2 \,\mathrm{d}x
+ \|\partial_y W_{\alpha}\|_{L_{\ell}^2(\mathbb{R}_{+}^2)}^2 \\[7pt]

\lem  \mathcal{L}\sum\limits_{|\alpha^{\prime}|\leq |\alpha|}\|W_{\alpha^{\prime}}\|_{L_{\ell}^2(\mathbb{R}_{+}^2)}^2
+ \mathcal{L}\sum\limits_{|\alpha^{\prime}|\leq |\alpha|-1}\|W_{\alpha^{\prime},1}\|_{L_{\ell+1}^2(\mathbb{R}_{+}^2)}^2 \\[10pt]\quad

+ \mathcal{L}\|\omega_y\|_{L^2_{\ell+1}(\mathbb{R}_{+}^2)}^2
+ \mathcal{L}\sum\limits_{|\alpha^{\prime}|\leq |\alpha|}\int\limits_{\mathbb{R}} \frac{1}{\beta -\frac{u_{yy}}{u_y}} (W_{\alpha^{\prime}})^2
\,\mathrm{d}x
+\|U\|_{H^{|\alpha|+1}(\mathbb{R})}^2.
\end{array}
\end{equation}
Thus, Lemma $\ref{Sect2_1_Existence_Estimate_Lemma}$ is proved.
\end{proof}

\subsection{Estimates for Normal Derivatives}
Next, We have the interior estimates and boundary estimates for the normal derivatives as the following lemma stated:
\begin{lemma}\label{Sect2_2_MixDerivatives_Lemma}
Assume $\delta_{\beta}\leq\beta<+\infty$, $\ell\geq\ell_0$, if $W_{\alpha,\sigma}$ satisfies $(\ref{Sect1_Existence_VorticityEq_2})$ when $\sigma\geq 1$,
then $W_{\alpha,\sigma}$ satisfies the a priori estimate $(\ref{Sect1_Existence_Estimates2})$.
\end{lemma}

\begin{proof}
Multiple $(\ref{Sect1_Existence_VorticityEq_2})_1$ with $ (1+y)^{2\ell+2\sigma}  W_{\alpha}$, integrate in $\mathbb{R}_{+}^2$,
we get
\begin{equation}\label{Sect2_2_Estimates_1}
\begin{array}{ll}
\frac{1}{2}\frac{\mathrm{d}}{\mathrm{d}t} \|W_{\alpha,\sigma}\|_{L_{\ell+\sigma}^2}^2
+ \|\partial_y W_{\alpha,\sigma}\|_{L_{\ell+\sigma}^2}^2

+ \int\limits_{\mathbb{R}} \partial_y W_{\alpha,\sigma}|_{y=0} \cdot W_{\alpha,\sigma}|_{y=0} \,\mathrm{d}x
:= I_5 + I_6,
\end{array}
\end{equation}
where
\begin{equation*}
\begin{array}{ll}
I_5 := \iint\limits_{\mathbb{R}_{+}^2} \big[- u\partial_x W_{\alpha,\sigma}
- Q_1 W_{\alpha,\sigma} \big]
\cdot (1+y)^{2\ell+2\sigma} W_{\alpha,\sigma} \mathrm{d}x\mathrm{d}y \\[8pt]\hspace{0.84cm}

- (2\ell+2\sigma)\iint\limits_{\mathbb{R}_{+}^2} \partial_y W_{\alpha}
\cdot (1+y)^{2\ell+2\sigma-1} W_{\alpha} \mathrm{d}x\mathrm{d}y \\[8pt]\hspace{0.84cm}

+ \iint\limits_{\mathbb{R}_{+}^2} \big[ - u_y\partial_y(\frac{\partial_{t,x}^{\alpha} \partial_y^{\sigma}u}{u_y}\frac{u_{yt}}{u_y})
- u u_y\partial_y(\frac{\partial_{t,x}^{\alpha} \partial_y^{\sigma}u}{u_y}\frac{u_{yx}}{u_y})
+ u_y\partial_y(\frac{\partial_{t,x}^{\alpha} \partial_y^{\sigma}u}{u_y}\frac{u_{yyy}}{u_y}) \big] \\[13pt]\hspace{0.84cm}
\cdot (1+y)^{2\ell+2\sigma} W_{\alpha,\sigma} \mathrm{d}x\mathrm{d}y, \\[13pt]

I_6 := - \iint\limits_{\mathbb{R}_{+}^2} \big[ u_y\partial_y\frac{[\partial_{t,x}^{\alpha}\partial_y^{\sigma},\, u_y]v}{u_y}
+ u_y\partial_y\frac{[\partial_{t,x}^{\alpha}\partial_y^{\sigma},\, u\partial_x]u}{u_y}\big]
\cdot (1+y)^{2\ell+2\sigma} W_{\alpha,\sigma} \mathrm{d}x\mathrm{d}y.
\end{array}
\end{equation*}

\noindent
Multiple $(\ref{Sect1_Existence_VorticityEq_2})_2$ with $  W_{\alpha,\sigma}$, integrate in $\mathbb{R}$,
we get
\begin{equation}\label{Sect2_2_Estimates_2}
\begin{array}{ll}
- \int\limits_{\mathbb{R}} \partial_y W_{\alpha,\sigma}|_{y=0}
\cdot W_{\alpha,\sigma}|_{y=0} \,\mathrm{d}x

:= I_7,
\end{array}
\end{equation}
where
\begin{equation*}
\begin{array}{ll}
I_7 := \int\limits_{\mathbb{R}} \frac{u_{yy}}{u_y}
 (W_{\alpha,\sigma})^2 \,\mathrm{d}x

+ \int\limits_{\mathbb{R}}  \mathcal{P}_1\Big(\sum\limits_{m=0}^{\sigma_1-1} (W_{\alpha_1+(1,0),\sigma_1-1-m}) (\frac{u_{yy}}{u_y})^m
\end{array}
\end{equation*}

\begin{equation*}
\begin{array}{ll}
\quad
 + \frac{1}{\beta - \frac{u_{yy}}{u_y}}(W_{\alpha_1+(1,0)} - \beta \partial_x\partial_{t,x}^{\alpha_1} U)(\frac{u_{yy}}{u_y})^{\sigma_1},

 \sum\limits_{m=0}^{\sigma_2-1} (W_{\alpha_2+(0,1),\sigma_2-1-m} (\frac{u_{yy}}{u_y})^m \\[9pt]\quad
 + \frac{1}{\beta - \frac{u_{yy}}{u_y}}(W_{\alpha_2+(0,1)} - \beta \partial_{t,x}^{\alpha_2} U)(\frac{u_{yy}}{u_y})^{\sigma_2} \Big)
 \cdot W_{\alpha,\sigma} \,\mathrm{d}x.
\end{array}
\end{equation*}

By $(\ref{Sect2_2_Estimates_1}) + (\ref{Sect2_2_Estimates_2})$, we get
\begin{equation}\label{Sect2_2_Estimates_3}
\begin{array}{ll}
\frac{\mathrm{d}}{\mathrm{d}t} \|W_{\alpha,\sigma}\|_{L_{\ell+\sigma}^2}^2
+ 2\|\partial_y W_{\alpha,\sigma}\|_{L_{\ell+\sigma}^2}^2

=2(I_5 + I_6 + I_7).
\end{array}
\end{equation}

By Lemma $\ref{Sect2_Preliminary_Lemma2}$, it is easy to obtain the estimate of $I_5$:
\begin{equation}\label{Sect2_2_Estimates_4}
\begin{array}{ll}
I_5 \leq q\| \partial_y W_{\alpha,\sigma}\|_{L_{\ell+\sigma}^2}^2
+ C_3\mathcal{L}\|W_{\alpha,\sigma}\|_{L_{\ell+\sigma}^2}^2.
\end{array}
\end{equation}

Next we estimate $I_6$:

When $\sigma_2 =0$,
$(1+y)^{\sigma_2-1}\partial_{t,x}^{\alpha_2}\partial_y^{\sigma_2}v
=(1+y)^{-1}\partial_{t,x}^{\alpha_2}v$ can be bounded.

\vspace{0.15cm}
When $\sigma_2 \geq 1$, $(1+y)^{\sigma_2-1}\partial_{t,x}^{\alpha_2}\partial_y^{\sigma_2}v$ can be bounded, because
\begin{equation*}
\begin{array}{ll}
(1+y)^{\sigma_2-1}\partial_{t,x}^{\alpha_2}\partial_y^{\sigma_2}v
= (1+y)^{\sigma_2-1}\partial_{t,x}^{\alpha_2}\partial_x\partial_y^{\sigma_2-1} u \\[7pt]

= (1+y)^{-\ell+1}(1+y)^{\ell+\sigma_2-2}\partial_{t,x}^{\alpha_2}\partial_x\partial_y^{\sigma_2-1} u
\leq (1+y)^{\ell+\sigma_2-2}\partial_{t,x}^{\alpha_2}\partial_x\partial_y^{\sigma_2-1} u.
\end{array}
\end{equation*}

When $|\alpha|>0$, by using the integration by parts, we estimate the first term of $I_6$ which is denoted by $I_{6,1}$:
\begin{equation}\label{Sect2_2_Estimates_5}
\begin{array}{ll}
I_{6,1}:= -\iint\limits_{\mathbb{R}_{+}^2} u_y\partial_y\frac{[\partial_{t,x}^{\alpha}\partial_y^{\sigma},\, u_y]v}{u_y} \cdot (1+y)^{2\ell+2\sigma} W_{\alpha,\sigma} \mathrm{d}x\mathrm{d}y \\[9pt]

= \iint\limits_{\mathbb{R}_{+}^2}  \sum\limits_{|\alpha_1|+\sigma_1>0}[(u_y\partial_y\frac{\partial_{t,x}^{\alpha_1}\partial_y^{\sigma_1}\tilde{u}
}{u_y})+ \partial_{t,x}^{\alpha_1}\partial_y^{\sigma_1}\tilde{u}\frac{u_{yy}}{u_y}]\partial_{t,x}^{\alpha_2}\partial_y^{\sigma_2}v
\\[13pt]\quad
\cdot[\frac{u_{yy}}{u_y}(1+y)^{2\ell+2\sigma} W_{\alpha,\sigma}
+ (2\ell+2\sigma)(1+y)^{2\ell+2\sigma-1}W_{\alpha,\sigma} \\[10pt]\quad
+(1+y)^{2\ell+2\sigma}\partial_y W_{\alpha,\sigma}] \mathrm{d}x\mathrm{d}y \\[9pt]

= \iint\limits_{\mathbb{R}_{+}^2} \sum\limits_{|\alpha_1|+\sigma_1>0}[(1+y)^{\ell+\sigma_1}W_{\alpha_1,\sigma_1}
+ (1+y)^{\ell+\sigma_1-1}\partial_{t,x}^{\alpha_1}\partial_y^{\sigma_1}\tilde{u}\cdot (1+y)\frac{u_{yy}}{u_y}]\\[13pt]\quad
\cdot(1+y)^{\sigma_2-1}\partial_{t,x}^{\alpha_2}\partial_y^{\sigma_2}v \cdot
 \big[(1+y)\frac{u_{yy}}{u_y}\cdot(1+y)^{\ell+\sigma} W_{\alpha,\sigma} \\[8pt]\quad
+ (2\ell+2\sigma)(1+y)^{\ell+\sigma}W_{\alpha,\sigma}\big] \mathrm{d}x\mathrm{d}y

+\iint\limits_{\mathbb{R}_{+}^2} \sum\limits_{|\alpha_1|+\sigma_1>0}
\partial_y\partial_{t,x}^{\alpha_1}\partial_y^{\sigma_1}\tilde{u}\cdot \partial_{t,x}^{\alpha_2}\partial_y^{\sigma_2}v \\[8pt]\quad
\cdot [(1+y)^{2\ell+2\sigma}\partial_y W_{\alpha,\sigma}] \mathrm{d}x\mathrm{d}y \\[9pt]

\lem  \mathcal{L}\sum\limits_{|\alpha^{\prime}|\leq |\alpha|,\sigma^{\prime}\leq\sigma}
\|W_{\alpha^{\prime},\sigma^{\prime}}\|_{L^2_{\ell+\sigma^{\prime}}(\mathbb{R}_{+}^2)}^2 \\[9pt]\quad

+ \iint\limits_{\mathbb{R}_{+}^2} \sum\limits_{|\alpha_1|+\sigma_1>0}
\partial_y\partial_{t,x}^{\alpha_1}\partial_y^{\sigma_1}\tilde{u}\cdot \partial_{t,x}^{\alpha_2}\partial_y^{\sigma_2}v
\cdot [(1+y)^{2\ell+2\sigma}\partial_y W_{\alpha,\sigma}] \mathrm{d}x\mathrm{d}y
\end{array}
\end{equation}

We need to estimate the last line of $(\ref{Sect2_2_Estimates_5})$, that is \\
$I_{6,1,1} := \iint\limits_{\mathbb{R}_{+}^2} \sum\limits_{|\alpha_1|+\sigma_1>0}
\partial_y\partial_{t,x}^{\alpha_1}\partial_y^{\sigma_1}\tilde{u}\cdot \partial_{t,x}^{\alpha_2}\partial_y^{\sigma_2}v
\cdot [(1+y)^{2\ell+2\sigma}\partial_y W_{\alpha,\sigma}] \mathrm{d}x\mathrm{d}y$.

When $\alpha_1=\alpha, \sigma_1=\sigma$, we have
\begin{equation}\label{Sect2_2_Estimates_7}
\begin{array}{ll}
I_{6,1,1} =\iint\limits_{\mathbb{R}_{+}^2} \sum\limits_{|\alpha_1|+\sigma_1>0}
\partial_y\partial_{t,x}^{\alpha}\partial_y^{\sigma}\tilde{u}\cdot v
\cdot [(1+y)^{2\ell+2\sigma}\partial_y W_{\alpha,\sigma}] \mathrm{d}x\mathrm{d}y \\[9pt]

= -\iint\limits_{\mathbb{R}_{+}^2}
\partial_y\partial_{t,x}^{\alpha}\partial_y^{\sigma}\omega \cdot v
\cdot [(1+y)^{2\ell+2\sigma} (\partial_{t,x}^{\alpha}\partial_y^{\sigma}\omega
-\partial_{t,x}^{\alpha}\partial_y^{\sigma}\tilde{u}\frac{u_{yy}}{u_y})
] \mathrm{d}x\mathrm{d}y \\[9pt]\quad
- (2\ell+2\sigma)\iint\limits_{\mathbb{R}_{+}^2}
(1+y)^{\ell+\sigma}\partial_{t,x}^{\alpha}\partial_y^{\sigma}\omega \cdot (1+y)^{-1}v
\cdot [(1+y)^{\ell+\sigma} W_{\alpha,\sigma}] \mathrm{d}x\mathrm{d}y \\[9pt]\quad
- \iint\limits_{\mathbb{R}_{+}^2}
(1+y)^{\ell+\sigma}\partial_{t,x}^{\alpha}\partial_y^{\sigma}\omega \cdot \partial_y v
\cdot [(1+y)^{\ell+\sigma} W_{\alpha,\sigma}] \mathrm{d}x\mathrm{d}y \\[9pt]

\leq \iint\limits_{\mathbb{R}_{+}^2}
(1+y)^{2\ell+2\sigma} (\partial_{t,x}^{\alpha}\partial_y^{\sigma}\omega)^2  \cdot [\partial_y v
+ (2\ell+2\sigma) (1+y)^{-1}v] \mathrm{d}x\mathrm{d}y \\[9pt]\quad

+ \iint\limits_{\mathbb{R}_{+}^2}
(1+y)^{\ell+\sigma} \partial_{t,x}^{\alpha}\partial_y^{\sigma}\omega \cdot \Big[\partial_y v \big((1+y)^{\ell-1}\partial_{t,x}^{\alpha}\partial_y^{\sigma}\tilde{u}\cdot(1+y)\frac{u_{yy}}{u_y}\big) \\[9pt]\quad
+ (2\ell+2\sigma)(1+y)^{-1}v \cdot\big((1+y)^{\ell-1}\partial_{t,x}^{\alpha}\partial_y^{\sigma}\tilde{u}\cdot(1+y)\frac{u_{yy}}{u_y}\big) \\[9pt]\quad
+ (1+y)^{-1}v \cdot \big((1+y)^{\ell+\sigma}\partial_{t,x}^{\alpha}\partial_y^{\sigma}\omega (1+y)\frac{u_{yy}}{u_y}
\\[9pt]\quad

+ (1+y)^{\ell-1}\partial_{t,x}^{\alpha}\partial_y^{\sigma}\tilde{u}
\cdot(1+y)^2(\partial_y\frac{u_{yy}}{u_y}) \big)
\Big] \mathrm{d}x\mathrm{d}y + C\|W_{\alpha,\sigma}\|_{L^2_{\ell+\sigma}(\mathbb{R}_{+}^2)}^2 \\[9pt]

\lem \mathcal{L}\sum\limits_{|\alpha^{\prime}|\leq |\alpha|,\sigma^{\prime}\leq \sigma}\|W_{\alpha^{\prime},\sigma^{\prime}}\|_{L^2_{\ell+\sigma^{\prime}}(\mathbb{R}_{+}^2)}^2.
\end{array}
\end{equation}

When $\alpha_1=\alpha,\sigma_1<\sigma$, we have
\begin{equation}\label{Sect2_2_Estimates_8}
\begin{array}{ll}
I_{6,1,1} =\iint\limits_{\mathbb{R}_{+}^2} \sum\limits_{|\alpha_1|+\sigma_1>0}
\partial_y\partial_{t,x}^{\alpha}\partial_y^{\sigma_1}\tilde{u}\cdot \partial_y^{\sigma_2}v
\cdot [(1+y)^{2\ell+2\sigma}\partial_y W_{\alpha,\sigma}] \mathrm{d}x\mathrm{d}y \\[9pt]

= -\iint\limits_{\mathbb{R}_{+}^2} \sum\limits_{|\alpha_1|+\sigma_1>0}
(1+y)^{\ell+\sigma_1+1}\partial_y\partial_{t,x}^{\alpha}\partial_y^{\sigma_1}\omega\cdot (1+y)^{\sigma_2-1}\partial_y^{\sigma_2}v
\\[9pt]\quad
\cdot [(1+y)^{\ell+\sigma}W_{\alpha,\sigma}] \mathrm{d}x\mathrm{d}y\\[9pt]\quad

- \iint\limits_{\mathbb{R}_{+}^2} \sum\limits_{|\alpha_1|+\sigma_1>0}
(1+y)^{\ell+\sigma_1}\partial_{t,x}^{\alpha}\partial_y^{\sigma_1}\omega\cdot (1+y)^{\sigma_2}\partial_y^{\sigma_2+1}v
\\[9pt]\quad
\cdot [(1+y)^{\ell+\sigma}W_{\alpha,\sigma}] \mathrm{d}x\mathrm{d}y\\[9pt]\quad

- (2\ell+2\sigma)\iint\limits_{\mathbb{R}_{+}^2} \sum\limits_{|\alpha_1|+\sigma_1>0}
(1+y)^{\ell+\sigma_1}\partial_{t,x}^{\alpha}\partial_y^{\sigma_1}\omega\cdot (1+y)^{\sigma_2-1}\partial_y^{\sigma_2}v
\\[9pt]\quad
\cdot [(1+y)^{\ell+\sigma}W_{\alpha,\sigma}] \mathrm{d}x\mathrm{d}y \\[8pt]

\lem  \mathcal{L}\sum\limits_{|\alpha^{\prime}|\leq |\alpha|,\sigma^{\prime}\leq \sigma}\|W_{\alpha^{\prime},\sigma^{\prime}}\|_{L^2_{\ell+\sigma^{\prime}}(\mathbb{R}_{+}^2)}^2.
\end{array}
\end{equation}

When $\alpha_1\neq\alpha$, obviously $|\alpha_1|\leq|\alpha|-1$, we have
\begin{equation*}
\begin{array}{ll}
I_{6,1,1} =\iint\limits_{\mathbb{R}_{+}^2} \sum\limits_{|\alpha_1|+\sigma_1>0}
\partial_y\partial_{t,x}^{\alpha_1}\partial_y^{\sigma_1}\tilde{u}\cdot \partial_{t,x}^{\alpha_2}\partial_y^{\sigma_2}v
\cdot [(1+y)^{2\ell+2\sigma}\partial_y W_{\alpha,\sigma}] \mathrm{d}x\mathrm{d}y \\[9pt]

= -\iint\limits_{\mathbb{R}_{+}^2} \sum\limits_{|\alpha_1|+\sigma_1>0}
(1+y)^{\ell+\sigma_1+1}\partial_y\partial_{t,x}^{\alpha_1}\partial_y^{\sigma_1}\omega\cdot (1+y)^{\sigma_2-1}\partial_{t,x}^{\alpha_2}\partial_y^{\sigma_2}v
\\[9pt]\quad
\cdot [(1+y)^{\ell+\sigma}W_{\alpha,\sigma}] \mathrm{d}x\mathrm{d}y\\[9pt]\quad

- \iint\limits_{\mathbb{R}_{+}^2} \sum\limits_{|\alpha_1|+\sigma_1>0}
(1+y)^{\ell+\sigma_1}\partial_{t,x}^{\alpha_1}\partial_y^{\sigma_1}\omega\cdot (1+y)^{\sigma_2}\partial_{t,x}^{\alpha_2}\partial_y^{\sigma_2+1}v
\\[9pt]\quad
\cdot [(1+y)^{\ell+\sigma}W_{\alpha,\sigma}] \mathrm{d}x\mathrm{d}y
\end{array}
\end{equation*}

\begin{equation}\label{Sect2_2_Estimates_9}
\begin{array}{ll}
\quad
- (2\ell+2\sigma)\iint\limits_{\mathbb{R}_{+}^2} \sum\limits_{|\alpha_1|+\sigma_1>0}
(1+y)^{\ell+\sigma_1}\partial_{t,x}^{\alpha_1}\partial_y^{\sigma_1}\omega\cdot (1+y)^{\sigma_2-1}\partial_{t,x}^{\alpha_2}\partial_y^{\sigma_2}v
\\[9pt]\quad
\cdot [(1+y)^{\ell+\sigma}W_{\alpha,\sigma}] \mathrm{d}x\mathrm{d}y \\[9pt]

\lem \mathcal{L}\sum\limits_{|\alpha^{\prime}|\leq |\alpha|-1}\|W_{\alpha^{\prime},\sigma+1}\|_{L^2_{\ell+\sigma+1}(\mathbb{R}_{+}^2)}^2
+ \mathcal{L}\sum\limits_{|\alpha^{\prime}|\leq |\alpha|,\sigma^{\prime}\leq \sigma}\|W_{\alpha^{\prime},\sigma^{\prime}}\|_{L^2_{\ell+\sigma^{\prime}}(\mathbb{R}_{+}^2)}^2.
\end{array}
\end{equation}

Thus, when $|\alpha|>0$, the first term of $I_6$ has the estimate:
\begin{equation}\label{Sect2_2_Estimates_10}
\begin{array}{ll}
I_{6,1} \lem \mathcal{L}\sum\limits_{|\alpha^{\prime}|\leq |\alpha|-1}\|W_{\alpha^{\prime},\sigma+1}\|_{L^2_{\ell+\sigma+1}(\mathbb{R}_{+}^2)}^2
+ \mathcal{L}\sum\limits_{|\alpha^{\prime}|\leq |\alpha|,\sigma^{\prime}\leq \sigma}\|W_{\alpha^{\prime},\sigma^{\prime}}\|_{L^2_{\ell+\sigma^{\prime}}(\mathbb{R}_{+}^2)}^2.
\end{array}
\end{equation}

When $|\alpha|=0$, the estimate $(\ref{Sect2_2_Estimates_9})$ will not appear for this case. Thus,
for $|\alpha|=0$, the first term of $I_6$ has the estimate:
\begin{equation}\label{Sect2_2_Estimates_11}
\begin{array}{ll}
I_{6,1} \lem \mathcal{L}\sum\limits_{|\alpha^{\prime}|\leq |\alpha|,\sigma^{\prime}\leq \sigma}\|W_{\alpha^{\prime},\sigma^{\prime}}\|_{L^2_{\ell+\sigma^{\prime}}(\mathbb{R}_{+}^2)}^2.
\end{array}
\end{equation}

Next we estimate the second term of $I_6$ which is easier:
\begin{equation}\label{Sect2_2_Estimates_12}
\begin{array}{ll}
- \iint\limits_{\mathbb{R}_{+}^2} u_y\partial_y\frac{[\partial_{t,x}^{\alpha}\partial_y^{\sigma},\, u\partial_x]u}{u_y} \cdot (1+y)^{2\ell+2\sigma} W_{\alpha,\sigma} \mathrm{d}x\mathrm{d}y \\[7pt]

=- \iint\limits_{\mathbb{R}_{+}^2} \sum\limits_{|\alpha_1|+\sigma_1>0}(1+y)^{-\ell+1} \big[(1+y)^{\ell+\sigma_1}
u_y\partial_y\frac{\partial_{t,x}^{\alpha_1}\partial_y^{\sigma_1}\tilde{u} }{u_y}
(1+y)^{\ell+\sigma_2-1}\partial_{t,x}^{\alpha_2}\partial_y^{\sigma_2}\partial_x\tilde{u}
\\[13pt]\quad
+ (1+y)^{\ell+\sigma_1-1}\partial_{t,x}^{\alpha_1}\partial_y^{\sigma_1}\tilde{u} (1+y)^{\ell+\sigma_2}\partial_{t,x}^{\alpha_2}\partial_y^{\sigma_2}\partial_x\omega \big]
\cdot (1+y)^{\ell+\sigma} W_{\alpha,\sigma} \mathrm{d}x\mathrm{d}y \\[7pt]\quad

- \iint\limits_{\mathbb{R}_{+}^2} \sum\limits_{|\alpha_1|+\sigma_1>0}(1+y)^{\ell+\sigma}
u_y\partial_y\frac{\partial_{t,x}^{\alpha_1}U \partial_{t,x}^{\alpha_2}\partial_y^{\sigma}\partial_x\tilde{u}
}{u_y} \cdot (1+y)^{\ell+\sigma} W_{\alpha,\sigma} \mathrm{d}x\mathrm{d}y \\[9pt]\quad

- \iint\limits_{\mathbb{R}_{+}^2} \sum\limits_{|\alpha_1|+\sigma_1>0}(1+y)^{\ell+\sigma}
u_y\partial_y\frac{\partial_{t,x}^{\alpha_1}\partial_y^{\sigma}\tilde{u} \partial_{t,x}^{\alpha_2}\partial_x U
}{u_y} \cdot (1+y)^{\ell+\sigma} W_{\alpha,\sigma} \mathrm{d}x\mathrm{d}y \\[9pt]

\lem \mathcal{L}\sum\limits_{|\alpha^{\prime}|\leq |\alpha|,\sigma^{\prime}\leq\sigma}\|W_{\alpha^{\prime},\sigma^{\prime}}\|_{L_{\ell+\sigma^{\prime}}^2(\mathbb{R}_{+}^2)}^2
+ \mathcal{L}\sum\limits_{|\alpha^{\prime}|\leq |\alpha|,\sigma^{\prime}\leq\sigma}\|W_{\alpha^{\prime}+(-1,1),\sigma^{\prime}}\|_{L_{\ell+\sigma^{\prime}}^2(\mathbb{R}_{+}^2)}^2 \\[9pt]\quad
+ \sum\limits_{|\alpha_1|\leq |\alpha|+1}\|\partial_{t,x}^{\alpha_1}U\|_{L^2(\mathbb{R})}^2.
\end{array}
\end{equation}

Thus, when $|\alpha|>0$,
\begin{equation}\label{Sect2_2_Estimates_13}
\begin{array}{ll}
I_6 \leq
\mathcal{L}\sum\limits_{|\alpha^{\prime}|\leq |\alpha|-1}\|W_{\alpha^{\prime},\sigma+1}\|_{L^2_{\ell+\sigma+1}(\mathbb{R}_{+}^2)}^2
+ C_4\|U\|_{H^{|\alpha|+1}(\mathbb{R})}^2\\[12pt]\qquad
+ C_4\mathcal{L}\sum\limits_{|\alpha^{\prime}|\leq |\alpha|,\sigma^{\prime}\leq \sigma}\|W_{\alpha^{\prime},\sigma^{\prime}}\|_{L^2_{\ell+\sigma^{\prime}}(\mathbb{R}_{+}^2)}^2.
\end{array}
\end{equation}
when $|\alpha|=0$,
\begin{equation}\label{Sect2_2_Estimates_14}
\begin{array}{ll}
I_6 \leq
C_5\|U\|_{H^{|\alpha|+1}(\mathbb{R})}^2
+ C_5\mathcal{L}\sum\limits_{|\alpha^{\prime}|\leq |\alpha|,\sigma^{\prime}\leq \sigma}\|W_{\alpha^{\prime},\sigma^{\prime}}\|_{L^2_{\ell+\sigma^{\prime}}(\mathbb{R}_{+}^2)}^2.
\end{array}
\end{equation}

Next we estimate $I_7$, the first term of $I_7$ is
\begin{equation}\label{Sect2_2_Estimates_15}
\begin{array}{ll}
\int\limits_{\mathbb{R}} \frac{u_{yy}}{u_y}
 (W_{\alpha,\sigma})^2 \,\mathrm{d}x
\leq q\| \partial_y W_{\alpha,\sigma}\|_{L_{\ell}^2(\mathbb{R}_{+}^2)}^2,
\end{array}
\end{equation}
where $q$ is small if $\ell_0$ is suitably large.

The second term of $I_7$ is
\begin{equation}\label{Sect2_2_Estimates_16}
\begin{array}{ll}
- \int\limits_{\mathbb{R}}  \mathcal{P}_1\Big(\sum\limits_{m=0}^{\sigma_1-1} (W_{\alpha_1+(1,0),\sigma_1-1-m}) (\frac{u_{yy}}{u_y})^m
 + \frac{1}{\beta - \frac{u_{yy}}{u_y}}(W_{\alpha_1+(1,0)}  \\[7pt]\quad
 - \beta \partial_x\partial_{t,x}^{\alpha_1} U)(\frac{u_{yy}}{u_y})^{\sigma_1},

 \sum\limits_{m=0}^{\sigma_2-1} (W_{\alpha_2+(0,1),\sigma_2-1-m} (\frac{u_{yy}}{u_y})^m
 + \frac{1}{\beta - \frac{u_{yy}}{u_y}}(W_{\alpha_2+(0,1)} \\[9pt]\quad

 - \beta \partial_{t,x}^{\alpha_2} U)(\frac{u_{yy}}{u_y})^{\sigma_2} \Big)\cdot W_{\alpha,\sigma} \,\mathrm{d}x \\[8pt]

\leq q \sum\limits_{|\alpha^{\prime}|\leq |\alpha|+1,\sigma^{\prime}\leq \sigma-1}\|\partial_y W_{\alpha^{\prime},\sigma^{\prime}}\|_{L_{\ell+\sigma^{\prime}}^2(\mathbb{R}_{+}^2)}^2 \\[9pt]\quad

 + C\frac{\beta^2}{(\beta - \frac{u_{yy}}{u_y})^2} \|U\|_{H^{|\alpha|+1}(\mathbb{R})}^2
 + C\int\limits_{\mathbb{R}}\frac{1}{\beta - \frac{u_{yy}}{u_y}}(W_{\alpha+(1,0)}^2 + W_{\alpha+(0,1)}^2) \,\mathrm{d}x.
\end{array}
\end{equation}
where $q$ is small if $\ell_0$ is suitably large.

Thus,
\begin{equation}\label{Sect2_2_Estimates_17}
\begin{array}{ll}
I_7 \leq q\| \partial_y W_{\alpha,\sigma}\|_{L_{\ell}^2(\mathbb{R}_{+}^2)}^2
+ q \sum\limits_{|\alpha^{\prime}|\leq |\alpha|+1,\sigma^{\prime}\leq \sigma-1}\|\partial_y W_{\alpha^{\prime},\sigma^{\prime}}\|_{L_{\ell+\sigma^{\prime}}^2(\mathbb{R}_{+}^2)}^2 \\[10pt]\hspace{0.73cm}

 + C_6\frac{\beta^2}{(\beta - \frac{u_{yy}}{u_y})^2} \|U\|_{H^{|\alpha|+1}(\mathbb{R})}^2
 + C_6\mathcal{L}\int\limits_{\mathbb{R}}\frac{1}{\beta - \frac{u_{yy}}{u_y}}(W_{\alpha+(1,0)}^2 + W_{\alpha+(0,1)}^2) \,\mathrm{d}x.
\end{array}
\end{equation}

Plug $(\ref{Sect2_2_Estimates_4})$, $(\ref{Sect2_2_Estimates_13})$, $(\ref{Sect2_2_Estimates_17})$ into $(\ref{Sect2_2_Estimates_3})$. When $|\alpha|>0$, we have
\begin{equation}\label{Sect2_2_Estimates_18}
\begin{array}{ll}
\frac{\mathrm{d}}{\mathrm{d}t} \|W_{\alpha,\sigma}\|_{L_{\ell+\sigma}^2}^2
+ \|\partial_y W_{\alpha,\sigma}\|_{L_{\ell+\sigma}^2}^2 \\[11pt]

\leq  C\mathcal{L}\sum\limits_{|\alpha^{\prime}|\leq |\alpha|-1}\|W_{\alpha^{\prime},\sigma+1}\|_{L^2_{\ell+\sigma+1}(\mathbb{R}_{+}^2)}^2
+ C\mathcal{L}\sum\limits_{|\alpha^{\prime}|\leq |\alpha|,\sigma^{\prime}\leq \sigma}\|W_{\alpha^{\prime},\sigma^{\prime}}\|_{L^2_{\ell+\sigma^{\prime}}(\mathbb{R}_{+}^2)}^2 \\[12pt]\quad

+ q \sum\limits_{|\alpha^{\prime}|\leq |\alpha|+1,\sigma^{\prime}\leq \sigma-1}\|\partial_y W_{\alpha^{\prime},\sigma^{\prime}}\|_{L_{\ell+\sigma^{\prime}}^2(\mathbb{R}_{+}^2)}^2
+ C\|U\|_{H^{|\alpha|+1}(\mathbb{R})}^2  \\[12pt]\quad

 + C\mathcal{L}\int\limits_{\mathbb{R}}\frac{1}{\beta - \frac{u_{yy}}{u_y}}(W_{\alpha+(1,0)}^2 + W_{\alpha+(0,1)}^2) \,\mathrm{d}x.
\end{array}
\end{equation}

Plug $(\ref{Sect2_2_Estimates_4})$, $(\ref{Sect2_2_Estimates_14})$, $(\ref{Sect2_2_Estimates_17})$ into $(\ref{Sect2_2_Estimates_3})$. When $|\alpha|=0$, we have
\begin{equation}\label{Sect2_2_Estimates_19}
\begin{array}{ll}
\frac{\mathrm{d}}{\mathrm{d}t} \|W_{\alpha,\sigma}\|_{L_{\ell+\sigma}^2(\mathbb{R}_{+}^2)}^2
+ \|\partial_y W_{\alpha,\sigma}\|_{L_{\ell+\sigma}^2(\mathbb{R}_{+}^2)}^2

\leq  C\mathcal{L}\sum\limits_{|\alpha^{\prime}|\leq |\alpha|,\sigma^{\prime}\leq \sigma}\|W_{\alpha^{\prime},\sigma^{\prime}}\|_{L^2_{\ell+\sigma^{\prime}}(\mathbb{R}_{+}^2)}^2  \\[12pt]\quad

+ q \sum\limits_{|\alpha^{\prime}|\leq |\alpha|+1,\sigma^{\prime}\leq \sigma-1}\|\partial_y W_{\alpha^{\prime},\sigma^{\prime}}\|_{L_{\ell+\sigma^{\prime}}^2(\mathbb{R}_{+}^2)}^2

+ C\|U\|_{H^{|\alpha|+1}(\mathbb{R})}^2  \\[12pt]\quad
 + C\mathcal{L}\int\limits_{\mathbb{R}}\frac{1}{\beta - \frac{u_{yy}}{u_y}}(W_{\alpha+(1,0)}^2 + W_{\alpha+(0,1)}^2) \,\mathrm{d}x.
\end{array}
\end{equation}
Thus, Lemma $\ref{Sect2_2_MixDerivatives_Lemma}$ is proved.
\end{proof}

Since $W_{0,1}\equiv 0$, we estimate $\tilde{W} =\omega_y$ alternatively.
\begin{lemma}\label{Sect2_2_TildeW_Lemma}
Assume $\delta_{\beta}\leq\beta<+\infty$, $\ell\geq\ell_0$,
if $\tilde{W}$ satisfies the system $(\ref{Sect1_VorticityY_Eq})$,
then $\tilde{W}$ satisfies the a priori estimate $(\ref{Sect1_VorticityY_Estimate})$.
\end{lemma}

\begin{proof}
Multiple $(\ref{Sect1_VorticityY_Eq})_1$ with $(1+y)^{2\ell+2} \tilde{W}$, integrate in $\mathbb{R}_{+}^2$, we get
\begin{equation}\label{Sect2_TildeW_1}
\begin{array}{ll}
\frac{1}{2}\frac{\mathrm{d}}{\mathrm{d}t}\| \tilde{W}\|_{L_{\ell+1}^2}^2
+ \| \partial_y\tilde{W}\|_{L_{\ell+1}^2}^2

+\int\limits_{\mathbb{R}} \partial_y\tilde{W}|_{y=0} \cdot \tilde{W}|_{y=0} \,\mathrm{d}x := I_{10},
\end{array}
\end{equation}
where
\vspace{-0.2cm}
\begin{equation*}
\begin{array}{ll}
I_{10} := \iint\limits_{\mathbb{R}_{+}^2} [-u\tilde{W}_x -v\tilde{W}_y +u_x\tilde{W} +\omega_x\int\limits_y^{+\infty} \tilde{W}\,\mathrm{d}\tilde{y}]
\cdot (1+y)^{2\ell+2} \tilde{W} \,\mathrm{d}x\mathrm{d}y \\[10pt]\hspace{1cm}

- (2\ell+2)\iint\limits_{\mathbb{R}_{+}^2} \tilde{W}_y \cdot (1+y)^{2\ell+1} \tilde{W} \,\mathrm{d}x\mathrm{d}y.
\end{array}
\end{equation*}

Multiple $(\ref{Sect1_VorticityY_Eq})_2$ with $\tilde{W}$, integrate in $\mathbb{R}$, we get
\begin{equation}\label{Sect2_TildeW_2}
\begin{array}{ll}
- \int\limits_{\mathbb{R}} \partial_y\tilde{W}|_{y=0} \cdot \tilde{W}|_{y=0} \,\mathrm{d}x
= -\int\limits_{\mathbb{R}} (u_{yt} + u u_{yx})|_{y=0} \cdot  u_{yy}|_{y=0} \,\mathrm{d}x := I_{11}.
\end{array}
\end{equation}

By $(\ref{Sect2_TildeW_1}) + (\ref{Sect2_TildeW_2})$, we get
\begin{equation}\label{Sect2_TildeW_3}
\begin{array}{ll}
\frac{1}{2}\frac{\mathrm{d}}{\mathrm{d}t}\| \tilde{W}\|_{L_{\ell+1}^2}^2
+ \| \partial_y\tilde{W}\|_{L_{\ell+1}^2}^2 = I_{10} + I_{11},
\end{array}
\end{equation}

Next we estimate $I_{10}$:
\begin{equation}\label{Sect2_TildeW_4}
\begin{array}{ll}
I_{10} = \iint\limits_{\mathbb{R}_{+}^2} (1+y)^{2\ell+2}\tilde{W}^2[\frac{3}{2}u_x  + \frac{1}{2}
v_y  + \ell v(1+y)^{-1}] \,\mathrm{d}x\mathrm{d}y \\[5pt]\hspace{0.9cm}

 +\iint\limits_{\mathbb{R}_{+}^2} [(1+y)u_{yx}] \cdot\int\limits_y^{+\infty} \tilde{W}\,\mathrm{d}\tilde{y}
\cdot (1+y)^{2\ell+1} \tilde{W} \,\mathrm{d}x\mathrm{d}y \\[10pt]\hspace{0.9cm}

+ (\ell+1)(2\ell+2)\iint\limits_{\mathbb{R}_{+}^2} (1+y)^{2\ell} \tilde{W}^2 \,\mathrm{d}x\mathrm{d}y

\lem \mathcal{L}\|\tilde{W}\|_{L_{\ell+1}^2}^2.
\end{array}
\end{equation}

Next we estimate $I_{11}$:
\begin{equation}\label{Sect2_TildeW_5}
\begin{array}{ll}
I_{11} = - \beta\int\limits_{\mathbb{R}} (u_{t} + u u_{x})|_{y=0} \cdot  (u_t + u u_x + p_x)|_{y=0} \,\mathrm{d}x \\[9pt]\hspace{0.54cm}

= - \beta\int\limits_{\mathbb{R}} (u_{t} + u u_{x})^2 \,\mathrm{d}x
 - \int\limits_{\mathbb{R}}  p_x \cdot(u_{yt} + u u_{yx})|_{y=0} \,\mathrm{d}x \\[9pt]\hspace{0.54cm}

\leq - \int\limits_{\mathbb{R}}  p_x \cdot(\tilde{u}_{yt} + u \tilde{u}_{yx})|_{y=0} \,\mathrm{d}x \\[9pt]\hspace{0.54cm}

= - \int\limits_{\mathbb{R}}  p_x \cdot \big(W_{(1,0)} + u W_{(0,1)}
+ \tilde{u}_{t}\frac{u_{yy}}{u_y}
+ u\tilde{u}_{x} \frac{u_{yy}}{u_y} \big)|_{y=0} \,\mathrm{d}x \\[9pt]\hspace{0.54cm}

= - \int\limits_{\mathbb{R}}  p_x \cdot \frac{\beta}{\beta -\frac{u_{yy}}{u_y}} \big(W_{(1,0)} + u W_{(0,1)}
- \frac{u_{yy}}{u_y} (U_t + u U_x) \big)|_{y=0} \,\mathrm{d}x \\[9pt]\hspace{0.54cm}

\lem \|W_{(1,0)}\|_{L^2(\mathbb{R})}^2 + \|W_{(0,1)}\|_{L^2(\mathbb{R})}^2 + \|U\|_{H^1(\mathbb{R})}^2 + \mathcal{L} \\[9pt]\hspace{0.54cm}

\leq q\|\partial_y W_{(1,0)}\|_{L_{\ell}^2(\mathbb{R}_{+}^2)}^2 + q\|\partial_y W_{(0,1)}\|_{L_{\ell}^2(\mathbb{R}_{+}^2)}^2 + C_7\|U\|_{H^1(\mathbb{R})}^2
+ C_7\mathcal{L}.
\end{array}
\end{equation}

Plug $(\ref{Sect2_TildeW_4})$ and $(\ref{Sect2_TildeW_5})$ into $(\ref{Sect2_TildeW_3})$, we have
\begin{equation}\label{Sect2_TildeW_6}
\begin{array}{ll}
\frac{\mathrm{d}}{\mathrm{d}t}\| \tilde{W}\|_{L_{\ell+1}^2(\mathbb{R}_{+}^2)}^2
+ \| \partial_y\tilde{W}\|_{L_{\ell+1}^2(\mathbb{R}_{+}^2)}^2 \lem \mathcal{L}\| \tilde{W}\|_{L_{\ell+1}^2(\mathbb{R}_{+}^2)}^2 + \|U\|_{H^1(\mathbb{R})}^2
+ \mathcal{L} \\[8pt]\hspace{4cm}

+ q\|\partial_y W_{(1,0)}\|_{L_{\ell}^2(\mathbb{R}_{+}^2)}^2 + q\|\partial_y W_{(0,1)}\|_{L_{\ell}^2(\mathbb{R}_{+}^2)}^2.
\end{array}
\end{equation}

When $\beta =+\infty$,
\begin{equation}\label{Sect2_TildeW_7}
\begin{array}{ll}
I_{11} = - \int\limits_{\mathbb{R}} u_{yt} \cdot  p_x|_{y=0} \,\mathrm{d}x
= - \int\limits_{\mathbb{R}}  p_x \tilde{u}_{yt} |_{y=0} \,\mathrm{d}x \\[9pt]\quad

= - \int\limits_{\mathbb{R}}  p_x \cdot \big(W_{(1,0)}
+ \tilde{u}_{t}\frac{u_{yy}}{u_y}\big)|_{y=0} \,\mathrm{d}x \\[9pt]\quad

= - \int\limits_{\mathbb{R}}  p_x \cdot \frac{\beta}{\beta -\frac{u_{yy}}{u_y}} \big(W_{(1,0)}
- \frac{u_{yy}}{u_y} U_t \big)|_{y=0} \,\mathrm{d}x \\[9pt]\quad

\leq q\|\partial_y W_{(1,0)}\|_{L_{\ell}^2(\mathbb{R}_{+}^2)}^2 + C_8\|U\|_{H^1(\mathbb{R})}^2+ C_8\mathcal{L}.
\end{array}
\end{equation}

Thus, Lemma $\ref{Sect2_2_TildeW_Lemma}$ is proved.
\end{proof}

The following lemma states that the Prandtl systems $(\ref{Sect1_PrandtlEq})$ and $(\ref{Sect1_PrandtlEq_Dirichlet})$ preserve
the full regularities and space decay rates.
\begin{lemma}\label{Sect2_FullRegularity_Lemma}
Assume the conditions in Theorem $\ref{Sect1_Main_Thm}$ are satisfied, $(u,v)$ is the solution of the Prandtl equations $(\ref{Sect1_PrandtlEq})$,
then $(\omega,u,v)$ satisfy the regularities $(\ref{Sect1_Solution_Regularity})$.

When $\beta=+\infty$, assume the conditions in Theorem $\ref{Sect1_Main_Thm_Dirichlet}$, $(u,v)$ is the solution of the Prandtl equations $(\ref{Sect1_PrandtlEq_Dirichlet})$, then $(\omega,u,v)$ satisfy the regularities $(\ref{Sect1_Solution_Regularity_Dirichlet})$.
\end{lemma}

\begin{proof}
If $\sigma_1<\sigma_2$, we estimate $W_{\alpha,\sigma_1}$ before estimating $W_{\alpha,\sigma_2}$.
Also, we estimate $W_{\alpha+(-1,+1),\sigma}$ before estimating $W_{\alpha,\sigma}$.

Coupling the estimates $(\ref{Sect1_Existence_Estimate1})$ with the estimates $(\ref{Sect1_Existence_Estimates2})$, we get
\begin{equation}\label{Sect2_FullRegularity_1}
\begin{array}{ll}
\frac{\mathrm{d}}{\mathrm{d}t}
\Big[\sum\limits_{|\alpha|+\sigma\leq k}\|W_{\alpha,\sigma}\|_{L_{\ell+\sigma}^2(\mathbb{R}_{+}^2)}^2
+ \sum\limits_{|\alpha|\leq k}\int\limits_{\mathbb{R}} \frac{1}{\beta - \frac{u_{yy}}{u_y}|_{y=0}}
 (W_{\alpha}|_{y=0})^2 \,\mathrm{d}x \Big] \\[10pt]\quad
+ \sum\limits_{|\alpha|+\sigma\leq k}\|\partial_y W_{\alpha,\sigma}\|_{L_{\ell+\sigma}^2(\mathbb{R}_{+}^2)}^2
\\[6pt]
\leq \lambda_1 \Big[\sum\limits_{|\alpha|+\sigma\leq k}\|W_{\alpha,\sigma}\|_{L_{\ell+\sigma}^2(\mathbb{R}_{+}^2)}^2
+ \sum\limits_{|\alpha|\leq k}\int\limits_{\mathbb{R}} \frac{1}{\beta - \frac{u_{yy}}{u_y}|_{y=0}} (W_{\alpha}|_{y=0})^2 \,\mathrm{d}x \Big]^{\frac{s}{2}}
\\[9pt]\quad
+ \lambda_1\|U\|_{H^{k+1}(\mathbb{R})}^2,
\end{array}
\end{equation}
where $\lambda_1>0$ is a constant, $s\geq 3$ is an integer.

Then we can prove the following estimate by using the comparison principle of ODE (see \cite{Masmoudi_Wong_2012}), which is easy to prove.
\begin{equation}\label{Sect2_FullRegularity_2}
\begin{array}{ll}
\sum\limits_{|\alpha|+\sigma\leq k}\Big[\|W_{\alpha,\sigma}\|_{L_{\ell+\sigma}^2(\mathbb{R}_{+}^2)}^2
+ \|\partial_y W_{\alpha,\sigma}\|_{L_{\ell+\sigma}^2([0,T]\times\mathbb{R}_{+}^2)}^2 \Big] \\[9pt]\quad

+ \sum\limits_{|\alpha|\leq k}\int\limits_{\mathbb{R}} \frac{1}{\beta - \frac{u_{yy}}{u_y}|_{y=0}}
 (W_{\alpha}|_{y=0})^2 \,\mathrm{d}x
\\[12pt]

\leq \Big[\sum\limits_{|\alpha|+\sigma\leq k}\|W_{\alpha,\sigma}|_{t=0}\|_{L_{\ell+\sigma}^2(\mathbb{R}_{+}^2)}^2
+ \sum\limits_{|\alpha|\leq k}\int\limits_{\mathbb{R}} \frac{1}{\beta - \frac{u_{yy}}{u_y}|_{t=0,y=0}}
 (W_{\alpha}|_{t=0,y=0})^2 \,\mathrm{d}x \\[10pt]\quad
 + \|U\|_{H^{k+1}([0,T]\times\mathbb{R})}^2 \Big]

 \cdot \Big\{ 1 - (s-1)\lambda_1\big[\sum\limits_{|\alpha|+\sigma\leq k}\|W_{\alpha,\sigma}|_{t=0}\|_{L_{\ell+\sigma}^2(\mathbb{R}_{+}^2)}^2 \\[8pt]\quad
+ \sum\limits_{|\alpha|\leq k}\int\limits_{\mathbb{R}} \frac{1}{\beta - \frac{u_{yy}}{u_y}|_{t=0,y=0}}
 (W_{\alpha}|_{t=0,y=0})^2 \,\mathrm{d}x
 + \|U\|_{H^{k+1}([0,T]\times\mathbb{R})}^2 \big]^{\frac{s-1}{2}} t \Big\}^{\frac{-2}{s-1}}, k\geq 6.
\end{array}
\end{equation}

Thus, when
\begin{equation}\label{Sect2_FullRegularity_3}
\begin{array}{ll}
\sum\limits_{|\alpha|+\sigma\leq k} \|W_{\alpha,\sigma}|_{t=0}\|_{L_{\ell+\sigma}^2(\mathbb{R}_{+}^2)}^2
+ \sum\limits_{|\alpha|\leq k}\int\limits_{\mathbb{R}} \frac{1}{\beta - \frac{u_{yy}}{u_y}|_{t=0,y=0}}
 (W_{\alpha}|_{t=0,y=0})^2 \,\mathrm{d}x  \\[9pt]\quad
 + \|U\|_{H^{k+1}([0,T]\times\mathbb{R})}^2 \leq \big(\frac{1-M^{-\frac{2}{s-1}}}{(s-1)\lambda_1 T}\big)^{\frac{2}{s-1}},
\end{array}
\end{equation}
we have
\begin{equation*}
\begin{array}{ll}
\sum\limits_{|\alpha|+\sigma\leq k}\Big[\|W_{\alpha,\sigma}\|_{L_{\ell+\sigma}^2(\mathbb{R}_{+}^2)}^2
+ \|\partial_y W_{\alpha,\sigma}\|_{L_{\ell+\sigma}^2([0,T]\times\mathbb{R}_{+}^2)}^2 \Big]  \\[10pt]\quad
+ \sum\limits_{|\alpha|\leq k}\int\limits_{\mathbb{R}} \frac{1}{\beta - \frac{u_{yy}}{u_y}|_{y=0}}
 (W_{\alpha}|_{y=0})^2 \,\mathrm{d}x
\end{array}
\end{equation*}

\begin{equation}\label{Sect2_FullRegularity_4}
\begin{array}{ll}
\leq M^2 \sum\limits_{|\alpha|+\sigma\leq k}\|W_{\alpha,\sigma}|_{t=0}\|_{L_{\ell+\sigma}^2(\mathbb{R}_{+}^2)}^2 \\[10pt]\quad
+ M^2 \sum\limits_{|\alpha|\leq k}\int\limits_{\mathbb{R}} \frac{1}{\beta - \frac{u_{yy}}{u_y}|_{t=0,y=0}}
 (W_{\alpha}|_{t=0,y=0})^2 \,\mathrm{d}x <+\infty.
\end{array}
\end{equation}

Take $\delta_{\beta} = \max\limits_{t\in [0,T]}\{\big|\frac{u_{yy}}{u_y}|_{y=0}\big|_{\infty}\} +1 <+\infty$.

Let $\beta\rto +\infty$ in the a priori estimate $(\ref{Sect2_FullRegularity_1})$, we get
\begin{equation}\label{Sect2_FullRegularity_1_Dirichlet}
\begin{array}{ll}
\frac{\mathrm{d}}{\mathrm{d}t}
\sum\limits_{|\alpha|+\sigma\leq k}\|W_{\alpha,\sigma}\|_{L_{\ell+\sigma}^2(\mathbb{R}_{+}^2)}^2
+ \sum\limits_{|\alpha|+\sigma\leq k}\|\partial_y W_{\alpha,\sigma}\|_{L_{\ell+\sigma}^2(\mathbb{R}_{+}^2)}^2
\\[12pt]

\leq \lambda_2\sum\limits_{|\alpha|+\sigma\leq k}\|W_{\alpha,\sigma}\|_{L_{\ell+\sigma}^2(\mathbb{R}_{+}^2)}^{\frac{s}{2}}
+ \lambda_2\|U\|_{H^{k+1}(\mathbb{R})}^2,
\end{array}
\end{equation}
where $\lambda_2>0$ is a constant, $s\geq 3$ is an integer.

Then we can prove the following estimate by using the comparison principle of ODE:
\begin{equation}\label{Sect2_FullRegularity_2_Dirichlet}
\begin{array}{ll}
\sum\limits_{|\alpha|+\sigma\leq k}\|W_{\alpha,\sigma}\|_{L_{\ell+\sigma}^2(\mathbb{R}_{+}^2)}^2
+ \sum\limits_{|\alpha|+\sigma\leq k}\|\partial_y W_{\alpha,\sigma}\|_{L_{\ell+\sigma}^2([0,T]\times\mathbb{R}_{+}^2)}^2
\\[12pt]

\leq \Big[\sum\limits_{|\alpha|+\sigma\leq k}\|W_{\alpha,\sigma}|_{t=0}\|_{L_{\ell+\sigma}^2(\mathbb{R}_{+}^2)}^2
 + \|U\|_{H^{k+1}([0,T]\times\mathbb{R})}^2 \Big]

 \cdot \Big\{ 1 - (s-1)\lambda_2 \\[10pt]\quad

 \cdot\big[\sum\limits_{|\alpha|+\sigma\leq k}\|W_{\alpha,\sigma}|_{t=0}\|_{L_{\ell+\sigma}^2(\mathbb{R}_{+}^2)}^2
 + \|U\|_{H^{k+1}([0,T]\times\mathbb{R})}^2 \big]^{\frac{s-1}{2}} t \Big\}^{-\frac{2}{s-1}}, \quad k\geq 6.
\end{array}
\end{equation}

Thus, when
\begin{equation}\label{Sect2_FullRegularity_3_Dirichlet}
\begin{array}{ll}
\sum\limits_{|\alpha|+\sigma\leq k} \|W_{\alpha,\sigma}|_{t=0}\|_{L_{\ell+\sigma}^2(\mathbb{R}_{+}^2)}^2
 + \|U\|_{H^{k+1}([0,T]\times\mathbb{R})}^2 \leq \big(\frac{1-M^{-\frac{2}{s-1}}}{(s-1)\lambda_2 T}\big)^{\frac{2}{s-1}},
\end{array}
\end{equation}
we have
\begin{equation}\label{Sect2_FullRegularity_4_Dirichlet}
\begin{array}{ll}
\sum\limits_{|\alpha|+\sigma\leq k}\Big[\|W_{\alpha,\sigma}\|_{L_{\ell+\sigma}^2(\mathbb{R}_{+}^2)}^2
+ \|\partial_y W_{\alpha,\sigma}\|_{L_{\ell+\sigma}^2([0,T]\times\mathbb{R}_{+}^2)}^2 \Big] \\[10pt]

\leq M^2 \sum\limits_{|\alpha|+\sigma\leq k}\Big[\|W_{\alpha,\sigma}|_{t=0}\|_{L_{\ell+\sigma}^2(\mathbb{R}_{+}^2)}^2
<+\infty.
\end{array}
\end{equation}

Thus, Lemma $\ref{Sect2_FullRegularity_Lemma}$ is proved.
\end{proof}

%%% find 3
\section{The Existence of the Prandtl Equations}
In this section, we construct Prandtl solutions and study their behaviors on the boundary.

\subsection{Iteration Scheme and Convergence of Approximate Solutions}
In this subsection, we construct the Prandtl solutions by using the iteration scheme $(\ref{Sect1_PrandtlEq_Iteration})$.

For the iteration scheme $(\ref{Sect1_PrandtlEq_Iteration})$, we have the equations of $W_{\alpha}^{n+1}$ as follows:
\begin{equation}\label{Sect3_Existence_VorticityEq_1}
\left\{\begin{array}{ll}
\partial_t W_{\alpha}^{n+1} + u^n\partial_x W_{\alpha}^{n+1} - \partial_{yy} W_{\alpha}^{n+1} + Q_1^n W_{\alpha}^{n+1}
+ \partial_y u^n\partial_y(\frac{\partial_{t,x}^{\alpha} \tilde{u}^{n+1}}{\partial_y u^n}\frac{\partial_{yt} u^n}{\partial_y u^n}) \\[8pt]\quad
+ u^n\partial_y u^n\partial_y(\frac{\partial_{t,x}^{\alpha} \tilde{u}^{n+1}}{\partial_y u^n}\frac{\partial_{yx} u^n}{\partial_y u^n})
- \partial_y u^n\partial_y(\frac{\partial_{t,x}^{\alpha} \tilde{u}^{n+1}}{\partial_y u^n}\frac{\partial_{yyy} u^n}{\partial_y u^n})\\[8pt]\quad

= - \partial_y u^n\partial_y\frac{[\partial_{t,x}^{\alpha},\, \partial_y u^n]v^{n+1}}{\partial_y u^n}
- \partial_y u^n\partial_y\frac{[\partial_{t,x}^{\alpha},\, u^n\partial_x]\tilde{u}^{n+1}}{\partial_y u^n} \\[8pt]\quad
- \partial_y u^n\partial_y\frac{[\partial_{t,x}^{\alpha},\, \tilde{u}^n\partial_x]U}{\partial_y u^n}
-\tilde{u}^n \partial_y u^n\partial_y\frac{\partial_x\partial_{t,x}^{\alpha} U}{\partial_y u^n}, \\[13pt]

\frac{1}{\beta - \frac{\partial_{yy}u^n}{\partial_y u^n}}(\partial_t W_{\alpha}^{n+1} + u^n\partial_x W_{\alpha}^{n+1})
- \partial_y W_{\alpha}^{n+1} - \frac{\partial_{yy}u^n}{\partial_y u^n}W_{\alpha}^{n+1} \\[11pt]\quad
+ Q_2 \cdot\frac{1}{\beta - \frac{\partial_{yy}u^n}{\partial_y u^n}}W_{\alpha}^{n+1}

= -[\partial_{t,x}^{\alpha},\,u^n\partial_x]\tilde{u}^{n+1} -[\partial_{t,x}^{\alpha},\,\tilde{u}^n\partial_x]U
+ Q_3, \ \, y=0, \\[13pt]

W_{\alpha}^{n+1}|_{t=0} = \partial_y u_0(x,y)
\partial_y\big(\frac{\partial_{t,x}^{\alpha}u_0(x,y)-\partial_{t,x}^{\alpha} U_0(x)}{\partial_y u_0(x,y)}\big) =W_{\alpha}^0,
\end{array}\right.
\end{equation}
where the lower order terms $Q_1, Q_2, Q_3$ are defined as
\begin{equation}\label{Sect3_Quantity_Definition_1}
\begin{array}{ll}
Q_1^n := - \frac{\partial_{yt} u^n}{\partial_y u^n} - \frac{u^n \partial_{yx}u^n}{\partial_y u^n} - \frac{\partial_{yyy}u^n}{\partial_y u^n} +2(\frac{\partial_{yy}u^n}{\partial_y u^n})^2, \\[10pt]

Q_2^n := \frac{(\frac{\partial_{yy}u^n}{\partial_y u^n})_t}{\beta - \frac{\partial_{yy} u^n}{\partial_y u^n}}
+ \frac{u^n(\frac{\partial_{yy}u^n}{\partial_y u^n})_x}{\beta - \frac{\partial_{yy} u^n}{\partial_y u^n}}
 - \frac{\partial_{yyy}u^n}{\partial_y u^n}, \quad y=0,
\\[14pt]

Q_3^n :=-\tilde{u}^n\partial_x\partial_{t,x}^{\alpha}U
+ \partial_t\big[\frac{\beta}{\beta - \frac{\partial_{yy}u^n}{\partial_y u^n}} \partial_{t,x}^{\alpha} U\big]
+ u^n\partial_x\big[\frac{\beta}{\beta - \frac{\partial_{yy} u^n}{\partial_y u^n}} \partial_{t,x}^{\alpha} U\big] \\[9pt]\hspace{1.05cm}
- \frac{\partial_{yyy}u^n}{\partial_y u^n}\cdot
\frac{\beta}{\beta - \frac{\partial_{yy} u^n}{\partial_y u^n}} \partial_{t,x}^{\alpha} U.
\end{array}
\end{equation}

For $(\ref{Sect1_PrandtlEq_Iteration})$, when $\sigma\geq 1$, we have the following equations for $W_{\alpha,\sigma}$:
\begin{equation}\label{Sect3_Existence_VorticityEq_2}
\left\{\begin{array}{ll}
\partial_t W_{\alpha,\sigma}^{n+1} + u^n\partial_x W_{\alpha,\sigma}^{n+1}
- \partial_{yy} W_{\alpha,\sigma}^{n+1} +Q_1^n W_{\alpha,\sigma}^{n+1}

+ \partial_y u^n\partial_y[\frac{\partial_{t,x}^{\alpha}\partial_y^{\sigma} u^{n+1}}{\partial_y u^n}\frac{\partial_{yt}u^n}{\partial_y u^n}] \\[6pt]\quad
+ u^n \partial_y u^n\partial_y[\frac{\partial_{t,x}^{\alpha}\partial_y^{\sigma} u^{n+1}}{\partial_y u^n}\frac{\partial_{yx}u^n}{\partial_y u^n}]
- \partial_y u^n\partial_y\big[\frac{\partial_{t,x}^{\alpha}\partial_y^{\sigma} u^{n+1}}{\partial_y u^n} \frac{\partial_{yyy}u^n}{\partial_y u^n} \big] \\[9pt]\quad

=  - \partial_y u^n\partial_y\frac{[\partial_{t,x}^{\alpha}\partial_y^{\sigma},\, u^n\partial_x]u^{n+1}}{\partial_y u^n}
- \partial_y u^n\partial_y\frac{[\partial_{t,x}^{\alpha}\partial_y^{\sigma},\, \partial_y u^n]v^{n+1}}{\partial_y u^n}, \\[8pt]

-\partial_y W_{\alpha,\sigma}^{n+1} - \frac{\partial_{yy}u^n}{\partial_y u^n}W_{\alpha,\sigma}^{n+1}

= \mathcal{P}_1\Big(\sum\limits_{m=0}^{\sigma_1-1} (W_{\alpha_1+(1,0),\sigma_1-1-m}^{n+1}) (\frac{\partial_{yy}u^n}{\partial_y u^n})^m \\[7pt]\quad
 + \frac{1}{\beta - \frac{\partial_{yy}u^n}{\partial_y u^n}}(W_{\alpha_1+(1,0)}^{n+1} - \beta \partial_x\partial_{t,x}^{\alpha_1} U)
 (\frac{\partial_{yy}u^n}{\partial_y u^n})^{\sigma_1}, \\[7pt]\quad

 \sum\limits_{m=0}^{\sigma_2-1} (W_{\alpha_2+(0,1),\sigma_2-1-m}^{n+1} (\frac{\partial_{yy}u^n}{\partial_y u^n})^m
 + \frac{1}{\beta - \frac{\partial_{yy}u^n}{\partial_y u^n}}(W_{\alpha_2+(0,1)}^{n+1} \\[7pt]\quad
 - \beta \partial_{t,x}^{\alpha_2} U)(\frac{\partial_{yy}u^n}{\partial_y u^n})^{\sigma_2} \Big),
 \quad y=0, \\[9pt]

W_{\alpha,\sigma}^{n+1}|_{t=0} = \partial_y u_0(x,y) \partial_y(\frac{\partial_{t,x}^{\alpha}\partial_y^{\sigma} u_0(x,y)}{\partial_y u_0(x,y)})
=W_{\alpha,\sigma}^0.
\end{array}\right.
\end{equation}
where $\alpha_1\leq\alpha, \alpha_2\leq\alpha, \sigma_1\leq\sigma,\sigma_2\leq\sigma$,
$\mathcal{P}_1$ is a polynomial whose degree is less than or equal to $2$, the explicit form of $\mathcal{P}_1$
is given by the right hand side of $(\ref{Appendix_2_Sigma1_Existence_BC_5})$.

The local existence of $(\ref{Sect1_PrandtlEq})$ is easy to prove by by applying Gr$\ddot{o}$nwall's inequality to a priori estimates, we give the following lemma without proof.
\begin{lemma}\label{Sect3_Existence_Lemma_Robin}
Considering the Robin boundary problem $(\ref{Sect1_PrandtlEq})$
under Oleinik's monotonicity assumption, for any integer $k\geq 6$, $U(t,x)\in C^{k+1}([0,+\infty)\times\mathbb{R})$, $U(t,x)>0$,
$\|U(t,x)\|_{H^{k+1}(\mathbb{R})}<+\infty$, $\|\omega_0\|_{\mathcal{H}_{\ell}^k(\mathbb{R}_{+}^2)}
+ \frac{1}{\sqrt{\beta}}\big\|\omega_0|_{y=0}\big\|_{\mathcal{H}_{\ell}^k(\mathbb{R})}<+\infty$,
there exist suitably large real numbers $\ell_0>1$, $\delta_{\beta}>0$ such that if $\ell\geq\ell_0$, $\beta\in [\delta_{\beta},+\infty)$, the compatible initial data satisfies
\begin{equation}\label{Sect3_Data_Conditions1_LocalExistence}
\left\{\begin{array}{ll}
\omega_0>0, \quad u_0|_{y=0}>0, \quad
(\partial_y u_0 - \beta u_0)|_{y=0} =0, \quad
\lim\limits_{y\rto +\infty} u_0 =U|_{t=0}, \\[6pt]

0< c_1 (1+y)^{-\theta} \leq \omega_0 \leq c_2 (1+y)^{-\theta}, \ \theta>\frac{\ell+1}{2},
\end{array}\right.
\end{equation}
 then there exist $T>0$ and a classical solution $(\omega,u,v)$ to the Prandtl system $(\ref{Sect1_PrandtlEq})$ in $[0,T]$ satisfying the regularities $(\ref{Sect1_Solution_Regularity})$.
\end{lemma}

The large time existence theorem of $(\ref{Sect1_PrandtlEq})$ is stated as follows:
\begin{theorem}\label{Sect3_Existence_Thm_Robin}
Assume the conditions are the same with the conditions in Theorem $\ref{Sect1_Main_Thm}$. For any fixed finite number $T\in (0,+\infty)$, there exist sufficiently small real number $0<\ve_1 = o(T^{-1})$ and
suitably large real numbers $\ell_0>1$, $\delta_{\beta}>0$ such that if $\ell\geq\ell_0$, $\beta\in [\delta_{\beta},+\infty)$, the compatible initial data satisfies the conditions $(\ref{Sect1_Data_Conditions1})$, then the Robin boundary problem $(\ref{Sect1_PrandtlEq})$ admits a classical solution $(\omega,u,v)$
in $[0,T]$ satisfying the regularities $(\ref{Sect1_Solution_Regularity})$.
\end{theorem}

\begin{proof}
Similar to the estimate $(\ref{Sect2_FullRegularity_1})$, we have the following estimate for $W_{\alpha,\sigma}^{n+1}$:
\begin{equation}\label{Sect3_ApproximateSol_Regularity_1}
\begin{array}{ll}
\frac{\mathrm{d}}{\mathrm{d}t}
\Big[\sum\limits_{|\alpha|+\sigma\leq k}\|W_{\alpha,\sigma}^{n+1}\|_{L_{\ell+\sigma}^2(\mathbb{R}_{+}^2)}^2
+ \sum\limits_{|\alpha|\leq k}\int\limits_{\mathbb{R}} \frac{1}{\beta - \frac{\partial_{yy}u^n}{\partial_y u^n}|_{y=0}}
 (W_{\alpha}^{n+1}|_{y=0})^2 \,\mathrm{d}x \Big] \\[10pt]\quad
+ \sum\limits_{|\alpha|+\sigma\leq k}\|\partial_y W_{\alpha,\sigma}^{n+1}\|_{L_{\ell+\sigma}^2(\mathbb{R}_{+}^2)}^2 \\[10pt]

\leq \lambda_1\mathcal{L}_n\sum\limits_{|\alpha|+\sigma\leq k}\|W_{\alpha,\sigma}^{n+1}\|_{L_{\ell+\sigma}^2(\mathbb{R}_{+}^2)}^2

+ \lambda_1\mathcal{L}_n\sum\limits_{|\alpha|\leq k}\int\limits_{\mathbb{R}} \frac{1}{\beta - \frac{\partial_{yy}u^n}{\partial_y u^n}|_{y=0}}
 (W_{\alpha}^{n+1}|_{y=0})^2 \,\mathrm{d}x \\[10pt]\quad

+ \lambda_1\|U\|_{H^{k+1}(\mathbb{R})}^2 + \lambda_1 \mathcal{L}_n.
\end{array}
\end{equation}

By using the induction method and integrating the growth rate of $\mathcal{L}_n$, it is easy to prove the grow rate of approximate solutions which is the same with the grow rate of Prandtl solutions.
Similar to the proof of Lemma $\ref{Sect2_FullRegularity_Lemma}$ and
by using the comparison principle of ODE, we can prove that there exist a class of sufficiently small data such that
\begin{equation}\label{Sect3_ApproximateSol_Regularity_2}
\begin{array}{ll}
\sum\limits_{|\alpha|+\sigma\leq k}\Big[\|W_{\alpha,\sigma}^{n+1}\|_{L_{\ell+\sigma}^2(\mathbb{R}_{+}^2)}^2
+ \|\partial_y W_{\alpha,\sigma}^{n+1}\|_{L_{\ell+\sigma}^2([0,T]\times\mathbb{R}_{+}^2)}^2 \Big] \\[9pt]\quad
 + \sum\limits_{|\alpha|\leq k}\int\limits_{\mathbb{R}} \frac{1}{\beta+\delta_{\beta}}
 (W_{\alpha}^{n+1}|_{y=0})^2 \,\mathrm{d}x

<+\infty,
\end{array}
\end{equation}
then we have the following regularities:
\begin{equation}\label{Sect3_ApproximateSol_Regularity_3}
\begin{array}{ll}
\omega^{n+1} \in \mathcal{H}_{\ell}^k(\mathbb{R}_{+}^2), \hspace{1.3cm}
\omega^{n+1},\ \partial_y\omega^{n+1} \in \mathcal{H}_{\ell}^k([0,T]\times\mathbb{R}_{+}^2),
\\[5pt]

u^{n+1}-U \in \mathcal{H}_{\ell-1}^k(\mathbb{R}_{+}^2)\cap \mathcal{H}_{\ell-1}^k([0,T]\times\mathbb{R}_{+}^2),
\\[5pt]

\partial_y^{j} u^{n+1}|_{y=0} \in H^{k-j}(\mathbb{R})\cap H^{k-j}([0,T]\times\mathbb{R}), \hspace{0.3cm} 0\leq j\leq k,\\[5pt]

\partial_{t,x}^{\alpha} v^{n+1} + y\cdot \partial_{t,x}^{\alpha}\partial_x U
\in L_{y, \ell-1}^{\infty}(L_{t,x}^2), \hspace{0.3cm} |\alpha|\leq k-1.
\end{array}
\end{equation}

The regularities $(\ref{Sect3_ApproximateSol_Regularity_3})$ is uniform with respect to $n$, thus there exist functions $\omega,u,v$
lie in our solution spaces, then $\omega,u,v$
 satisfy the regularities $(\ref{Sect1_Solution_Regularity})$.
 It is also standard to verify $\omega,u,v$ satisfy the Prandtl system $(\ref{Sect1_PrandtlEq})$.
\end{proof}

The local existence of $(\ref{Sect1_PrandtlEq_Dirichlet})$ is easy to prove by by applying Gr$\ddot{o}$nwall's inequality to a priori estimates, we give the following lemma without proof.
\begin{lemma}\label{Sect3_Existence_Lemma_Dirichlet}
Considering the Dirichlet boundary problem $(\ref{Sect1_PrandtlEq_Dirichlet})$
under Oleinik's monotonicity assumption, for any integer $k\geq 6$, $U(t,x)\in C^{k+1}([0,+\infty)\times\mathbb{R})$, $U(t,x)>0$, $\|U(t,x)\|_{H^{k+1}(\mathbb{R})}<+\infty$ and $\|\omega_0\|_{\mathcal{H}_{\ell}^k(\mathbb{R}_{+}^2)}<+\infty$,
there exist suitably large real numbers $\ell_0>1$ such that if $\ell\geq\ell_0$, the compatible initial data satisfies
\begin{equation}\label{Sect3_Data_Conditions2_LocalExistence}
\left\{\begin{array}{ll}
\omega_0>0, \quad u_0|_{y=0}=0, \quad \lim\limits_{y\rto +\infty} u_0 =U|_{t=0}, \\[8pt]

0< c_1 (1+y)^{-\theta} \leq \omega_0 \leq c_2 (1+y)^{-\theta}, \ \theta>\frac{\ell+1}{2},
\end{array}\right.
\end{equation}
 then there exist $T>0$ and a classical solution $(\omega,u,v)$ to the Prandtl system $(\ref{Sect1_PrandtlEq_Dirichlet})$ in $[0,T]$ satisfying the regularities $(\ref{Sect1_Solution_Regularity_Dirichlet})$.
\end{lemma}

The large time existence theorem of $(\ref{Sect1_PrandtlEq_Dirichlet})$ is stated as follows:
\begin{theorem}\label{Sect3_Existence_Thm_Dirichlet}
Assume the conditions are the same with the conditions in Theorem $\ref{Sect1_Main_Thm_Dirichlet}$.
For any fixed finite number $T\in (0,+\infty)$, there exist sufficiently small real number $0<\ve_2 =o(T^{-1})$
and suitably large real numbers $\ell_0>1$ such that if $\ell\geq\ell_0$, the compatible initial data and $U(t,x)$ satisfy
the conditions $(\ref{Sect1_Data_Conditions2})$, then the Dirichlet boundary problem $(\ref{Sect1_PrandtlEq_Dirichlet})$ admits a classical solution $(\omega,u,v)$
in $[0,T]$ satisfying the regularities $(\ref{Sect1_Solution_Regularity_Dirichlet})$.
\end{theorem}

\begin{proof}

Similar to the estimate $(\ref{Sect2_FullRegularity_2})$, we have the following estimate for $W_{\alpha,\sigma}^{n+1}$:
\begin{equation}\label{Sect3_ApproximateSol_Regularity_1_Dirichlet}
\begin{array}{ll}
\frac{\mathrm{d}}{\mathrm{d}t}\sum\limits_{|\alpha|+\sigma\leq k}\|W_{\alpha,\sigma}^{n+1}\|_{L_{\ell+\sigma}^2(\mathbb{R}_{+}^2)}^2
+ \sum\limits_{|\alpha|+\sigma\leq k}\|\partial_y W_{\alpha,\sigma}^{n+1}\|_{L_{\ell+\sigma}^2(\mathbb{R}_{+}^2)}^2 \\[10pt]

\leq \lambda_2\mathcal{L}_n\sum\limits_{|\alpha|+\sigma\leq k}\| W_{\alpha,\sigma}^{n+1}\|_{L_{\ell+\sigma}^2(\mathbb{R}_{+}^2)}^2
+ \lambda_2\|U\|_{H^{k+1}(\mathbb{R})}^2  + \lambda_2 \mathcal{L}_n.
\end{array}
\end{equation}

Similar to the proof of Theorem $\ref{Sect3_Existence_Thm_Robin}$, we can prove the following regularities:
\begin{equation}\label{Sect3_ApproximateSol_Regularity_2_Dirichlet}
\begin{array}{ll}
\omega^{n+1} \in \mathcal{H}_{\ell}^k(\mathbb{R}_{+}^2), \hspace{1.3cm}
\omega^{n+1},\ \partial_y\omega^{n+1} \in \mathcal{H}_{\ell}^k([0,T]\times\mathbb{R}_{+}^2),
\\[5pt]

u^{n+1}-U \in \mathcal{H}_{\ell-1}^k(\mathbb{R}_{+}^2)\cap \mathcal{H}_{\ell-1}^k([0,T]\times\mathbb{R}_{+}^2),
\\[5pt]

\partial_y^{j} u^{n+1}|_{y=0} \in H^{k-j}(\mathbb{R})\cap H^{k-j}([0,T]\times\mathbb{R}), \hspace{0.3cm} 0\leq j\leq k,\\[5pt]

\partial_{t,x}^{\alpha} v^{n+1} + y\cdot \partial_{t,x}^{\alpha}\partial_x U
\in L_{y, \ell-1}^{\infty}(L_{t,x}^2), \hspace{0.3cm} |\alpha|\leq k-1.
\end{array}
\end{equation}

The regularities $(\ref{Sect3_ApproximateSol_Regularity_2_Dirichlet})$ is uniform with respect to $n$, thus there exist functions $\omega,u,v$
lie in our solution spaces, then $\omega,u,v$
 satisfy the regularities $(\ref{Sect1_Solution_Regularity_Dirichlet})$.
 It is also standard to verify $\omega,u,v$ satisfy the Prandtl system $(\ref{Sect1_PrandtlEq_Dirichlet})$.
\end{proof}

\subsection{The Derivatives on the Boundary}
In this subsection, we study the asymptotic behaviors of the derivatives on the boundary as $\beta\rto +\infty$.
\begin{lemma}\label{Sect5_Lemmas_Rate}
As $\beta\rto +\infty$,
\begin{equation}\label{Sect5_Lemmas_Rate_0}
\begin{array}{ll}
\big\|u|_{y=0}\big\|_{H^k(\mathbb{R})}=O(\frac{1}{\sqrt{\beta}}),\quad \big\|\omega|_{y=0}\big\|_{H^k(\mathbb{R})}=O(\sqrt{\beta}),
\end{array}
\end{equation}
\end{lemma}
\begin{proof}
By Lemma $\ref{Sect2_FullRegularity_Lemma}$, we have the following a priori estimate:
\begin{equation}\label{Sect5_Lemmas_Rate_1}
\begin{array}{ll}
\sum\limits_{|\alpha|\leq k}\int\limits_{\mathbb{R}} \frac{1}{\beta - \frac{u_{yy}}{u_y}|_{y=0}}
 (W_{\alpha}|_{y=0})^2 \,\mathrm{d}x

\leq M^2\sum\limits_{|\alpha|+\sigma\leq k}\|W_{\alpha,\sigma}|_{t=0}\|_{L_{\ell+\sigma}^2(\mathbb{R}_{+}^2)}^2 \\[11pt]\quad

+ M^2\sum\limits_{|\alpha|\leq k}\Big[\int\limits_{\mathbb{R}} \frac{1}{\beta - \frac{u_{yy}}{u_y}|_{t=0,y=0}}
 (W_{\alpha}|_{t=0,y=0})^2 \,\mathrm{d}x + M^2\|U\|_{H^{k+1}(\mathbb{R})}^2 := C_9.
\end{array}
\end{equation}

Since $W_{\alpha}|_{y=0} =(\beta -\frac{u_{yy}}{u_y})\partial_{t,x}^{\alpha}u|_{y=0} + \partial_{t,x}^{\alpha}U\frac{u_{yy}}{u_y}$,
$(\ref{Sect5_Lemmas_Rate_1})$ implies
\begin{equation}\label{Sect5_Lemmas_Rate_2}
\begin{array}{ll}
\int\limits_{\mathbb{R}} \frac{1}{\beta - \frac{u_{yy}}{u_y}}
 [(\beta -\frac{u_{yy}}{u_y})\partial_{t,x}^{\alpha}u|_{y=0} + \partial_{t,x}^{\alpha}U\frac{u_{yy}}{u_y}]^2 \,\mathrm{d}x \leq C_9,
\end{array}
\end{equation}
then
\vspace{-0.2cm}
\begin{equation}\label{Sect5_Lemmas_Rate_3}
\begin{array}{ll}
\int\limits_{\mathbb{R}}
 (\beta - \frac{u_{yy}}{u_y})(\partial_{t,x}^{\alpha}u|_{y=0})^2 \,\mathrm{d}x
 \leq C_9  + \frac{\delta_{\beta}^2}{\beta-\delta_{\beta}}\|U\|_{H^{\alpha}(\mathbb{R})}^2 <+\infty, \\[10pt]

(\beta + \delta_{\beta})\int\limits_{\mathbb{R}}
(\partial_{t,x}^{\alpha}u|_{y=0})^2 \,\mathrm{d}x
 \leq 2 C_9  + C\mathcal{L}\|U\|_{H^{\alpha}(\mathbb{R})}^2 <+\infty,
\end{array}
\end{equation}
Then $\big\|\partial_{t,x}^{\alpha}u|_{y=0}\big\|_{L^2(\mathbb{R})}=O(\frac{1}{\sqrt{\beta}})$ as $\beta\rto +\infty$.
Thus, Lemma $\ref{Sect5_Lemmas_Rate}$ is proved.
\end{proof}

%%% find 4
\section{Interior Estimates and Boundary Estimates for the Stability}
In this section, we develop a priori estimates for the stability of the Robin boundary problem $(\ref{Sect1_PrandtlEq})$. Note that these estimates are uniform with respect to $\beta$, then we have a priori estimates for the stability of the Dirichlet boundary problem $(\ref{Sect1_PrandtlEq_Dirichlet})$ by passing to the limit.

For the proof of the stability, the lower order term $\mathcal{L}$ is related to $\bar{u},\bar{v}$ rather than $\delta u,\delta v$,
thus $\mathcal{L}$ bounded by the norms of $W_{\alpha,\sigma}$.

When $\sigma =0$, we have the estimate for $\W_{\alpha}$:
\begin{lemma}\label{Sect4_1_Stability_Estimate_Lemma}
When $\sigma =0$, $\W_{\alpha}$ satisfies $(\ref{Sect1_Stability_VorticityEq_1})$ with $\sigma=0$, then
$(\ref{Sect1_Stability_VorticityEq_1})$ with $\sigma=0$ produces the  a priori estimate $(\ref{Sect1_Estimates_4})$.
\end{lemma}

\begin{proof}
When $\sigma=0$, multiple $(\ref{Sect1_Stability_VorticityEq_1})_1$ with $ (1+y)^{2\ell}  \W_{\alpha}$, integrate in $\mathbb{R}_{+}^2$,
we get
\begin{equation}\label{Sect4_1_Estimates_1}
\begin{array}{ll}
\frac{1}{2}\frac{\mathrm{d}}{\mathrm{d}t} \|\W_{\alpha}\|_{L_{\ell}^2}^2
+ \|\partial_y \W_{\alpha}\|_{L_{\ell}^2}^2

+ \int\limits_{\mathbb{R}} \partial_y \W_{\alpha}|_{y=0}
\cdot \W_{\alpha}|_{y=0} \,\mathrm{d}x
:= {I\!I}_1 + {I\!I}_2,
\end{array}
\end{equation}
where
\begin{equation*}
\begin{array}{ll}
{I\!I}_1 :=
- \iint\limits_{\mathbb{R}_{+}^2} \Big(\bar{u}\partial_x \W_{\alpha} + \bar{v}\partial_y \W_{\alpha}
+ Q_4\W_{\alpha,\sigma}
+ \bar{u}_y\partial_y[\frac{\partial_{t,x}^{\alpha}\delta u}{\bar{u}_y} Q_5] \Big) \\[13pt]\hspace{1.05cm}
\cdot (1+y)^{2\ell} \W_{\alpha} \mathrm{d}x\mathrm{d}y

-2\ell \iint\limits_{\mathbb{R}_{+}^2} \partial_y\W_{\alpha} \cdot (1+y)^{2\ell-1}\W_{\alpha}\mathrm{d}x\mathrm{d}y
\end{array}
\end{equation*}

\begin{equation*}
\begin{array}{ll}
{I\!I}_2 :=
- \iint\limits_{\mathbb{R}_{+}^2} \bar{u}_y\partial_y\big(\frac{[\partial_{t,x}^{\alpha},\, \bar{u} \partial_x]\delta u}{\bar{u}_y}\big)
\cdot (1+y)^{2\ell} \W_{\alpha} \mathrm{d}x\mathrm{d}y \\[11pt]\hspace{1.01cm}
- \iint\limits_{\mathbb{R}_{+}^2} \bar{u}_y\partial_y\big(\frac{[\partial_{t,x}^{\alpha},\, \bar{u}_x]\delta u}{\bar{u}_y}\big)
\cdot (1+y)^{2\ell} \W_{\alpha} \mathrm{d}x\mathrm{d}y \\[11pt]\hspace{1cm}
- \iint\limits_{\mathbb{R}_{+}^2} \bar{u}_y\partial_y\big(\frac{[\partial_{t,x}^{\alpha},\, \bar{v}\partial_y]\delta u}{\bar{u}_y}\big)
\cdot (1+y)^{2\ell} \W_{\alpha} \mathrm{d}x\mathrm{d}y \\[11pt]\hspace{1cm}
- \iint\limits_{\mathbb{R}_{+}^2} \bar{u}_y\partial_y\big(\frac{[\partial_{t,x}^{\alpha},\, \bar{u}_y]\delta v}{\bar{u}_y}\big)
\cdot (1+y)^{2\ell} \W_{\alpha} \mathrm{d}x\mathrm{d}y.
\end{array}
\end{equation*}

\noindent
Multiple $(\ref{Sect1_Stability_VorticityEq_1})_2$ with $  \W_{\alpha}$, integrate in $\mathbb{R}$,
we get
\begin{equation}\label{Sect4_1_Estimates_2}
\begin{array}{ll}
\frac{1}{2}\frac{\mathrm{d}}{\mathrm{d}t} \int\limits_{\mathbb{R}} \frac{1}{\beta - \frac{\bar{u}_{yy}}{\bar{u}_y}|_{y=0}}
 (\W_{\alpha}|_{y=0})^2 \,\mathrm{d}x

- \int\limits_{\mathbb{R}} \partial_y \W_{\alpha}|_{y=0}
\cdot \W_{\alpha}|_{y=0} \,\mathrm{d}x
:= {I\!I}_3 + {I\!I}_4,
\end{array}
\end{equation}
where
\begin{equation*}
\begin{array}{ll}
{I\!I}_3 := \frac{1}{2} \int\limits_{\mathbb{R}} \partial_t(\frac{\bar{u}_{yy}}{\bar{u}_y})
 \frac{1}{(\beta - \frac{\bar{u}_{yy}}{\bar{u}_y})^2} (\W_{\alpha})^2 \,\mathrm{d}x

- \int\limits_{\mathbb{R}}  \big(\frac{1}{\beta - \frac{\bar{u}_{yy}}{\bar{u}_y}}
 \bar{u}\partial_x \W_{\alpha} \big) \cdot \W_{\alpha} \,\mathrm{d}x \\[10pt]\hspace{1cm}

+ \int\limits_{\mathbb{R}}  [\frac{\bar{u}_{yy}}{\bar{u}_y}\W_{\alpha}
- \frac{1}{\beta - \frac{\bar{u}_{yy}}{\bar{u}_y}}\W_{\alpha}Q_6] \cdot \W_{\alpha} \,\mathrm{d}x, \\[12pt]

{I\!I}_4 := - \sum\limits_{\alpha_1>0}\int\limits_{\mathbb{R}} (\partial_{t,x}^{\alpha_1}\bar{u} \cdot
\frac{\W_{\alpha_2+(0,1)}}{\beta-\frac{\bar{u}_{yy}}{\bar{u}_y}}
+ \partial_{t,x}^{\alpha_1}\partial_x\bar{u} \cdot
\frac{\W_{\alpha_2}}{\beta-\frac{\bar{u}_{yy}}{\bar{u}_y}})  \cdot \W_{\alpha} \,\mathrm{d}x.
\end{array}
\end{equation*}

By $(\ref{Sect4_1_Estimates_1}) + (\ref{Sect4_1_Estimates_2})$, we get
\begin{equation}\label{Sect4_1_Estimates_3}
\begin{array}{ll}
\frac{\mathrm{d}}{\mathrm{d}t} \|\W_{\alpha}\|_{L_{\ell}^2}^2
+\frac{\mathrm{d}}{\mathrm{d}t} \int\limits_{\mathbb{R}} \frac{1}{\beta - \frac{\bar{u}_{yy}}{\bar{u}_y}|_{y=0}}
 (\W_{\alpha}|_{y=0})^2 \,\mathrm{d}x
+ 2\|\partial_y \W_{\alpha}\|_{L_{\ell}^2}^2 \\[8pt]

=2({I\!I}_1 + {I\!I}_2 + {I\!I}_3 + {I\!I}_4).
\end{array}
\end{equation}

It is easy to obtain the estimate of ${I\!I}_1$:
\begin{equation}\label{Sect4_1_Estimates_4}
\begin{array}{ll}
{I\!I}_1 \leq q\|\partial_y \W_{\alpha}\|_{L_{\ell}^2}^2 + C\mathcal{L}\|\W_{\alpha}\|_{L_{\ell}^2}^2 \\[10pt]\hspace{1cm}
+ \iint\limits_{\mathbb{R}_{+}^2} (\W_{\alpha})^2
\cdot[\partial_y\bar{v} (1+y)^{2\ell} + 2\ell(1+y)^{2\ell-1}\bar{v}] \mathrm{d}x\mathrm{d}y
\\[9pt]\hspace{0.55cm}

\leq q\|\partial_y \W_{\alpha}\|_{L_{\ell}^2}^2 + C_{10}\mathcal{L}\|\W_{\alpha}\|_{L_{\ell}^2}^2.
\end{array}
\end{equation}

Next, we estimate the first term of $I\!I_2$:
\begin{equation}\label{Sect4_1_Estimates_5_1}
\begin{array}{ll}
- \iint\limits_{\mathbb{R}_{+}^2} \bar{u}_y\partial_y\big(\frac{[\partial_{t,x}^{\alpha},\, \bar{u} \partial_x]\delta u}{\bar{u}_y}\big)
\cdot (1+y)^{2\ell} \W_{\alpha} \mathrm{d}x\mathrm{d}y \\[9pt]

= - \iint\limits_{\mathbb{R}_{+}^2}\sum\limits_{\alpha_1>0}
(1+y)^{\ell}\bar{u}_y\partial_y\big(\frac{\partial_{t,x}^{\alpha_2}\partial_x\delta u}{\bar{u}_y}\big)\cdot \partial_{t,x}^{\alpha_1}\bar{u}
\cdot (1+y)^{\ell} \W_{\alpha} \mathrm{d}x\mathrm{d}y \\[11pt]\quad

- \iint\limits_{\mathbb{R}_{+}^2}\sum\limits_{\alpha_1>0}
(1+y)^{\ell-1}\partial_{t,x}^{\alpha_2}\partial_x\delta u \cdot (1+y)^{\ell}\partial_{t,x}^{\alpha_1}\bar{\omega}
\cdot (1+y)^{\ell} \W_{\alpha} \mathrm{d}x\mathrm{d}y \\[11pt]

\lem \big(\mathcal{L}\sum\limits_{|\alpha^{\prime}|\leq |\alpha|}\|W_{\alpha^{\prime}}\|_{L_{\ell}^2}\big)
\cdot \big(\sum\limits_{|\alpha^{\prime}|\leq |\alpha|}\|\W_{\alpha^{\prime}}\|_{L_{\ell}^2}^2 \big).
\end{array}
\end{equation}

We estimate the second term of $I\!I_2$:
\begin{equation*}
\begin{array}{ll}
- \iint\limits_{\mathbb{R}_{+}^2} \bar{u}_y\partial_y\big(\frac{[\partial_{t,x}^{\alpha},\, \bar{u}_x]\delta u}{\bar{u}_y}\big)
\cdot (1+y)^{2\ell} \W_{\alpha} \mathrm{d}x\mathrm{d}y \\[9pt]

= - \iint\limits_{\mathbb{R}_{+}^2} \sum\limits_{\alpha_1>0} (1+y)^{\ell}\bar{u}_y\partial_y\big(\frac{\partial_{t,x}^{\alpha_2}\delta u}{\bar{u}_y}\big)
\cdot\partial_{t,x}^{\alpha_1}\partial_x\bar{u}
\cdot (1+y)^{\ell} \W_{\alpha} \mathrm{d}x\mathrm{d}y
\end{array}
\end{equation*}

\begin{equation}\label{Sect4_1_Estimates_5_2}
\begin{array}{ll}
\quad
- \iint\limits_{\mathbb{R}_{+}^2} \sum\limits_{\alpha_1>0} (1+y)^{\ell-1}\partial_{t,x}^{\alpha_2}\delta u \cdot (1+y)^{\ell}
\partial_{t,x}^{\alpha_1}\partial_x\bar{\omega}\cdot (1+y)^{\ell} \W_{\alpha} \mathrm{d}x\mathrm{d}y
\\[11pt]

\lem \big(\mathcal{L}\sum\limits_{|\alpha^{\prime}|\leq |\alpha|+1}\|W_{\alpha^{\prime}}\|_{L_{\ell}^2}\big)
\cdot \big(\sum\limits_{|\alpha^{\prime}|\leq |\alpha|}\|\W_{\alpha^{\prime}}\|_{L_{\ell}^2}^2 \big).
\end{array}
\end{equation}

We estimate the third term of $I\!I_2$:
\begin{equation}\label{Sect4_1_Estimates_5_3}
\begin{array}{ll}
- \iint\limits_{\mathbb{R}_{+}^2} \bar{u}_y\partial_y\big(\frac{[\partial_{t,x}^{\alpha},\, \bar{v}\partial_y]\delta u}{\bar{u}_y}\big)
\cdot (1+y)^{2\ell} \W_{\alpha} \mathrm{d}x\mathrm{d}y \\[7pt]

= - \iint\limits_{\mathbb{R}_{+}^2}\sum\limits_{\alpha_1>0} (1+y)^{-1}\partial_{t,x}^{\alpha_1}\bar{v}\cdot
(1+y)^{\ell+1}\bar{u}_y\partial_y\big(\frac{\partial_{t,x}^{\alpha_2}
\partial_y\delta u}{\bar{u}_y}\big)
\cdot (1+y)^{\ell} \W_{\alpha} \mathrm{d}x\mathrm{d}y \\[9pt]\quad
+ \iint\limits_{\mathbb{R}_{+}^2}\sum\limits_{\alpha_1>0} \partial_{t,x}^{\alpha_1}\partial_x \bar{u} \cdot
(1+y)^{\ell}\partial_{t,x}^{\alpha_2}\delta\omega
\cdot (1+y)^{\ell} \W_{\alpha} \mathrm{d}x\mathrm{d}y \\[9pt]

\lem \big(\mathcal{L}\sum\limits_{|\alpha^{\prime}|\leq |\alpha|+1}\|W_{\alpha^{\prime}}\|_{L_{\ell}^2}\big)
\cdot \big(\|\W_{\alpha}\|_{L_{\ell}^2}^2
+ \sum\limits_{|\alpha^{\prime}|\leq |\alpha|-1}\|\W_{\alpha^{\prime},1}\|_{L_{\ell+1}^2}^2 \big) \\[11pt]\quad

+ \big(\mathcal{L}\sum\limits_{|\alpha^{\prime}|\leq |\alpha|+1}\|W_{\alpha^{\prime}}\|_{L_{\ell}^2}\big)
\cdot \big(\sum\limits_{|\alpha^{\prime}|\leq |\alpha|}\|\W_{\alpha^{\prime}}\|_{L_{\ell}^2}^2 \big).
\end{array}
\end{equation}

We estimate the fourth term of $I\!I_2$:
\begin{equation}\label{Sect4_1_Estimates_5_4}
\begin{array}{ll}
- \iint\limits_{\mathbb{R}_{+}^2} \bar{u}_y\partial_y\big(\frac{[\partial_{t,x}^{\alpha},\, \bar{u}_y]\delta v}{\bar{u}_y}\big)
\cdot (1+y)^{2\ell} \W_{\alpha} \mathrm{d}x\mathrm{d}y \\[9pt]

= - \iint\limits_{\mathbb{R}_{+}^2}\sum\limits_{\alpha_1>0} (1+y)^{\ell+1}\bar{u}_y\partial_y\big(\frac{\partial_{t,x}^{\alpha_1}\partial_y\bar{u}}{\bar{u}_y}\big)
(1+y)^{-1}\partial_{t,x}^{\alpha_2}\delta v
\cdot (1+y)^{\ell} \W_{\alpha} \mathrm{d}x\mathrm{d}y \\[9pt]\quad
+ \iint\limits_{\mathbb{R}_{+}^2}\sum\limits_{\alpha_1>0} (1+y)^{\ell}\partial_{t,x}^{\alpha_1}\bar{\omega}\cdot (1+y)^{\ell-1}\partial_x\partial_{t,x}^{\alpha_2}
\delta u \cdot (1+y)^{\ell} \W_{\alpha} \mathrm{d}x\mathrm{d}y
\\[14pt]

\lem \sum\limits_{|\alpha^{\prime}|\leq |\alpha|}\big(\mathcal{L}\|W_{\alpha^{\prime},1}\|_{L_{\ell+1}^2}
+ \mathcal{L}\|W_{\alpha^{\prime}}\|_{L_{\ell}^2}\big)
\cdot \big(\sum\limits_{|\alpha^{\prime}|\leq |\alpha|}\|\W_{\alpha^{\prime}}\|_{L_{\ell}^2}^2 \big).
\end{array}
\end{equation}

\vspace{-0.2cm}
Thus,
\begin{equation}\label{Sect4_1_Estimates_5}
\begin{array}{ll}
{I\!I}_2
\leq C_{11} \big(\mathcal{L}\sum\limits_{|\alpha^{\prime}|\leq |\alpha|}\|W_{\alpha^{\prime},1}\|_{L_{\ell+1}^2}
+ \mathcal{L}\sum\limits_{|\alpha^{\prime}|\leq |\alpha|+1}\|W_{\alpha^{\prime}}\|_{L_{\ell}^2}\big) \\[11pt]\hspace{1cm}
\cdot \big(\sum\limits_{|\alpha^{\prime}|\leq |\alpha|-1}\|\W_{\alpha^{\prime},1}\|_{L_{\ell+1}^2}^2
+ \sum\limits_{|\alpha^{\prime}|\leq |\alpha|}\|\W_{\alpha^{\prime}}\|_{L_{\ell}^2}^2 \big).
\end{array}
\end{equation}

\vspace{-0.2cm}
Next we estimate $I\!I_3$:
\begin{equation}\label{Sect4_1_Estimates_6}
\begin{array}{ll}
I\!I_3 \leq q\|\partial_y \W_{\alpha}\|_{L_{\ell}^2}^2 + C_{12}\mathcal{L}\int\limits_{\mathbb{R}}
 \frac{1}{\beta - \frac{\bar{u}_{yy}}{\bar{u}_y}|_{y=0}} (\W_{\alpha}|_{y=0})^2 \,\mathrm{d}x.
\end{array}
\end{equation}

\vspace{-0.2cm}
Next we estimate $I\!I_4$:
\begin{equation}\label{Sect4_1_Estimates_7}
\begin{array}{ll}
I\!I_4 \lem (\sum\limits_{|\alpha^{\prime}|\leq |\alpha|+1}\|W_{\alpha^{\prime}}\|_{L^2(\mathbb{R})} + \|U\|_{H^{|\alpha|+1}})\sum\limits_{|\alpha^{\prime}|\leq |\alpha|}\int\limits_{\mathbb{R}}
 \frac{1}{\beta - \frac{\bar{u}_{yy}}{\bar{u}_y}|_{y=0}} (\W_{\alpha^{\prime}}|_{y=0})^2 \,\mathrm{d}x.
\end{array}
\end{equation}

Plug $(\ref{Sect4_1_Estimates_4}),  (\ref{Sect4_1_Estimates_5}), (\ref{Sect4_1_Estimates_6}), (\ref{Sect4_1_Estimates_7})$ into
$(\ref{Sect4_1_Estimates_3})$, we get
\begin{equation}\label{Sect4_1_Estimates_8}
\begin{array}{ll}
\frac{\mathrm{d}}{\mathrm{d}t} \|\W_{\alpha}\|_{L_{\ell}^2}^2
+\frac{\mathrm{d}}{\mathrm{d}t} \int\limits_{\mathbb{R}} \frac{1}{\beta - \frac{\bar{u}_{yy}}{\bar{u}_y}|_{y=0}}
 (\W_{\alpha}|_{y=0})^2 \,\mathrm{d}x
+ \|\partial_y \W_{\alpha}\|_{L_{\ell}^2}^2 \\[9pt]

\lem \big(\mathcal{L}\sum\limits_{|\alpha^{\prime}|\leq |\alpha|}\|W_{\alpha^{\prime},1}\|_{L_{\ell+1}^2}
+ \mathcal{L}\sum\limits_{|\alpha^{\prime}|\leq |\alpha|+1}\|W_{\alpha^{\prime}}\|_{L_{\ell}^2} + \|U\|_{H^{|\alpha|+1}} +\mathcal{L}\big) \\[11pt]\quad
\cdot \Big[\sum\limits_{|\alpha^{\prime}|\leq |\alpha|-1}\|\W_{\alpha^{\prime},1}\|_{L_{\ell+1}^2}^2
+ \sum\limits_{|\alpha^{\prime}|\leq |\alpha|} \big(\|\W_{\alpha^{\prime}}\|_{L_{\ell}^2}^2
+ \int\limits_{\mathbb{R}}
 \frac{1}{\beta - \frac{\bar{u}_{yy}}{\bar{u}_y}|_{y=0}} (\W_{\alpha^{\prime}}|_{y=0})^2 \,\mathrm{d}x \big) \Big].
\end{array}
\end{equation}

Thus, Lemma $\ref{Sect4_1_Stability_Estimate_Lemma}$ is proved.
\end{proof}

When $\sigma \geq 1,\ |\alpha|\leq p-\sigma$, we have the estimate for $\W_{\alpha,\sigma}$:
\begin{lemma}\label{Sect4_2_Stability_Estimate_Lemma}
When $\sigma \geq 1$, $\W_{\alpha,\sigma}$ satisfies $(\ref{Sect1_Stability_VorticityEq_1})$ with $\sigma\geq 1$, then
$(\ref{Sect1_Stability_VorticityEq_1})$ with $\sigma\geq 1$ produces the a priori estimate $(\ref{Sect1_Estimates_5})$.
\end{lemma}

\begin{proof}
When $\sigma\geq 1$, multiple $(\ref{Sect1_Stability_VorticityEq_1})_1$ with $ (1+y)^{2\ell+2\sigma}  \W_{\alpha,\sigma}$, integrate in $\mathbb{R}_{+}^2$,
we get
\begin{equation}\label{Sect4_2_Estimates_1}
\begin{array}{ll}
\frac{1}{2}\frac{\mathrm{d}}{\mathrm{d}t} \|\W_{\alpha,\sigma}\|_{L_{\ell+\sigma}^2}^2
+ \|\partial_y \W_{\alpha,\sigma}\|_{L_{\ell+\sigma}^2}^2
+ \int\limits_{\mathbb{R}} \partial_y \W_{\alpha,\sigma}|_{y=0}
\cdot \W_{\alpha,\sigma}|_{y=0} \,\mathrm{d}x

:= {I\!I}_5 + {I\!I}_6,
\end{array}
\end{equation}
where
\begin{equation*}
\begin{array}{ll}
{I\!I}_5 := -(2\ell+2\sigma) \iint\limits_{\mathbb{R}_{+}^2} \partial_y\W_{\alpha} \cdot (1+y)^{2\ell+2\sigma-1}\W_{\alpha}\mathrm{d}x\mathrm{d}y
\\[9pt]\hspace{1.05cm}
- \iint\limits_{\mathbb{R}_{+}^2} \Big(\bar{u}\partial_x \W_{\alpha,\sigma} + \bar{v}\partial_y \W_{\alpha,\sigma}
+ Q_4\W_{\alpha,\sigma}
+ \bar{u}_y\partial_y[\frac{\partial_{t,x}^{\alpha}\partial_y^{\sigma}\delta u}{\bar{u}_y} Q_5] \Big) \\[11pt]\hspace{1.05cm}
\cdot (1+y)^{2\ell+2\sigma} \W_{\alpha,\sigma} \mathrm{d}x\mathrm{d}y \\[10pt]

{I\!I}_6 :=
- \iint\limits_{\mathbb{R}_{+}^2} \bar{u}_y\partial_y\big(\frac{[\partial_{t,x}^{\alpha}\partial_y^{\sigma},\, \bar{u} \partial_x]\delta u}{\bar{u}_y}\big)
\cdot (1+y)^{2\ell+2\sigma} \W_{\alpha,\sigma} \mathrm{d}x\mathrm{d}y \\[8pt]\hspace{1cm}
- \iint\limits_{\mathbb{R}_{+}^2} \bar{u}_y\partial_y\big(\frac{[\partial_{t,x}^{\alpha}\partial_y^{\sigma},\, \bar{u}_x]\delta u}{\bar{u}_y}\big)
\cdot (1+y)^{2\ell+2\sigma} \W_{\alpha,\sigma} \mathrm{d}x\mathrm{d}y \\[8pt]\hspace{1cm}
- \iint\limits_{\mathbb{R}_{+}^2} \bar{u}_y\partial_y\big(\frac{[\partial_{t,x}^{\alpha}\partial_y^{\sigma},\, \bar{v}\partial_y]\delta u}{\bar{u}_y}\big)
\cdot (1+y)^{2\ell+2\sigma} \W_{\alpha,\sigma} \mathrm{d}x\mathrm{d}y \\[8pt]\hspace{1cm}
- \iint\limits_{\mathbb{R}_{+}^2} \bar{u}_y\partial_y\big(\frac{[\partial_{t,x}^{\alpha}\partial_y^{\sigma},\, \bar{u}_y]\delta v}{\bar{u}_y}\big)
\cdot (1+y)^{2\ell+2\sigma} \W_{\alpha,\sigma} \mathrm{d}x\mathrm{d}y.
\end{array}
\end{equation*}

\noindent
Multiple $(\ref{Sect1_Stability_VorticityEq_1})_2$ with $  \W_{\alpha,\sigma}$, integrate in $\mathbb{R}$,
we get
\begin{equation}\label{Sect4_2_Estimates_2}
\begin{array}{ll}
\frac{1}{2}\frac{\mathrm{d}}{\mathrm{d}t} \int\limits_{\mathbb{R}} \frac{1}{\beta - \frac{\bar{u}_{yy}}{\bar{u}_y}|_{y=0}}
 (\W_{\alpha,\sigma}|_{y=0})^2 \,\mathrm{d}x

- \int\limits_{\mathbb{R}} \partial_y \W_{\alpha,\sigma}|_{y=0}
\cdot \W_{\alpha,\sigma}|_{y=0} \,\mathrm{d}x

:= {I\!I}_7,
\end{array}
\end{equation}
where
\vspace{-0.2cm}
\begin{equation*}
\begin{array}{ll}
{I\!I}_7 := \int\limits_{\mathbb{R}}  \mathcal{P}_2\Big(\mathcal{L}+ \sum\limits_{|\alpha^{\prime}|\leq |\alpha|+1}\sum\limits_{m=0}^{\sigma} (W_{\alpha^{\prime},\sigma-m}) (\frac{\bar{u}_{yy}}{\bar{u}_y})^m \\[8pt]\hspace{1cm}
 + \partial_x \partial_{t,x}^{\alpha^{\prime}} (\bar{u}-U)(\frac{\bar{u}_{yy}}{\bar{u}_y})^{\sigma},

\sum\limits_{|\alpha^{\prime}|\leq |\alpha|+1}\sum\limits_{m=0}^{\sigma-1} (\W_{\alpha^{\prime},\sigma-1-m}) (\frac{\bar{u}_{yy}}{\bar{u}_y})^m
 \\[9pt]\hspace{1cm}
 + \sum\limits_{|\alpha^{\prime}|\leq |\alpha|+1}\frac{1}{\beta - \frac{\bar{u}_{yy}}{\bar{u}_y}}\W_{\alpha^{\prime}}(\frac{\bar{u}_{yy}}{\bar{u}_y})^{\sigma}\Big) \cdot \W_{\alpha,\sigma} \,\mathrm{d}x.
\end{array}
\end{equation*}

By $(\ref{Sect4_2_Estimates_1}) + (\ref{Sect4_2_Estimates_2})$, we get
\begin{equation}\label{Sect4_2_Estimates_3}
\begin{array}{ll}
\frac{\mathrm{d}}{\mathrm{d}t} \|\W_{\alpha,\sigma}\|_{L_{\ell}^2}^2
+\frac{\mathrm{d}}{\mathrm{d}t} \int\limits_{\mathbb{R}} \frac{1}{\beta - \frac{\bar{u}_{yy}}{\bar{u}_y}}
 (\W_{\alpha,\sigma})^2 \,\mathrm{d}x
+ 2\|\partial_y \W_{\alpha,\sigma}\|_{L_{\ell}^2}^2

=2({I\!I}_5 + {I\!I}_6 + {I\!I}_7).
\end{array}
\end{equation}

Similar to $(\ref{Sect4_1_Estimates_4})$, we have
\begin{equation}\label{Sect4_2_Estimates_4}
\begin{array}{ll}
{I\!I}_5 \leq q\| \partial_y\W_{\alpha,\sigma}\|_{L_{\ell+\sigma}^2}^2
+ C_{13}\mathcal{L}\|\W_{\alpha,\sigma}\|_{L_{\ell+\sigma}^2}^2.
\end{array}
\end{equation}

Next we estimate the first term of $I\!I_6$:
\begin{equation}\label{Sect4_2_Estimates_5_1}
\begin{array}{ll}
- \iint\limits_{\mathbb{R}_{+}^2} \bar{u}_y\partial_y\big(\frac{[\partial_{t,x}^{\alpha}\partial_y^{\sigma},\, \bar{u} \partial_x]\delta u}{\bar{u}_y}\big)
\cdot (1+y)^{2\ell+2\sigma} \W_{\alpha,\sigma} \mathrm{d}x\mathrm{d}y \\[6pt]

= - \iint\limits_{\mathbb{R}_{+}^2}\sum\limits_{|\alpha_1|+\sigma_1>0}
(1+y)^{\ell+\sigma_2}\bar{u}_y\partial_y\big(\frac{\partial_{t,x}^{\alpha_2}\partial_y^{\sigma_2}\partial_x\delta u}{\bar{u}_y}\big)\cdot
(1+y)^{\ell+\sigma_1 -1}\partial_{t,x}^{\alpha_1}\partial_y^{\sigma_1}\bar{u}

\\[11pt]\quad
\cdot (1+y)^{\ell+\sigma} \W_{\alpha,\sigma} \mathrm{d}x\mathrm{d}y \\[6pt]\quad

- \iint\limits_{\mathbb{R}_{+}^2}\sum\limits_{|\alpha_1|+\sigma_1>0}
(1+y)^{\ell+\sigma_2-1}\partial_{t,x}^{\alpha_2}\partial_y^{\sigma_2}\partial_x\delta u \cdot (1+y)^{\ell+\sigma_1}\partial_{t,x}^{\alpha_1}\partial_y^{\sigma_1}\bar{\omega}

\\[8pt]\quad
\cdot (1+y)^{\ell+\sigma} \W_{\alpha,\sigma} \mathrm{d}x\mathrm{d}y \\[7pt]

\lem Q_7 \big(\sum\limits_{|\alpha^{\prime}|\leq |\alpha|,\sigma^{\prime}\leq \sigma}\|\W_{\alpha^{\prime},\sigma^{\prime}}\|_{L_{\ell+\sigma^{\prime}}^2}^2
+ \sum\limits_{|\alpha^{\prime}|\leq |\alpha|+1,\sigma^{\prime}\leq \sigma-1}\|\W_{\alpha^{\prime},\sigma^{\prime}}\|_{L_{\ell+\sigma^{\prime}}^2}^2\big).
\end{array}
\end{equation}

We estimate the second term of $I\!I_6$:
\begin{equation}\label{Sect4_2_Estimates_5_2}
\begin{array}{ll}
- \iint\limits_{\mathbb{R}_{+}^2} \bar{u}_y\partial_y\big(\frac{[\partial_{t,x}^{\alpha}\partial_y^{\sigma},\, \bar{u}_x]\delta u}{\bar{u}_y}\big)
\cdot (1+y)^{2\ell+2\sigma} \W_{\alpha,\sigma} \mathrm{d}x\mathrm{d}y \\[9pt]

= - \iint\limits_{\mathbb{R}_{+}^2} \sum\limits_{\alpha_1>0} (1+y)^{\ell+\sigma_2}\bar{u}_y\partial_y
\big(\frac{\partial_{t,x}^{\alpha_2}\partial_y^{\sigma_2}\delta u}{\bar{u}_y}\big)
\cdot (1+y)^{\ell+\sigma_1-1}\partial_{t,x}^{\alpha_1}\partial_y^{\sigma_1}\partial_x\bar{u} \\[9pt]\quad
\cdot (1+y)^{\ell+\sigma} \W_{\alpha,\sigma} \mathrm{d}x\mathrm{d}y \\[6pt]\quad

- \iint\limits_{\mathbb{R}_{+}^2} \sum\limits_{\alpha_1>0} (1+y)^{\ell+\sigma_2-1}\partial_{t,x}^{\alpha_2}\partial_y^{\sigma_2}\delta u \cdot (1+y)^{\ell+\sigma_1}
\partial_{t,x}^{\alpha_1}\partial_y^{\sigma_1}\partial_x\bar{\omega} \\[9pt]\quad
\cdot (1+y)^{\ell+\sigma} \W_{\alpha,\sigma} \mathrm{d}x\mathrm{d}y
\\[7pt]

\lem Q_7 \sum\limits_{|\alpha^{\prime}|\leq |\alpha|,\sigma^{\prime}\leq \sigma}\|\W_{\alpha^{\prime},\sigma^{\prime}}\|_{L_{\ell+\sigma^{\prime}}^2}^2.
\end{array}
\end{equation}

We estimate the third term of $I\!I_6$.

When $\sigma\leq p-1, 0<|\alpha|\leq p-\sigma$, we have
\begin{equation}\label{Sect4_2_Estimates_5_3}
\begin{array}{ll}
- \iint\limits_{\mathbb{R}_{+}^2} \bar{u}_y\partial_y\big(\frac{[\partial_{t,x}^{\alpha}\partial_y^{\sigma},\, \bar{v}\partial_y]\delta u}{\bar{u}_y}\big)
\cdot (1+y)^{2\ell+2\sigma} \W_{\alpha,\sigma} \mathrm{d}x\mathrm{d}y \\[7pt]

= - \iint\limits_{\mathbb{R}_{+}^2}\sum\limits_{|\alpha_1|+\sigma_1>0} (1+y)^{\sigma_1-1}\partial_{t,x}^{\alpha_1}\partial_y^{\sigma_1}\bar{v}\cdot
(1+y)^{\ell+\sigma_2+1}\bar{u}_y\partial_y\big(\frac{\partial_{t,x}^{\alpha_2}\partial_y^{\sigma_2}
\partial_y\delta u}{\bar{u}_y}\big) \\[9pt]\quad
\cdot (1+y)^{\ell+\sigma} \W_{\alpha,\sigma} \mathrm{d}x\mathrm{d}y \\[7pt]\quad

+ \iint\limits_{\mathbb{R}_{+}^2}\sum\limits_{|\alpha_1|+\sigma_1>0} (1+y)^{\ell+\sigma_1-1}
\partial_{t,x}^{\alpha_1}\partial_y^{\sigma_1}\partial_x \bar{u} \cdot
(1+y)^{\ell+\sigma_2}\partial_{t,x}^{\alpha_2}\partial_y^{\sigma_2}\delta\omega \\[9pt]\quad
\cdot (1+y)^{\ell+\sigma} \W_{\alpha,\sigma} \mathrm{d}x\mathrm{d}y \\[7pt]

\lem Q_7 \big(\sum\limits_{|\alpha^{\prime}|\leq |\alpha|,\sigma^{\prime}\leq \sigma}\|\W_{\alpha^{\prime},\sigma^{\prime}}\|_{L_{\ell+\sigma^{\prime}}^2}^2
+ \sum\limits_{|\alpha^{\prime}|\leq |\alpha|-1}\|\W_{\alpha^{\prime},\sigma+1}\|_{L_{\ell+\sigma+1}^2}^2 \big).
\end{array}
\end{equation}

When $\sigma\leq p, |\alpha|=0$, we have
\begin{equation*}
\begin{array}{ll}
- \iint\limits_{\mathbb{R}_{+}^2} \bar{u}_y\partial_y\big(\frac{[\partial_{t,x}^{\alpha}\partial_y^{\sigma},\, \bar{v}\partial_y]\delta u}{\bar{u}_y}\big)
\cdot (1+y)^{2\ell+2\sigma} \W_{\alpha,\sigma} \mathrm{d}x\mathrm{d}y \\[8pt]

= \iint\limits_{\mathbb{R}_{+}^2}  \sum\limits_{\sigma_1>0}\partial_y^{\sigma_1}\bar{v}\partial_y^{\sigma_2+1}\delta u
\cdot (1+y)^{2\ell+2\sigma} \partial_y\W_{\alpha,\sigma} \mathrm{d}x\mathrm{d}y \hspace{3cm}
\end{array}
\end{equation*}

\begin{equation}\label{Sect4_2_Estimates_5_4}
\begin{array}{ll}
\quad
+ \iint\limits_{\mathbb{R}_{+}^2}  \sum\limits_{\sigma_1>0} (1+y)^{\sigma_1-1}\partial_y^{\sigma_1}\bar{v}
(1+y)^{\ell+\sigma_2}\partial_y^{\sigma_2+1}\delta u \\[8pt]\quad
\cdot\big((1+y)\frac{\bar{u}_{yy}}{\bar{u}_y}(1+y)^{\ell+\sigma}\W_{\alpha,\sigma}
+ (2\ell+2\sigma) (1+y)^{\ell+\sigma} \W_{\alpha,\sigma}
\big) \mathrm{d}x\mathrm{d}y \\[8pt]

\lem \iint\limits_{\mathbb{R}_{+}^2} -\sum\limits_{\sigma_1>0}
(1+y)^{\ell+\sigma_1-1}\partial_y^{\sigma_1}\partial_x\bar{u}\cdot (1+y)^{\ell+\sigma_2}\partial_y^{\sigma_2+1}\delta u
\cdot (1+y)^{\ell+\sigma} \W_{\alpha,\sigma} \\[8pt]\quad

+ \sum\limits_{\sigma_1>0}(1+y)^{\sigma_1-1}\partial_y^{\sigma_1}\bar{v}\cdot (1+y)^{\ell+\sigma_2+1}\partial_y^{\sigma_2+2}\delta u
\cdot (1+y)^{\ell+\sigma} \W_{\alpha,\sigma} \\[8pt]\quad

+ (2\ell+2\sigma)\sum\limits_{\sigma_1>0}(1+y)^{\sigma_1-1}\partial_y^{\sigma_1}\bar{v}\cdot(1+y)^{\ell+\sigma_2}\partial_y^{\sigma_2+1}\delta u
\cdot (1+y)^{\ell+\sigma} \W_{\alpha,\sigma} \mathrm{d}x\mathrm{d}y \\[8pt]\quad

+ \iint\limits_{\mathbb{R}_{+}^2}  \sum\limits_{\sigma_1>0} (1+y)^{\sigma_1-1}\partial_y^{\sigma_1}\bar{v}
(1+y)^{\ell+\sigma_2}\partial_y^{\sigma_2+1}\delta u \\[8pt]\quad
\cdot\big((1+y)\frac{\bar{u}_{yy}}{\bar{u}_y}(1+y)^{\ell+\sigma}\W_{\alpha,\sigma}
+ (2\ell+2\sigma) (1+y)^{\ell+\sigma} \W_{\alpha,\sigma}
\big) \mathrm{d}x\mathrm{d}y \\[8pt]

\lem Q_7 \sum\limits_{|\alpha^{\prime}|\leq |\alpha|,\sigma^{\prime}\leq \sigma}\|\W_{\alpha^{\prime},\sigma^{\prime}}\|_{L_{\ell+\sigma^{\prime}}^2}^2 .
\end{array}
\end{equation}

\vspace{-0.2cm}
We estimate the fourth term of $I\!I_6$:
\begin{equation}\label{Sect4_2_Estimates_5_5}
\begin{array}{ll}
- \iint\limits_{\mathbb{R}_{+}^2} \bar{u}_y\partial_y\big(\frac{[\partial_{t,x}^{\alpha}\partial_y^{\sigma},\, \bar{u}_y]\delta v}{\bar{u}_y}\big)
\cdot (1+y)^{2\ell+2\sigma} \W_{\alpha,\sigma} \mathrm{d}x\mathrm{d}y \\[9pt]

= - \iint\limits_{\mathbb{R}_{+}^2}\sum\limits_{\alpha_1>0} (1+y)^{\ell+\sigma_1+1}\bar{u}_y\partial_y\big(\frac{\partial_{t,x}^{\alpha_1}\partial_y^{\sigma_1}\partial_y\bar{u}}{\bar{u}_y}\big)
(1+y)^{\sigma_2-1}\partial_{t,x}^{\alpha_2}\partial_y^{\sigma_2}\delta v \\[9pt]\quad
\cdot (1+y)^{\ell+\sigma} \W_{\alpha,\sigma} \mathrm{d}x\mathrm{d}y \\[9pt]\quad

+ \iint\limits_{\mathbb{R}_{+}^2}\sum\limits_{\alpha_1>0} (1+y)^{\ell+\sigma_1}\partial_{t,x}^{\alpha_1}\partial_y^{\sigma_1}\bar{\omega}\cdot (1+y)^{\ell+\sigma_2-1}\partial_x\partial_{t,x}^{\alpha_2}\partial_y^{\sigma_2}\delta u \\[9pt]\quad
\cdot (1+y)^{\ell+\sigma} \W_{\alpha,\sigma} \mathrm{d}x\mathrm{d}y \\[9pt]

\lem Q_7 \sum\limits_{|\alpha^{\prime}|\leq |\alpha|+1,\sigma^{\prime}\leq \sigma-1}\|\W_{\alpha^{\prime},\sigma^{\prime}}\|_{L_{\ell+\sigma^{\prime}}^2}^2.
\end{array}
\end{equation}

Thus, when $0<\sigma\leq p-1,\ 0<|\alpha|\leq p-\sigma$,
\begin{equation}\label{Sect4_2_Estimates_5_6}
\begin{array}{ll}
{I\!I}_6
\lem Q_7 \cdot \big(\sum\limits_{|\alpha^{\prime}|\leq |\alpha|,\sigma^{\prime}\leq \sigma}\|\W_{\alpha^{\prime},\sigma^{\prime}}\|_{L_{\ell+\sigma^{\prime}}^2}^2
+ \sum\limits_{|\alpha^{\prime}|\leq |\alpha|+1,\sigma^{\prime}\leq \sigma-1}\|\W_{\alpha^{\prime},\sigma^{\prime}}\|_{L_{\ell+\sigma^{\prime}}^2}^2
\\[15pt]\hspace{0.95cm}
+ \sum\limits_{|\alpha^{\prime}|\leq |\alpha|-1}\|\W_{\alpha^{\prime},\sigma+1}\|_{L_{\ell+\sigma+1}^2}^2\big).
\end{array}
\end{equation}

When $0<\sigma\leq p,\ |\alpha|=0$,
\begin{equation}\label{Sect4_2_Estimates_5_7}
\begin{array}{ll}
{I\!I}_6
\lem Q_7 \cdot \big(\sum\limits_{|\alpha^{\prime}|\leq |\alpha|,\sigma^{\prime}\leq \sigma}\|\W_{\alpha^{\prime},\sigma^{\prime}}\|_{L_{\ell+\sigma^{\prime}}^2}^2
+ \sum\limits_{|\alpha^{\prime}|\leq |\alpha|+1,\sigma^{\prime}\leq \sigma-1}\|\W_{\alpha^{\prime},\sigma^{\prime}}\|_{L_{\ell+\sigma^{\prime}}^2}^2\big).
\end{array}
\end{equation}

\vspace{-0.2cm}
Next we estimate $I\!I_7$:
\begin{equation}\label{Sect4_2_Estimates_6}
\begin{array}{ll}
I\!I_7 \leq
Q_7 \cdot\big(\sum\limits_{|\alpha^{\prime}| \leq |\alpha|+1, \sigma^{\prime}\leq \sigma-1}\|\W_{\alpha^{\prime},\sigma^{\prime}}|_{y=0}\|_{L^2}^2
\\[12pt]\hspace{1.75cm}
+ \sum\limits_{|\alpha^{\prime}|\leq |\alpha|+1}\int\limits_{\mathbb{R}}
 \frac{1}{\beta - \frac{\bar{u}_{yy}}{\bar{u}_y}|_{y=0}} (\W_{\alpha^{\prime}}|_{y=0})^2 \,\mathrm{d}x\big),
\end{array}
\end{equation}
where $Q_7$ is defined as $(\ref{Sect1_Estimates_6})$. Note that the second term of $Q_7$ is bounded.
Since $\sigma+1\leq k$ and $|\alpha^{\prime}|+\sigma^{\prime} \leq p$, we have
\begin{equation}\label{Sect4_2_Estimates_8}
\begin{array}{ll}
\sum\limits_{|\alpha^{\prime}| \leq |\alpha|, \sigma^{\prime}\leq \sigma}\|W_{\alpha^{\prime},\sigma^{\prime}}|_{y=0}\|_{L^2(\mathbb{R})}

\lem \sum\limits_{|\alpha^{\prime}| \leq |\alpha|, \sigma^{\prime}\leq \sigma+1}\|W_{\alpha^{\prime},\sigma^{\prime}}\|_{L_{\ell+\sigma^{\prime}(\mathbb{R}_{+}^2)}^2}.
\end{array}
\end{equation}

\begin{equation}\label{Sect4_2_Estimates_9}
\begin{array}{ll}
I\!I_7
\leq q\sum\limits_{|\alpha^{\prime}| \leq |\alpha|+1, \sigma^{\prime}\leq \sigma-1}\|\partial_y \W_{\alpha^{\prime},\sigma^{\prime}}\|_{L_{\ell+\sigma}^2}^2
\\[10pt]\hspace{1cm}
+ Q_7\sum\limits_{|\alpha^{\prime}|\leq |\alpha|+1}\int\limits_{\mathbb{R}}
 \frac{1}{\beta - \frac{\bar{u}_{yy}}{\bar{u}_y}|_{y=0}} (\W_{\alpha^{\prime}}|_{y=0})^2 \,\mathrm{d}x,
\end{array}
\end{equation}
where $q$ is small if $\ell_0$ is suitably large compared with $Q_7$.

When $0<\sigma\leq p-1,\ 0<|\alpha|\leq p-\sigma$, plug $(\ref{Sect4_2_Estimates_4}),  (\ref{Sect4_2_Estimates_5_6}), (\ref{Sect4_2_Estimates_9})$ into
$(\ref{Sect4_2_Estimates_3})$,
\begin{equation}\label{Sect4_2_Estimates_10}
\begin{array}{ll}
\frac{\mathrm{d}}{\mathrm{d}t} \|\W_{\alpha,\sigma}\|_{L_{\ell+\sigma}^2}^2
+\frac{\mathrm{d}}{\mathrm{d}t} \int\limits_{\mathbb{R}} \frac{1}{\beta - \frac{\bar{u}_{yy}}{\bar{u}_y}|_{y=0}}
 (\W_{\alpha}|_{y=0})^2 \,\mathrm{d}x
+ \|\partial_y \W_{\alpha,\sigma}\|_{L_{\ell+\sigma}^2}^2 \\[9pt]

\leq  Q_7 \big(\sum\limits_{|\alpha^{\prime}|\leq |\alpha|,\sigma^{\prime}\leq \sigma}\|\W_{\alpha^{\prime},\sigma^{\prime}}\|_{L_{\ell+\sigma^{\prime}}^2}^2
+ \sum\limits_{|\alpha^{\prime}|\leq |\alpha|+1,\sigma^{\prime}\leq \sigma-1}\|\W_{\alpha^{\prime},\sigma^{\prime}}\|_{L_{\ell+\sigma^{\prime}}^2}^2
\\[9pt]\quad
+ \sum\limits_{|\alpha^{\prime}|\leq |\alpha|-1}\|\W_{\alpha^{\prime},\sigma+1}\|_{L_{\ell+\sigma+1}^2}^2
+ \sum\limits_{|\alpha^{\prime}|\leq |\alpha|+1}\int\limits_{\mathbb{R}} \frac{1}{\beta - \frac{\bar{u}_{yy}}{\bar{u}_y}|_{y=0}} (\W_{\alpha}|_{y=0})^2 \,\mathrm{d}x \big) \\[12pt]\quad

+ q\sum\limits_{|\alpha^{\prime}| \leq |\alpha|+1, \sigma^{\prime}\leq \sigma-1}\|\partial_y \W_{\alpha^{\prime},\sigma^{\prime}}\|_{L_{\ell+\sigma}^2}^2.
\end{array}
\end{equation}

When $\sigma\leq p,\ |\alpha|=0$, plug $(\ref{Sect4_2_Estimates_4}),  (\ref{Sect4_2_Estimates_5_7}), (\ref{Sect4_2_Estimates_9})$ into
$(\ref{Sect4_2_Estimates_3})$, we get
\begin{equation}\label{Sect4_2_Estimates_11}
\begin{array}{ll}
\frac{\mathrm{d}}{\mathrm{d}t} \|\W_{\alpha,\sigma}\|_{L_{\ell+\sigma}^2}^2
+\frac{\mathrm{d}}{\mathrm{d}t} \int\limits_{\mathbb{R}} \frac{1}{\beta - \frac{\bar{u}_{yy}}{\bar{u}_y}|_{y=0}}
 (\W_{\alpha}|_{y=0})^2 \,\mathrm{d}x
+ \|\partial_y \W_{\alpha,\sigma}\|_{L_{\ell+\sigma}^2}^2 \\[9pt]

\leq   Q_7 \big(\sum\limits_{|\alpha^{\prime}|\leq |\alpha|,\sigma^{\prime}\leq \sigma}\|\W_{\alpha^{\prime},\sigma^{\prime}}\|_{L_{\ell+\sigma^{\prime}}^2}^2
+ \sum\limits_{|\alpha^{\prime}|\leq |\alpha|+1,\sigma^{\prime}\leq \sigma-1}\|\W_{\alpha^{\prime},\sigma^{\prime}}\|_{L_{\ell+\sigma^{\prime}}^2}^2
\\[12pt]\quad
+ \sum\limits_{|\alpha^{\prime}|\leq |\alpha|+1}\int\limits_{\mathbb{R}} \frac{1}{\beta - \frac{\bar{u}_{yy}}{\bar{u}_y}|_{y=0}} (\W_{\alpha}|_{y=0})^2 \,\mathrm{d}x \big)

+ q\sum\limits_{|\alpha^{\prime}| \leq |\alpha|+1, \sigma^{\prime}\leq \sigma-1}\|\partial_y \W_{\alpha^{\prime},\sigma^{\prime}}\|_{L_{\ell+\sigma}^2}^2.
\end{array}
\end{equation}

Thus, Lemma $\ref{Sect4_2_Stability_Estimate_Lemma}$ is proved.
\end{proof}

%%% find 5
\section{The Stability and Uniqueness of the Prandtl Equations}
In this section, we prove the stability and uniqueness of classical solutions to the Prandtl systems.

\begin{theorem}\label{Sect5_Unique_Stable_Thm}
Assume $(u^1,v^1)$ and $(u^2,v^2)$ be two classical solutions to the Robin boundary problem $(\ref{Sect1_PrandtlEq})$ with the initial vorticity $\omega_0^1$ and $\omega_0^2$ respectively, where $\omega_0^1$ and $\omega_0^2$ satisfy the conditions in Theorem $\ref{Sect1_Main_Thm}$, then
\begin{equation}\label{Sect5_Unique_Stable}
\begin{array}{ll}
\sum\limits_{|\alpha|+\sigma\leq p}\|\W_{\alpha,\sigma}\|_{L_{\ell+\sigma}^2}^2
+ \sum\limits_{|\alpha|\leq p}\int\limits_{\mathbb{R}} \frac{1}{\beta - \frac{\bar{u}_{yy}}{\bar{u}_y}|_{y=0}}
 (\W_{\alpha}|_{y=0})^2 \,\mathrm{d}x \\[10pt]

\leq C(\ve_1,T)\big[\|\omega^1_0-\omega^2_0\|_{\mathcal{H}_{\ell}^p(\mathbb{R}_{+}^2)}
+ \frac{1}{\beta -\delta_{\beta}} \big\|\omega_0^1|_{y=0} -\omega_0^2|_{y=0} \big\|_{H^p(\mathbb{R})}^2 \big],
\end{array}
\end{equation}
for all $p\leq k-1$.
\end{theorem}

\begin{proof}
Based on the a priori estimates $(\ref{Sect1_Estimates_4}), (\ref{Sect1_Estimates_5})$, we have
\begin{equation}\label{Sect5_Stability_1}
\begin{array}{ll}
\frac{\mathrm{d}}{\mathrm{d}t}\sum\limits_{|\alpha|+\sigma\leq p} \|\W_{\alpha,\sigma}\|_{L_{\ell+\sigma}^2}^2

+ \frac{\mathrm{d}}{\mathrm{d}t}\sum\limits_{|\alpha|\leq p} \int\limits_{\mathbb{R}} \frac{1}{\beta - \frac{\bar{u}_{yy}}{\bar{u}_y}|_{y=0}}
 (\W_{\alpha}|_{y=0})^2 \,\mathrm{d}x \\[12pt]

\leq C_{19}\sum\limits_{|\alpha|\leq k,\sigma\leq k-|\alpha|}Q_7(\alpha,\sigma) \cdot \Big(

\sum\limits_{|\alpha|+\sigma\leq p} \|\W_{\alpha,\sigma}\|_{L_{\ell+\sigma}^2}^2 \\[12pt]\quad

+ \sum\limits_{|\alpha|\leq p} \int\limits_{\mathbb{R}} \frac{1}{\beta - \frac{\bar{u}_{yy}}{\bar{u}_y}|_{y=0}}
 (\W_{\alpha}|_{y=0})^2 \,\mathrm{d}x \Big),
\end{array}
\end{equation}
where $s\geq 1$.

Integrate $(\ref{Sect5_Stability_1})$ from $t=0$ to $T$, we get
\begin{equation}\label{Sect5_Stability_2}
\begin{array}{ll}
\sum\limits_{|\alpha|+\sigma\leq p}\|\W_{\alpha,\sigma}\|_{L_{\ell+\sigma}^2}^2
+ \sum\limits_{|\alpha|\leq p}\int\limits_{\mathbb{R}} \frac{1}{\beta - \frac{\bar{u}_{yy}}{\bar{u}_y}|_{y=0}}
 (\W_{\alpha}|_{y=0})^2 \,\mathrm{d}x \\[10pt]

\leq \Big(\sum\limits_{|\alpha|+\sigma\leq p}\|\W_{\alpha,\sigma}|_{t=0}\|_{L_{\ell+\sigma}^2}^2
+ \sum\limits_{|\alpha|\leq p}\int\limits_{\mathbb{R}} \frac{1}{\beta - \frac{\bar{u}_{yy}}{\bar{u}_y}|_{t=0,y=0}}
 (\W_{\alpha}|_{t=0,y=0})^2 \,\mathrm{d}x \Big) \\[10pt]\quad

\cdot \exp\{ C_{19}\int\limits_0^t \sum\limits_{|\alpha|\leq k,\sigma\leq k-|\alpha|}Q_7(\alpha,\sigma) \mathrm{d}t \} \\[10pt]

\leq \Big(\sum\limits_{|\alpha|+\sigma\leq p}\|\W_{\alpha,\sigma}|_{t=0}\|_{L_{\ell+\sigma}^2}^2
+ \sum\limits_{|\alpha|\leq p}\int\limits_{\mathbb{R}} \frac{1}{\beta - \frac{\bar{u}_{yy}}{\bar{u}_y}|_{t=0,y=0}}
 (\W_{\alpha}|_{t=0,y=0})^2 \,\mathrm{d}x \Big) \\[10pt]\quad

\cdot \exp\Big\{C_{19} T\Big[\sum\limits_{|\alpha|+\sigma\leq k}\|W_{\alpha,\sigma}|_{t=0}\|_{L_{\ell+\sigma}^2(\mathbb{R}_{+}^2)}^2
+ \sum\limits_{|\alpha|\leq k}\int\limits_{\mathbb{R}} \frac{1}{\beta - \frac{u_{yy}}{u_y}}
 (W_{\alpha})^2|_{t=0,y=0} \,\mathrm{d}x \\[9pt]\quad

 + \|U\|_{H^{k+1}([0,T]\times\mathbb{R})}^2 \Big]^{\frac{s}{2}}
 \cdot M^\frac{s}{2} \Big\} \\[11pt]

\leq C(\ve_1,T) \Big(\sum\limits_{|\alpha|+\sigma\leq p}\|\W_{\alpha,\sigma}|_{t=0}\|_{L_{\ell+\sigma}^2}^2
+ \sum\limits_{|\alpha|\leq p}\int\limits_{\mathbb{R}} \frac{1}{\beta - \frac{\bar{u}_{yy}}{\bar{u}_y}|_{t=0,y=0}}
 (\W_{\alpha}|_{t=0,y=0})^2 \,\mathrm{d}x \Big) \\[12pt]

\leq C(\ve_1,T)\big[\|\omega^1_0-\omega^2_0\|_{\mathcal{H}_{\ell}^p(\mathbb{R}_{+}^2)}
+ \frac{1}{\beta -\delta_{\beta}} \|\omega_0^1|_{y=0} -\omega_0^2|_{y=0}\|_{H^p(\mathbb{R})}^2 \big].
\end{array}
\end{equation}
Thus, the stability of $(\ref{Sect1_PrandtlEq})$ has been proved.

Next, we prove the uniqueness of the Robin boundary problem $(\ref{Sect1_PrandtlEq})$.
Assume $u_0^1 = u_0^2$ in $\mathbb{R}_{+}^2$, then $u_0^1-u_0^2\equiv 0$ implies that $\omega_0^1- \omega_0^2\equiv 0$.
On the boundary $\{y=0\}$, $\omega_0^1|_{y=0}- \omega_0^2|_{y=0} = \beta(u_0^1|_{y=0}-u_0^2|_{y=0}) = 0$.
By the inequality $(\ref{Sect5_Unique_Stable})$, we have the uniqueness of the Robin boundary problem $(\ref{Sect1_PrandtlEq})$.

Thus, Theorem $\ref{Sect5_Unique_Stable_Thm}$ is proved.
\end{proof}

When $\beta=+\infty$, we have the following theorem:
\begin{theorem}\label{Sect5_Unique_Stable_Thm_Dirichlet}
Assume $(u^1,v^1)$ and $(u^2,v^2)$ be two classical solutions to the Dirichlet boundary problem $(\ref{Sect1_PrandtlEq_Dirichlet})$ with the initial vorticity $\omega_0^1$ and $\omega_0^2$ respectively, where $\omega_0^1$ and $\omega_0^2$ satisfy the conditions in Theorem $\ref{Sect1_Main_Thm_Dirichlet}$, then
\begin{equation}\label{Sect5_Unique_Stable_Dirichlet}
\begin{array}{ll}
\sum\limits_{|\alpha|+\sigma\leq p} \|\W_{\alpha,\sigma}\|_{L_{\ell+\sigma}^2([0,T]\times\mathbb{R}_{+}^2)}^2

\leq  C(\ve_2,T)\big[\|\omega^1_0-\omega^2_0\|_{\mathcal{H}_{\ell}^p(\mathbb{R}_{+}^2)},
\end{array}
\end{equation}
for all $p\leq k-1$.
\end{theorem}

\begin{proof}
Similar to $(\ref{Sect5_Stability_1})$,
\begin{equation}\label{Sect5_Stability_1_Dirichlet}
\begin{array}{ll}
\sum\limits_{|\alpha|+\sigma\leq p} \frac{\mathrm{d}}{\mathrm{d}t} \|\W_{\alpha,\sigma}\|_{L_{\ell+\sigma}^2}^2
\leq C_{20}\sum\limits_{|\alpha|\leq k,\sigma\leq k-|\alpha|}Q_7(\alpha,\sigma) \cdot

\sum\limits_{|\alpha|+\sigma\leq p} \|\W_{\alpha,\sigma}\|_{L_{\ell+\sigma}^2}^2,
\end{array}
\end{equation}
where $s\geq 1$.

Integrate $(\ref{Sect5_Stability_1_Dirichlet})$ from $t=0$ to $T$, we get
\begin{equation}\label{Sect5_Stability_2_Dirichlet}
\begin{array}{ll}
\sum\limits_{|\alpha|+\sigma\leq p}\|\W_{\alpha,\sigma}\|_{L_{\ell+\sigma}^2}^2 \\[9pt]

\leq \sum\limits_{|\alpha|+\sigma\leq p}\|\W_{\alpha,\sigma}|_{t=0}\|_{L_{\ell+\sigma}^2}^2
\cdot \exp\{ \int\limits_0^t C_{20}\sum\limits_{|\alpha|\leq k,\sigma\leq k-|\alpha|}Q_7(\alpha,\sigma) \mathrm{d}t \} \\[12pt]

\leq \sum\limits_{|\alpha|+\sigma\leq p}\|\W_{\alpha,\sigma}|_{t=0}\|_{L_{\ell+\sigma}^2}^2
\cdot \exp\Big\{C_{20} T\Big[\sum\limits_{|\alpha|+\sigma\leq k}\|W_{\alpha,\sigma}|_{t=0}\|_{L_{\ell+\sigma}^2(\mathbb{R}_{+}^2)}^2

\\[10pt]\quad
 + \|U\|_{H^{k+1}([0,T]\times\mathbb{R})}^2 \Big]^{\frac{s}{2}}
 \cdot M^{\frac{s}{2}} \Big\}

\leq C(\ve_2,T) \sum\limits_{|\alpha|+\sigma\leq p}\|\W_{\alpha,\sigma}|_{t=0}\|_{L_{\ell+\sigma}^2}^2 \\[13pt]

\leq C(\ve_2,T)\|\omega^1_0-\omega^2_0\|_{\mathcal{H}_{\ell}^p(\mathbb{R}_{+}^2)}^2.
\end{array}
\end{equation}

Thus, the stability of $(\ref{Sect1_PrandtlEq_Dirichlet})$ has been proved.

Next, we prove the uniqueness of the Dirichlet boundary problem $(\ref{Sect1_PrandtlEq_Dirichlet})$.

Assume $u_0^1 = u_0^2$ in $\mathbb{R}_{+}^2$, $u_0^1-u_0^2\equiv 0$ implies that $\omega_0^1- \omega_0^2\equiv 0$.
By the inequality $(\ref{Sect5_Unique_Stable_Dirichlet})$, we have the uniqueness of the Robin boundary problem $(\ref{Sect1_PrandtlEq})$.

Thus, Theorem $\ref{Sect5_Unique_Stable_Thm_Dirichlet}$ is proved.
\end{proof}

\appendix
%%% find 6
\section{Derivation of the Equations and their Boundary Conditions}

In this appendix, we derive the equations and boundary conditions.
The first three lemmas are used to prove the existence of the Prandtl equations.

\begin{lemma}\label{Appendix_1_Lemma_1}
If $u$ satisfies the Prandtl equations $(\ref{Sect1_PrandtlEq})$, then
$W_{\alpha} =u_y\partial_y\big(\frac{\partial_{t,x}^{\alpha} \tilde{u}}{u_y}\big)$ satisfies the system $(\ref{Sect1_Existence_VorticityEq_1})$.
If $\beta=+\infty$, $(\ref{Sect1_Existence_VorticityEq_1})_2$ is equivalent to $(\ref{Sect1_BC_Dirichlet_1})$.
\end{lemma}

\begin{proof}
We have the following transformation of the equations:
\begin{equation}\label{Appendix_1_Existence_VorticityEq_1}
\begin{array}{ll}
u_t + u u_x + v u_y + p_x= u_{yy}, \\[8pt]

u_t + u u_x + v u_y -U_t -U U_x= u_{yy}, \\[9pt]

\partial_y\frac{\partial_t \partial_{t,x}^{\alpha} u}{u_y} + u \partial_y\frac{\partial_x\partial_{t,x}^{\alpha} u}{u_y}
+ \partial_x\partial_{t,x}^{\alpha} u+ \partial_y\partial_{t,x}^{\alpha} v
- \partial_y\frac{\partial_t \partial_{t,x}^{\alpha} U + U \partial_x \partial_{t,x}^{\alpha} U}{u_y} \\[9pt]\quad

= \partial_y\frac{\partial_{yy}\partial_{t,x}^{\alpha} u}{u_y} - \partial_y\frac{[\partial_{t,x}^{\alpha},\, u\partial_x]u}{u_y}
- \partial_y\frac{[\partial_{t,x}^{\alpha},\, u_y]v}{u_y}
+ \partial_y\frac{[\partial_{t,x}^{\alpha},\, U\partial_x]U}{u_y}, \\[12pt]

\partial_y\frac{\partial_t \partial_{t,x}^{\alpha} \tilde{u}}{u_y} + u \partial_y\frac{\partial_x\partial_{t,x}^{\alpha} \tilde{u}}{u_y}
- \partial_y\frac{\partial_{yy}\partial_{t,x}^{\alpha} \tilde{u}}{u_y} =
- \partial_y\frac{[\partial_{t,x}^{\alpha},\, u_y]v}{u_y} \\[9pt]\quad
- \partial_y\frac{[\partial_{t,x}^{\alpha},\, u\partial_x]\tilde{u}}{u_y}
- \partial_y\frac{[\partial_{t,x}^{\alpha},\, \tilde{u}\partial_x]U}{u_y}
-\tilde{u} \partial_y\frac{\partial_x\partial_{t,x}^{\alpha} U}{u_y}, \\[12pt]

u_y\partial_y\frac{\partial_t \partial_{t,x}^{\alpha} \tilde{u}}{u_y} + u u_y\partial_y\frac{\partial_x\partial_{t,x}^{\alpha} \tilde{u}}{u_y}
- u_y\partial_y\frac{\partial_{yy}\partial_{t,x}^{\alpha} \tilde{u}}{u_y} =
- u_y\partial_y\frac{[\partial_{t,x}^{\alpha},\, u_y]v}{u_y} \\[9pt]\quad
- u_y\partial_y\frac{[\partial_{t,x}^{\alpha},\, u\partial_x]\tilde{u}}{u_y}
- u_y\partial_y\frac{[\partial_{t,x}^{\alpha},\, \tilde{u}\partial_x]U}{u_y}
-\tilde{u} u_y\partial_y\frac{\partial_x\partial_{t,x}^{\alpha} U}{u_y},
\end{array}
\end{equation}

We calculate the following three terms, the first term is
\begin{equation}\label{Appendix_1_Existence_VorticityEq_2_1}
\begin{array}{ll}
u_y\partial_y\frac{\partial_{yy}\partial_{t,x}^{\alpha} \tilde{u}}{u_y}
= u_y\partial_y[\partial_y\frac{\partial_y\partial_{t,x}^{\alpha} \tilde{u}}{u_y}
+ \frac{\partial_y\partial_{t,x}^{\alpha} \tilde{u}}{u_y}\frac{u_{yy}}{u_y}] \\[9pt]

= u_y\partial_y[\partial_y(\partial_y\frac{\partial_{t,x}^{\alpha} \tilde{u}}{u_y} + \frac{\partial_{t,x}^{\alpha} \tilde{u}}{u_y}\frac{u_{yy}}{u_y})
+ (\partial_y\frac{\partial_{t,x}^{\alpha} \tilde{u}}{u_y} + \frac{\partial_{t,x}^{\alpha} \tilde{u}}{u_y}\frac{u_{yy}}{u_y})\frac{u_{yy}}{u_y}] \\[9pt]

= u_y\partial_y[\frac{1}{u_y}\partial_y W_{\alpha}
+ \frac{\partial_{t,x}^{\alpha} \tilde{u}}{u_y}\partial_y(\frac{u_{yy}}{u_y})]
+ u_y\partial_y[\frac{u_{yy}}{(u_y)^2}W_{\alpha} + \frac{\partial_{t,x}^{\alpha} \tilde{u}}{u_y}(\frac{u_{yy}}{u_y})^2] \\[9pt]

= u_y\partial_y[\frac{1}{u_y}\partial_y W_{\alpha} + \frac{u_{yy}}{(u_y)^2}W_{\alpha}]
+ u_y\partial_y[\frac{\partial_{t,x}^{\alpha} \tilde{u}}{u_y}\partial_y(\frac{u_{yy}}{u_y})
+ \frac{\partial_{t,x}^{\alpha} \tilde{u}}{u_y}(\frac{u_{yy}}{u_y})^2] \\[9pt]

= \partial_{yy} W_{\alpha} + (\frac{u_{yyy}}{u_y}-2(\frac{u_{yy}}{u_y})^2)W_{\alpha}
+ u_y\partial_y(\frac{\partial_{t,x}^{\alpha} \tilde{u}}{u_y}\frac{u_{yyy}}{u_y}).
\end{array}
\end{equation}

The second term is
\begin{equation}\label{Appendix_1_Existence_VorticityEq_2_2}
\begin{array}{ll}
u_y\partial_y\frac{\partial_t \partial_{t,x}^{\alpha} \tilde{u}}{u_y}
= u_y\partial_y(\partial_t\frac{\partial_{t,x}^{\alpha} \tilde{u}}{u_y}
+ \frac{\partial_{t,x}^{\alpha} \tilde{u}}{u_y}\frac{u_{yt}}{u_y})

=  u_y\partial_t(\partial_y\frac{\partial_{t,x}^{\alpha} \tilde{u}}{u_y})
+ u_y\partial_y(\frac{\partial_{t,x}^{\alpha} \tilde{u}}{u_y}\frac{u_{yt}}{u_y}) \\[9pt]

=  u_y\partial_t(\frac{1}{u_y}W_{\alpha})
+ u_y\partial_y(\frac{\partial_{t,x}^{\alpha} \tilde{u}}{u_y}\frac{u_{yt}}{u_y})

=  \partial_t W_{\alpha} - \frac{u_{yt}}{u_y}W_{\alpha}
+ u_y\partial_y(\frac{\partial_{t,x}^{\alpha} \tilde{u}}{u_y}\frac{u_{yt}}{u_y}).
\end{array}
\end{equation}

The third term is
\begin{equation}\label{Appendix_1_Existence_VorticityEq_2_3}
\begin{array}{ll}
u_y\partial_y\frac{\partial_x \partial_{t,x}^{\alpha} \tilde{u}}{u_y}
= u_y\partial_y(\partial_x\frac{\partial_{t,x}^{\alpha} \tilde{u}}{u_y}
+ \frac{\partial_{t,x}^{\alpha} \tilde{u}}{u_y}\frac{u_{yx}}{u_y})

=  u_y\partial_x(\partial_y\frac{\partial_{t,x}^{\alpha} \tilde{u}}{u_y})
+ u_y\partial_y(\frac{\partial_{t,x}^{\alpha} \tilde{u}}{u_y}\frac{u_{yx}}{u_y}) \\[9pt]

=  u_y\partial_x(\frac{1}{u_y}W_{\alpha})
+ u_y\partial_y(\frac{\partial_{t,x}^{\alpha} \tilde{u}}{u_y}\frac{u_{yx}}{u_y})

=  \partial_x W_{\alpha} - \frac{u_{yx}}{u_y}W_{\alpha}
+ u_y\partial_y(\frac{\partial_{t,x}^{\alpha} \tilde{u}}{u_y}\frac{u_{yx}}{u_y}).
\end{array}
\end{equation}

Plug $(\ref{Appendix_1_Existence_VorticityEq_2_1})$, $(\ref{Appendix_1_Existence_VorticityEq_2_2})$, $(\ref{Appendix_1_Existence_VorticityEq_2_3})$ into the last equation of $(\ref{Appendix_1_Existence_VorticityEq_1})$, we get
\begin{equation}\label{Appendix_1_Existence_VorticityEq_3}
\begin{array}{ll}
\partial_t W_{\alpha} + u\partial_x W_{\alpha} - \partial_{yy} W_{\alpha}
- (\frac{u_{yt}}{u_y}+ \frac{u u_{yx}}{u_y} + \frac{u_{yyy}}{u_y}-2(\frac{u_{yy}}{u_y})^2) W_{\alpha} \\[8pt]\quad
+ u_y\partial_y(\frac{\partial_{t,x}^{\alpha} \tilde{u}}{u_y}\frac{u_{yt}}{u_y})
+ u u_y\partial_y(\frac{\partial_{t,x}^{\alpha} \tilde{u}}{u_y}\frac{u_{yx}}{u_y})
- u_y\partial_y(\frac{\partial_{t,x}^{\alpha} \tilde{u}}{u_y}\frac{u_{yyy}}{u_y})\\[10pt]

= - u_y\partial_y\frac{[\partial_{t,x}^{\alpha},\, u_y]v}{u_y}
- u_y\partial_y\frac{[\partial_{t,x}^{\alpha},\, u\partial_x]\tilde{u}}{u_y}
- u_y\partial_y\frac{[\partial_{t,x}^{\alpha},\, \tilde{u}\partial_x]U}{u_y}
-\tilde{u} u_y\partial_y\frac{\partial_x\partial_{t,x}^{\alpha} U}{u_y},
\end{array}
\end{equation}

Next, we derive the boundary condition for $W_{\alpha}$ by using the following equations on the boundary, where $v|_{y=0}=0$. We can use the following equation on the boundary:
\begin{equation}\label{Appendix_1_Existence_BC_1}
\left\{\begin{array}{ll}
u_t + u u_x -U_t -UU_x= u_{yy}, \\[7pt]
(u_y -\beta u)|_{y=0} =0,
\end{array}\right.
\end{equation}

Apply the tangential differential operator $\partial_{t,x}^{\alpha}$ to $(\ref{Appendix_1_Existence_BC_1})_1$, then we get
\begin{equation}\label{Appendix_1_Existence_BC_2}
\begin{array}{ll}
\partial_t\partial_{t,x}^{\alpha} u + u \partial_x\partial_{t,x}^{\alpha} u -\partial_t\partial_{t,x}^{\alpha}U -U\partial_x\partial_{t,x}^{\alpha}U - \partial_{yy}\partial_{t,x}^{\alpha} u \\[8pt]\quad
= -[\partial_{t,x}^{\alpha},\,u\partial_x]u + [\partial_{t,x}^{\alpha},\,U\partial_x]U, \\[13pt]

\partial_t\partial_{t,x}^{\alpha} \tilde{u} + u \partial_x\partial_{t,x}^{\alpha} \tilde{u} - \partial_{yy}\partial_{t,x}^{\alpha} \tilde{u}
= -\tilde{u}\partial_x\partial_{t,x}^{\alpha}U
-[\partial_{t,x}^{\alpha},\,u\partial_x]\tilde{u} -[\partial_{t,x}^{\alpha},\,\tilde{u}\partial_x]U,
\end{array}
\end{equation}

Apply the tangential differential operator $\partial_{t,x}^{\alpha}$ to $(\ref{Appendix_1_Existence_BC_1})_2$, then we get
\begin{equation}\label{Appendix_1_Existence_BC_3}
\begin{array}{ll}
\partial_{t,x}^{\alpha} \tilde{u}= \frac{1}{\beta}\partial_y\partial_{t,x}^{\alpha} \tilde{u} -\partial_{t,x}^{\alpha} U.
\end{array}
\end{equation}

On the boundary, we calculate $\frac{\partial_{t,x}^{\alpha}\tilde{u}}{u_y}$:
\begin{equation}\label{Appendix_1_Existence_BC_4}
\begin{array}{ll}
W_{\alpha} = u_y\partial_y \frac{\partial_{t,x}^{\alpha}\tilde{u}}{u_y}
= \partial_y\partial_{t,x}^{\alpha}\tilde{u}- \partial_{t,x}^{\alpha}\tilde{u} \frac{u_{yy}}{u_y}
= \partial_{t,x}^{\alpha}\tilde{u}(\beta - \frac{u_{yy}}{u_y})
+ \beta \partial_{t,x}^{\alpha} U ,\\[10pt]

\partial_{t,x}^{\alpha}\tilde{u} =\frac{1}{\beta - \frac{u_{yy}}{u_y}}(W_{\alpha} - \beta \partial_{t,x}^{\alpha} U) ,
\end{array}
\end{equation}

While
\begin{equation}\label{Appendix_1_Existence_BC_5}
\begin{array}{ll}
\partial_{yy} \partial_{t,x}^{\alpha}\tilde{u} = \partial_y[\partial_y(\frac{\partial_{t,x}^{\alpha}\tilde{u}}{u_y}u_y)]
= \partial_y[\partial_y(\frac{\partial_{t,x}^{\alpha}\tilde{u}}{u_y})u_y
+ \frac{\partial_{t,x}^{\alpha}\tilde{u}}{u_y}u_{yy}] \\[9pt]

= \partial_y W_{\alpha} + \frac{u_{yy}}{u_y}W_{\alpha} + \partial_{t,x}^{\alpha}\tilde{u}\frac{u_{yyy}}{u_y} \\[9pt]

= \partial_y W_{\alpha} + \frac{u_{yy}}{u_y}W_{\alpha} + \frac{u_{yyy}}{u_y}\cdot
\frac{1}{\beta - \frac{u_{yy}}{u_y}}W_{\alpha}
- \frac{u_{yyy}}{u_y}\cdot
\frac{\beta}{\beta - \frac{u_{yy}}{u_y}} \partial_{t,x}^{\alpha} U,
\end{array}
\end{equation}

By using $(\ref{Appendix_1_Existence_BC_3})$, we calculate $\partial_t\partial_{t,x}^{\alpha} \tilde{u}$:
\begin{equation}\label{Appendix_1_Existence_BC_6}
\begin{array}{ll}
\partial_t\partial_{t,x}^{\alpha} \tilde{u}

= \partial_t\Big[\frac{1}{\beta - \frac{u_{yy}}{u_y}}(W_{\alpha} - \beta \partial_{t,x}^{\alpha} U)\Big] \\[10pt]

= \frac{1}{\beta - \frac{u_{yy}}{u_y}}\partial_t W_{\alpha}
+ \frac{(\frac{u_{yy}}{u_y})_t}{(\beta - \frac{u_{yy}}{u_y})^2}W_{\alpha}
- \partial_t\Big[\frac{\beta}{\beta - \frac{u_{yy}}{u_y}} \partial_{t,x}^{\alpha} U\Big].
\end{array}
\end{equation}

Similarly, we calculate $\frac{\partial_x\partial_{t,x}^{\alpha} \tilde{u}}{u_y}$:
\begin{equation}\label{Appendix_1_Existence_BC_7}
\begin{array}{ll}
\partial_x\partial_{t,x}^{\alpha} \tilde{u}

= \partial_x\Big[\frac{1}{\beta - \frac{u_{yy}}{u_y}}(W_{\alpha} - \beta \partial_{t,x}^{\alpha} U)\Big] \\[10pt]

= \frac{1}{\beta - \frac{u_{yy}}{u_y}}\partial_x W_{\alpha}
+ \frac{(\frac{u_{yy}}{u_y})_x}{(\beta - \frac{u_{yy}}{u_y})^2}W_{\alpha}
- \partial_x\Big[\frac{\beta}{\beta - \frac{u_{yy}}{u_y}} \partial_{t,x}^{\alpha} U\Big].
\end{array}
\end{equation}

It follows from the last equation in $(\ref{Appendix_1_Existence_BC_2})$, we have
\begin{equation}\label{Appendix_1_Existence_BC_8}
\begin{array}{ll}
\frac{1}{\beta - \frac{u_{yy}}{u_y}}\partial_t W_{\alpha}
+ \frac{(\frac{u_{yy}}{u_y})_t}{(\beta - \frac{u_{yy}}{u_y})^2}W_{\alpha}
- \partial_t\Big[\frac{\beta}{\beta - \frac{u_{yy}}{u_y}} \partial_{t,x}^{\alpha} U\Big] \\[10pt]\quad

+ \frac{u}{\beta - \frac{u_{yy}}{u_y}} \partial_x W_{\alpha}
+ \frac{u(\frac{u_{yy}}{u_y})_x}{(\beta - \frac{u_{yy}}{u_y})^2}W_{\alpha}
- u\partial_x\Big[\frac{\beta}{\beta - \frac{u_{yy}}{u_y}} \partial_{t,x}^{\alpha} U\Big] \\[10pt]\quad

-\partial_y W_{\alpha} -\frac{u_{yy}}{u_y}W_{\alpha}
- \frac{u_{yyy}}{u_y}\cdot \frac{1}{\beta - \frac{u_{yy}}{u_y}}W_{\alpha}
+ \frac{u_{yyy}}{u_y}\cdot \frac{\beta}{\beta - \frac{u_{yy}}{u_y}} \partial_{t,x}^{\alpha} U

\\[10pt]
= -\tilde{u}\partial_x\partial_{t,x}^{\alpha}U
-[\partial_{t,x}^{\alpha},\,u\partial_x]\tilde{u} -[\partial_{t,x}^{\alpha},\,\tilde{u}\partial_x]U,

\\[15pt]
\frac{1}{\beta - \frac{u_{yy}}{u_y}}(\partial_t W_{\alpha} + u\partial_x W_{\alpha})
- \partial_y W_{\alpha} - \frac{u_{yy}}{u_y}W_{\alpha}
+ Q_2 \cdot\frac{1}{\beta - \frac{u_{yy}}{u_y}}W_{\alpha}

\\[10pt]
= -[\partial_{t,x}^{\alpha},\,u\partial_x]\tilde{u} -[\partial_{t,x}^{\alpha},\,\tilde{u}\partial_x]U
+ Q_3.
\end{array}
\end{equation}
where $Q_2,Q_3$ are defined in $(\ref{Sect1_Quantity_Definition_1})$.

When $\beta =+\infty$, we have the Dirichlet boundary condition, then we derive the boundary condition for the vorticity $W_{\alpha}$ by using the following equations on the boundary:
\begin{equation}\label{Appendix_1_Existence_DirichletBC_1}
\left\{\begin{array}{ll}
u_{yy} = p_x =-U_t -UU_x, \\[8pt]
u|_{y=0} =0,
\end{array}\right.
\end{equation}

Apply the tangential differential operator $\partial_{t,x}^{\alpha}$ to the two equations in $(\ref{Appendix_1_Existence_DirichletBC_1})$, then we get
\begin{equation*}
\begin{array}{ll}
-\partial_t \partial_{t,x}^{\alpha} U -U\partial_x \partial_{t,x}^{\alpha} U
- [\partial_{t,x}^{\alpha},\, U\partial_x]U

= \partial_y W_{\alpha} + \frac{u_{yy}}{u_y}W_{\alpha}
- \frac{u_{yyy}}{u_y} \partial_{t,x}^{\alpha} U,
\end{array}
\end{equation*}

Thus, we get the boundary condition for $W_{\alpha}$ when $\beta=+\infty$:
\begin{equation}\label{Appendix_1_Existence_DirichletBC_4}
\begin{array}{ll}
-\partial_y W_{\alpha} - \frac{u_{yy}}{u_y}W_{\alpha}
= \partial_t \partial_{t,x}^{\alpha} U +U\partial_x \partial_{t,x}^{\alpha} U
+ [\partial_{t,x}^{\alpha},\, U\partial_x]U - \frac{u_{yyy}}{u_y} \partial_{t,x}^{\alpha} U.
\end{array}
\end{equation}

Note that $\lim\limits_{\beta\rto +\infty}Q_3 = \partial_t\partial_{t,x}^{\alpha}U + U\partial_x\partial_{t,x}^{\alpha}U
- \frac{u_{yyy}}{u_y}\partial_{t,x}^{\alpha} U$. When $\beta=+\infty$, $(\ref{Sect1_PrandtlEq})_2$ is equivalent to $(\ref{Sect1_BC_Dirichlet_1})$. Thus, Lemma $\ref{Appendix_1_Lemma_1}$ is proved.
\end{proof}

When $\sigma\geq 1$, we derive the equations for $W_{\alpha,\sigma} =\partial_y\big(\frac{\partial_{t,x}^{\alpha}\partial_y^{\sigma}u}{\partial_y u}\big)$.
\begin{lemma}\label{Appendix_2_Lemma_1}
If $u$ satisfies the Prandtl equations $(\ref{Sect1_PrandtlEq})$, $\sigma\geq 1$, then
$W_{\alpha,\sigma}$ satisfies the system $(\ref{Sect1_Existence_VorticityEq_2})$.
\end{lemma}

\begin{proof}
We have the following transformation of the equations:
\begin{equation}\label{Appendix_2_VorticityEq_1}
\begin{array}{ll}
u_t + u u_x + v u_y + p_x= u_{yy}, \\[8pt]

\frac{\partial_t \partial_{t,x}^{\alpha}\partial_y^{\sigma} u}{u_y} + \frac{u \partial_x\partial_{t,x}^{\alpha}\partial_y^{\sigma} u}{u_y}
+ \partial_{t,x}^{\alpha}\partial_y^{\sigma} v
= \frac{\partial_{yy}\partial_{t,x}^{\alpha}\partial_y^{\sigma} u}{u_y} - \frac{[\partial_{t,x}^{\alpha}\partial_y^{\sigma},\, u\partial_x]u}{u_y}
- \frac{[\partial_{t,x}^{\alpha}\partial_y^{\sigma},\, u_y]v}{u_y}, \\[10pt]

\partial_y\frac{\partial_t \partial_{t,x}^{\alpha}\partial_y^{\sigma} u}{u_y} + u \partial_y\frac{\partial_x\partial_{t,x}^{\alpha}\partial_y^{\sigma} u}{u_y}
+ \partial_x\partial_{t,x}^{\alpha}\partial_y^{\sigma} u+ \partial_y\partial_{t,x}^{\alpha}\partial_y^{\sigma} v  \\[7pt]\quad
= \partial_y\frac{\partial_{yy}\partial_{t,x}^{\alpha}\partial_y^{\sigma} u}{u_y} - \partial_y\frac{[\partial_{t,x}^{\alpha}\partial_y^{\sigma},\, u\partial_x]u}{u_y}
- \partial_y\frac{[\partial_{t,x}^{\alpha}\partial_y^{\sigma},\, u_y]v}{u_y}, \\[11pt]

u_y\partial_y\frac{\partial_t \partial_{t,x}^{\alpha}\partial_y^{\sigma} u}{u_y}
+ u u_y\partial_y\frac{\partial_x\partial_{t,x}^{\alpha}\partial_y^{\sigma} u}{u_y}
- u_y\partial_y\frac{\partial_{yy}\partial_{t,x}^{\alpha}\partial_y^{\sigma} u}{u_y}  \\[7pt]\quad
=  - u_y\partial_y\frac{[\partial_{t,x}^{\alpha}\partial_y^{\sigma},\, u\partial_x]u}{u_y}
- u_y\partial_y\frac{[\partial_{t,x}^{\alpha}\partial_y^{\sigma},\, u_y]v}{u_y},
\end{array}
\end{equation}

We calculate the following three terms:

Similar to $(\ref{Appendix_1_Existence_VorticityEq_2_1})$, we have
\begin{equation}\label{Appendix_2_VorticityEq_2}
\begin{array}{ll}
u_y\partial_y\frac{\partial_{yy}\partial_{t,x}^{\alpha}\partial_y^{\sigma} u}{u_y}
= u_y\partial_y\big[\partial_y\frac{\partial_y \partial_{t,x}^{\alpha}\partial_y^{\sigma} u}{u_y} + \frac{\partial_y \partial_{t,x}^{\alpha}\partial_y^{\sigma} u}{u_y}\frac{u_{yy}}{u_y}\big] \\[10pt]

= u_y\partial_y\big[\partial_y(\partial_y\frac{\partial_{t,x}^{\alpha}\partial_y^{\sigma} u}{u_y}
+ \frac{\partial_{t,x}^{\alpha}\partial_y^{\sigma} u}{u_y}\frac{u_{yy}}{u_y})
+ (\partial_y\frac{\partial_{t,x}^{\alpha}\partial_y^{\sigma} u}{u_y}
+ \frac{\partial_{t,x}^{\alpha}\partial_y^{\sigma} u}{u_y}\frac{u_{yy}}{u_y})
\frac{u_{yy}}{u_y}\big] \\[10pt]

= u_y\partial_y\big[\frac{1}{u_y}\partial_y W_{\alpha,\sigma}
+ \frac{\partial_{t,x}^{\alpha}\partial_y^{\sigma} u}{u_y}\partial_y(\frac{u_{yy}}{u_y})]
+ u_y\partial_y\big[\frac{u_{yy}}{(u_y)^2}W_{\alpha,\sigma}
+ \frac{\partial_{t,x}^{\alpha}\partial_y^{\sigma} u}{u_y}(\frac{u_{yy}}{u_y})^2\big] \\[10pt]

= u_y\partial_y\big[\frac{1}{u_y}\partial_y W_{\alpha,\sigma}
+ \frac{u_{yy}}{(u_y)^2}W_{\alpha,\sigma}]
+ u_y\partial_y\big[\frac{\partial_{t,x}^{\alpha}\partial_y^{\sigma} u}{u_y}\partial_y(\frac{u_{yy}}{u_y})
+ \frac{\partial_{t,x}^{\alpha}\partial_y^{\sigma} u}{u_y}(\frac{u_{yy}}{u_y})^2\big] \\[10pt]

= \partial_{yy} W_{\alpha,\sigma}
+ (\frac{u_{yyy}}{u_y} - 2(\frac{u_{yy}}{u_y})^2)W_{\alpha,\sigma}
+ u_y\partial_y\big[\frac{\partial_{t,x}^{\alpha}\partial_y^{\sigma} u}{u_y} \frac{u_{yyy}}{u_y} \big],
\end{array}
\end{equation}

Similar to $(\ref{Appendix_1_Existence_VorticityEq_2_2})$, we have
\begin{equation}\label{Appendix_2_VorticityEq_3}
\begin{array}{ll}
u_y\partial_y\frac{\partial_t \partial_{t,x}^{\alpha}\partial_y^{\sigma} u}{u_y}
= u_y\partial_y[\partial_t(\frac{\partial_{t,x}^{\alpha}\partial_y^{\sigma} u}{u_y})]
+ u_y\partial_y[\frac{\partial_{t,x}^{\alpha}\partial_y^{\sigma} u}{u_y}\frac{u_{yt}}{u_y}] \\[10pt]

= u_y\partial_t[\frac{1}{u_y}W_{\alpha,\sigma}]
+ u_y\partial_y[\frac{\partial_{t,x}^{\alpha}\partial_y^{\sigma} u}{u_y}\frac{u_{yt}}{u_y}]

=\partial_t W_{\alpha,\sigma} -\frac{u_{yt}}{u_y}W_{\alpha,\sigma}
+ u_y\partial_y[\frac{\partial_{t,x}^{\alpha}\partial_y^{\sigma} u}{u_y}\frac{u_{yt}}{u_y}],
\end{array}
\end{equation}

Similar to $(\ref{Appendix_1_Existence_VorticityEq_2_3})$, we have
\begin{equation}\label{Appendix_2_VorticityEq_4}
\begin{array}{ll}
u_y\partial_y\frac{\partial_x \partial_{t,x}^{\alpha}\partial_y^{\sigma} u}{u_y}
= u_y\partial_y[\partial_x(\frac{\partial_{t,x}^{\alpha}\partial_y^{\sigma} u}{u_y})]
+ u_y\partial_y[\frac{\partial_{t,x}^{\alpha}\partial_y^{\sigma} u}{u_y}\frac{u_{yx}}{u_y}] \\[10pt]

= u_y\partial_x[\frac{1}{u_y}W_{\alpha,\sigma}]
+ u_y\partial_y[\frac{\partial_{t,x}^{\alpha}\partial_y^{\sigma} u}{u_y}\frac{u_{yx}}{u_y}]

=\partial_x W_{\alpha,\sigma} -\frac{u_{yx}}{u_y}W_{\alpha,\sigma}
+ u_y\partial_y[\frac{\partial_{t,x}^{\alpha}\partial_y^{\sigma} u}{u_y}\frac{u_{yx}}{u_y}],
\end{array}
\end{equation}

Plug $(\ref{Appendix_2_VorticityEq_2})$, $(\ref{Appendix_2_VorticityEq_2})$, $(\ref{Appendix_2_VorticityEq_2})$ into the last equation of $(\ref{Appendix_2_VorticityEq_1})$
\begin{equation}\label{Appendix_2_VorticityEq_5}
\begin{array}{ll}
\partial_t W_{\alpha,\sigma} + u\partial_x W_{\alpha,\sigma}
- \partial_{yy} W_{\alpha,\sigma} +Q_1 W_{\alpha,\sigma} \\[9pt]\quad

+ u_y\partial_y[\frac{\partial_{t,x}^{\alpha}\partial_y^{\sigma} u}{u_y}\frac{u_{yt}}{u_y}]
+ u u_y\partial_y[\frac{\partial_{t,x}^{\alpha}\partial_y^{\sigma} u}{u_y}\frac{u_{yx}}{u_y}]
- u_y\partial_y\big[\frac{\partial_{t,x}^{\alpha}\partial_y^{\sigma} u}{u_y} \frac{u_{yyy}}{u_y} \big] \\[12pt]

=  - u_y\partial_y\frac{[\partial_{t,x}^{\alpha}\partial_y^{\sigma},\, u\partial_x]u}{u_y}
- u_y\partial_y\frac{[\partial_{t,x}^{\alpha}\partial_y^{\sigma},\, u_y]v}{u_y}.
\end{array}
\end{equation}

Next we derive the boundary condition for $W_{\alpha,\sigma}$, where $\partial_{t,x}^{\alpha^{\prime}}v|_{y=0} =0$ for some $\alpha^{\prime}$.
Note that $\partial_{t,x}^{\alpha}\partial_y^{\sigma}(u_y -\beta u)|_{y=0}$ may not be zero when $\sigma\geq 1$.
\begin{equation}\label{Appendix_2_Sigma1_Existence_BC_1}
\begin{array}{ll}
\partial_{yy} \partial_{t,x}^{\alpha}\partial_y^{\sigma}u = \partial_y[\partial_y(\frac{\partial_{t,x}^{\alpha}\partial_y^{\sigma}u}{u_y}u_y)]
= \partial_y[\partial_y(\frac{\partial_{t,x}^{\alpha}\partial_y^{\sigma}u}{u_y})u_y
+ \frac{\partial_{t,x}^{\alpha}\partial_y^{\sigma}u}{u_y}u_{yy}] \\[9pt]

= \partial_y W_{\alpha,\sigma} + \frac{u_{yy}}{u_y}W_{\alpha,\sigma} + \partial_{t,x}^{\alpha}\partial_y^{\sigma}u\frac{u_{yyy}}{u_y},
\end{array}
\end{equation}

When $\sigma\geq 1$, we have the following transformation of the equations:
\begin{equation}\label{Appendix_2_Sigma1_Existence_BC_2}
\begin{array}{ll}
u_t + u u_x + v u_y + p_x= u_{yy}, \\[10pt]

\partial_{yy}\partial_{t,x}^{\alpha}\partial_y^{\sigma} u
= \partial_t \partial_{t,x}^{\alpha}\partial_y^{\sigma} u + u \partial_x\partial_{t,x}^{\alpha}\partial_y^{\sigma} u
+ [\partial_{t,x}^{\alpha}\partial_y^{\sigma},\, u\partial_x]u
+ \partial_{t,x}^{\alpha}\partial_y^{\sigma} (v u_y), \\[10pt]

\partial_y W_{\alpha,\sigma} + \frac{u_{yy}}{u_y}W_{\alpha,\sigma}

= \partial_t \partial_{t,x}^{\alpha}\partial_y^{\sigma} u + u \partial_x\partial_{t,x}^{\alpha}\partial_y^{\sigma} u
+ [\partial_{t,x}^{\alpha}\partial_y^{\sigma},\, u\partial_x]u \\[8pt]\qquad
+ \partial_{t,x}^{\alpha} (v \partial_y^{\sigma+1} u)
+ \sum\limits_{\sigma_1\geq 1,\sigma_2\leq \sigma-1}\partial_{t,x}^{\alpha} (\partial_y^{\sigma_1} v \partial_y^{\sigma_2+1} u)
- \partial_{t,x}^{\alpha}\partial_y^{\sigma}u\frac{u_{yyy}}{u_y}, \\[12pt]\quad

= \partial_t \partial_{t,x}^{\alpha}\partial_y^{\sigma} u + u \partial_x\partial_{t,x}^{\alpha}\partial_y^{\sigma} u
+ [\partial_{t,x}^{\alpha}\partial_y^{\sigma},\, u\partial_x]u \\[8pt]\qquad
- \sum\limits_{\sigma_1\geq 1,\sigma_2\leq \sigma-1}\partial_{t,x}^{\alpha} (\partial_y^{\sigma_1-1} \partial_x u \partial_y^{\sigma_2+1} u)
- \partial_{t,x}^{\alpha}\partial_y^{\sigma}u\frac{u_{yyy}}{u_y}
\\[12pt]\quad

= \partial_t \partial_{t,x}^{\alpha}\partial_y^{\sigma} u + u \partial_x\partial_{t,x}^{\alpha}\partial_y^{\sigma} u
- \partial_{t,x}^{\alpha}\partial_y^{\sigma}u\frac{u_{yyy}}{u_y} \\[8pt]\qquad
+ \sum\limits_{|\alpha_1|+\sigma_1>0}(\partial_{t,x}^{\alpha_1}\partial_y^{\sigma_1}u) (\partial_{t,x}^{\alpha_2}\partial_y^{\sigma_2}\partial_x u)
- \sum\limits_{|\alpha_1|<|\alpha|}(\partial_{t,x}^{\alpha_1} \partial_y^{\sigma_1} \partial_x u)
(\partial_{t,x}^{\alpha_2}\partial_y^{\sigma_2} u).
\end{array}
\end{equation}

By the induction method, we have
\begin{equation}\label{Appendix_2_Sigma1_Existence_BC_3}
\begin{array}{ll}
\partial_t \partial_{t,x}^{\alpha}\partial_y^{\sigma} u
= \partial_y (\frac{\partial_t \partial_{t,x}^{\alpha}\partial_y^{\sigma-1} u}{u_y}u_y)
= W_{\alpha+(1,0),\sigma-1} + \partial_t \partial_{t,x}^{\alpha}\partial_y^{\sigma-1} u\frac{u_{yy}}{u_y} \\[8pt]

= W_{\alpha+(1,0),\sigma-1} + W_{\alpha+(1,0),\sigma-2} \frac{u_{yy}}{u_y}
 + \partial_t \partial_{t,x}^{\alpha}\partial_y^{\sigma-2} u(\frac{u_{yy}}{u_y})^2 = \cdots \\[8pt]

= \sum\limits_{m=0}^{\sigma-1} W_{\alpha+(1,0),\sigma-1-m} (\frac{u_{yy}}{u_y})^m
 + \partial_t \partial_{t,x}^{\alpha} \tilde{u}(\frac{u_{yy}}{u_y})^{\sigma}.
\end{array}
\end{equation}

Similarly,
\begin{equation}\label{Appendix_2_Sigma1_Existence_BC_4}
\begin{array}{ll}
\partial_x \partial_{t,x}^{\alpha}\partial_y^{\sigma} u

= \sum\limits_{m=0}^{\sigma-1} W_{\alpha+(0,1),\sigma-1-m} (\frac{u_{yy}}{u_y})^m
 + \partial_x \partial_{t,x}^{\alpha} \tilde{u}(\frac{u_{yy}}{u_y})^{\sigma}.
\end{array}
\end{equation}

Plug $(\ref{Appendix_2_Sigma1_Existence_BC_3})$, $(\ref{Appendix_2_Sigma1_Existence_BC_4})$ into $(\ref{Appendix_2_Sigma1_Existence_BC_2})$, we have
\begin{equation}\label{Appendix_2_Sigma1_Existence_BC_5}
\begin{array}{ll}
\partial_y W_{\alpha,\sigma} + \frac{u_{yy}}{u_y}W_{\alpha,\sigma}

= \sum\limits_{m=0}^{\sigma-1} (W_{\alpha+(1,0),\sigma-1-m} + u W_{\alpha+(0,1),\sigma-1-m}) (\frac{u_{yy}}{u_y})^m \\[8pt]\quad
 + \partial_t \partial_{t,x}^{\alpha} \tilde{u}(\frac{u_{yy}}{u_y})^{\sigma}
 + u\partial_x \partial_{t,x}^{\alpha} \tilde{u}(\frac{u_{yy}}{u_y})^{\sigma}

 - \sum\limits_{m=0}^{\sigma-1} W_{\alpha,\sigma-1-m} (\frac{u_{yy}}{u_y})^m \frac{u_{yyy}}{u_y}\\[8pt]\quad
 - \partial_{t,x}^{\alpha} \tilde{u}(\frac{u_{yy}}{u_y})^{\sigma} \frac{u_{yyy}}{u_y}

 + \sum\limits_{|\alpha_1|+\sigma_1>0} \Big(\sum\limits_{m=0}^{\sigma_1-1} W_{\alpha_1,\sigma_1-1-m} (\frac{u_{yy}}{u_y})^m
 + \partial_x \partial_{t,x}^{\alpha_1} \tilde{u}(\frac{u_{yy}}{u_y})^{\sigma_1} \Big) \\[10pt]\quad

 \cdot\Big(\sum\limits_{m=0}^{\sigma_2-1} W_{\alpha_2 +(0,1),\sigma_2-1-m} (\frac{u_{yy}}{u_y})^m
 + \partial_x\partial_{t,x}^{\alpha_2} \tilde{u}(\frac{u_{yy}}{u_y})^{\sigma_2} \Big)  \\[10pt]\quad

 + \sum\limits_{|\alpha_1|\leq |\alpha|-1}\Big(\sum\limits_{m=0}^{\sigma_1-1} W_{\alpha_1+(0,1),\sigma_1-1-m} (\frac{u_{yy}}{u_y})^m
 + \partial_x \partial_{t,x}^{\alpha_1} \tilde{u}(\frac{u_{yy}}{u_y})^{\sigma_1}\Big) \\[10pt]\quad

 \cdot\Big(\sum\limits_{m=0}^{\sigma_2-1} W_{\alpha_2,\sigma_2-1-m} (\frac{u_{yy}}{u_y})^m
 + \partial_{t,x}^{\alpha_2} \tilde{u}(\frac{u_{yy}}{u_y})^{\sigma_2} \Big), \hspace{2cm}
\end{array}
\end{equation}

Since $\partial_{t,x}^{\alpha}\tilde{u} =\frac{1}{\beta - \frac{u_{yy}}{u_y}}(W_{\alpha} - \beta \partial_{t,x}^{\alpha} U)$,
there exists a polynomial $\mathcal{P}_1$ whose degree is less than or equal to $2$, such that
\begin{equation}\label{Appendix_2_Sigma1_Existence_BC_6}
\begin{array}{ll}
-\partial_y W_{\alpha,\sigma} - \frac{u_{yy}}{u_y}W_{\alpha,\sigma}

= \mathcal{P}_1\Big(\sum\limits_{m=0}^{\sigma_1-1} (W_{\alpha_1+(1,0),\sigma_1-1-m}) (\frac{u_{yy}}{u_y})^m \\[8pt]\quad
 + \frac{1}{\beta - \frac{u_{yy}}{u_y}}(W_{\alpha_1+(1,0)} - \beta \partial_x\partial_{t,x}^{\alpha_1} U)(\frac{u_{yy}}{u_y})^{\sigma_1},

 \sum\limits_{m=0}^{\sigma_2-1} (W_{\alpha_2+(0,1),\sigma_2-1-m} (\frac{u_{yy}}{u_y})^m \\[9pt]\quad
 + \frac{1}{\beta - \frac{u_{yy}}{u_y}}(W_{\alpha_2+(0,1)} - \beta \partial_{t,x}^{\alpha_2} U)(\frac{u_{yy}}{u_y})^{\sigma_2} \Big).
\end{array}
\end{equation}
where $\alpha_1\leq\alpha, \alpha_2\leq\alpha, \sigma_1\leq\sigma,\sigma_2\leq\sigma$.
\end{proof}

\begin{lemma}\label{Appendix_3_Lemma_1}
If $u$ satisfies the Prandtl equations $(\ref{Sect1_PrandtlEq})$, then $\tilde{W} = \omega_y$ satisfies $(\ref{Sect1_VorticityY_Eq})$.
\end{lemma}

\begin{proof}
We have the following transformations of the equations:
\begin{equation}\label{Appendix_3_Eq1}
\begin{array}{ll}
\omega_t + u\omega_x + v\omega_y =\omega_{yy}, \\[7pt]

\partial_t \omega_y + u_y\omega_x + u\partial_x\omega_y + v_y \omega_y + v\partial_y\omega_y =\partial_{yy}\omega_y, \\[7pt]

\tilde{W}_t + \omega\omega_x + u\tilde{W}_x +v\tilde{W}_y -u_x\tilde{W}  =\tilde{W}_{yy}, \\[6pt]

\tilde{W}_t + u\tilde{W}_x +v\tilde{W}_y -u_x\tilde{W} - \omega_x\int_y^{+\infty} \tilde{W}\,\mathrm{d}\tilde{y} =\tilde{W}_{yy}.
\end{array}
\end{equation}

\vspace{-0.3cm}
We have the boundary condition for $\tilde{W}$:
\begin{equation}\label{Appendix_3_Eq2}
\begin{array}{ll}
\tilde{W}_y = \omega_{yy} = u_{yyy} = u_{yt} + (u u_x)_y + (v u_y)_y = u_{yt} + u u_{yx}.
\end{array}
\end{equation}
When $\beta =+\infty$, $\tilde{W}_y = \partial_y p_x =0$, is compatible with $(\ref{Appendix_3_Eq2})$.

Thus, Lemma $\ref{Appendix_3_Lemma_1}$ is proved.
\end{proof}

The last lemma is used to prove the stability of the Prandtl equations.
\begin{lemma}\label{Appendix_4_Lemma_1}
Assume $u^1,u^2$ are two Prandtl solutions, $\delta u=u^1-u^2$ satisfies the Prandtl system $(\ref{Sect1_SolDifference_Eq})$,
$W_{\alpha,\sigma}$ satisfies the system $(\ref{Sect1_Stability_VorticityEq_1})$.
\end{lemma}

\begin{proof}
We have the following transformation of the equations:
\begin{equation}\label{Appendix_4_Stability_VorticityEq_1}
\begin{array}{ll}
(\delta u)_t + \bar{u} (\delta u)_x + (\delta u) \bar{u}_x + \bar{v}(\delta u)_y + (\delta v) \bar{u}_y - \partial_{yy}(\delta u) = 0, \\[11pt]

\frac{\partial_t\partial_{t,x}^{\alpha}\partial_y^{\sigma} \delta u}{\bar{u}_y}
+ \frac{\bar{u} \partial_x\partial_{t,x}^{\alpha}\partial_y^{\sigma} \delta u}{\bar{u}_y}
+ \frac{\partial_{t,x}^{\alpha}\partial_y^{\sigma}\delta u \bar{u}_x}{\bar{u}_y}
+ \frac{\bar{v}\partial_y\partial_{t,x}^{\alpha}\partial_y^{\sigma} \delta u}{\bar{u}_y}
+ \partial_{t,x}^{\alpha}\partial_y^{\sigma}\delta v
- \frac{\partial_{yy}\partial_{t,x}^{\alpha}\partial_y^{\sigma} \delta u}{\bar{u}_y} \\[7pt]\quad
= -\frac{[\partial_{t,x}^{\alpha}\partial_y^{\sigma},\, \bar{u} \partial_x]\delta u}{\bar{u}_y}
-\frac{[\partial_{t,x}^{\alpha}\partial_y^{\sigma},\, \bar{u}_x]\delta u}{\bar{u}_y}
-\frac{[\partial_{t,x}^{\alpha}\partial_y^{\sigma},\, \bar{v}\partial_y]\delta u}{\bar{u}_y}
-\frac{[\partial_{t,x}^{\alpha}\partial_y^{\sigma},\, \bar{u}_y]\delta v}{\bar{u}_y}, \\[13pt]

\partial_y\big(\frac{\partial_t\partial_{t,x}^{\alpha}\partial_y^{\sigma} \delta u}{\bar{u}_y}\big)
+ \bar{u}\partial_y\big(\frac{\partial_x\partial_{t,x}^{\alpha}\partial_y^{\sigma} \delta u}{\bar{u}_y}\big)
+ \partial_x\partial_{t,x}^{\alpha}\partial_y^{\sigma} \delta u
+ \frac{\bar{u}_x\partial_y\partial_{t,x}^{\alpha}\partial_y^{\sigma}\delta u}{\bar{u}_y} \\[9pt]\quad
- \bar{u}_x\frac{\partial_{t,x}^{\alpha}\partial_y^{\sigma}\delta u}{\bar{u}_y}\frac{\bar{u}_{yy}}{\bar{u}_y}
+ \frac{\partial_{t,x}^{\alpha}\partial_y^{\sigma}\delta u }{\bar{u}_y}\partial_{xy}\bar{u}

+ \frac{\bar{v}_y\partial_y\partial_{t,x}^{\alpha}\partial_y^{\sigma} \delta u}{\bar{u}_y}
+ \bar{v}\partial_y\big(\frac{\partial_y\partial_{t,x}^{\alpha}\partial_y^{\sigma} \delta u}{\bar{u}_y}\big) \\[9pt]\quad
+ \partial_y\partial_{t,x}^{\alpha}\partial_y^{\sigma}\delta v

- \partial_y\big(\frac{\partial_{yy}\partial_{t,x}^{\alpha}\partial_y^{\sigma} \delta u}{\bar{u}_y}\big)
= -\partial_y\big(\frac{[\partial_{t,x}^{\alpha}\partial_y^{\sigma},\, \bar{u} \partial_x]\delta u}{\bar{u}_y}\big)
-\partial_y\big(\frac{[\partial_{t,x}^{\alpha}\partial_y^{\sigma},\, \bar{u}_x]\delta u}{\bar{u}_y}\big) \\[9pt]\quad
-\partial_y\big(\frac{[\partial_{t,x}^{\alpha}\partial_y^{\sigma},\, \bar{v}\partial_y]\delta u}{\bar{u}_y}\big)
-\partial_y\big(\frac{[\partial_{t,x}^{\alpha}\partial_y^{\sigma},\, \bar{u}_y]\delta v}{\bar{u}_y}\big), \\[14pt]

\bar{u}_y\partial_y\big(\frac{\partial_t\partial_{t,x}^{\alpha}\partial_y^{\sigma} \delta u}{\bar{u}_y}\big)
+ \bar{u}\bar{u}_y\partial_y\big(\frac{\partial_x\partial_{t,x}^{\alpha}\partial_y^{\sigma} \delta u}{\bar{u}_y}\big)
+ \bar{v}\bar{u}_y\partial_y\big(\frac{\partial_y\partial_{t,x}^{\alpha}\partial_y^{\sigma} \delta u}{\bar{u}_y}\big)
- \bar{u}_y\partial_y\big(\frac{\partial_{yy}\partial_{t,x}^{\alpha}\partial_y^{\sigma} \delta u}{\bar{u}_y}\big)
\\[9pt]\quad
+ \partial_{t,x}^{\alpha}\partial_y^{\sigma}\delta u (\partial_{xy}\bar{u} - \bar{u}_x \frac{\bar{u}_{yy}}{\bar{u}_y})

= -\bar{u}_y\partial_y\big(\frac{[\partial_{t,x}^{\alpha}\partial_y^{\sigma},\, \bar{u} \partial_x]\delta u}{\bar{u}_y}\big)
-\bar{u}_y\partial_y\big(\frac{[\partial_{t,x}^{\alpha}\partial_y^{\sigma},\, \bar{u}_x]\delta u}{\bar{u}_y}\big) \\[9pt]\quad
-\bar{u}_y\partial_y\big(\frac{[\partial_{t,x}^{\alpha}\partial_y^{\sigma},\, \bar{v}\partial_y]\delta u}{\bar{u}_y}\big)
-\bar{u}_y\partial_y\big(\frac{[\partial_{t,x}^{\alpha}\partial_y^{\sigma},\, \bar{u}_y]\delta v}{\bar{u}_y}\big),
\end{array}
\end{equation}

We calculate the following four terms:

Similar to $(\ref{Appendix_1_Existence_VorticityEq_2_1})$, we have
\begin{equation}\label{Appendix_4_Stability_VorticityEq_2}
\begin{array}{ll}
\bar{u}_y\partial_y\frac{\partial_{yy}\partial_{t,x}^{\alpha}\partial_y^{\sigma}\delta u}{\bar{u}_y}
= \bar{u}_y\partial_y\big[\partial_y\frac{\partial_y \partial_{t,x}^{\alpha}\partial_y^{\sigma}\delta u}{\bar{u}_y} + \frac{\partial_y \partial_{t,x}^{\alpha}\partial_y^{\sigma}\delta u}{\bar{u}_y}\frac{\bar{u}_{yy}}{\bar{u}_y}\big] \\[10pt]

= \bar{u}_y\partial_y\big[\partial_y(\partial_y\frac{\partial_{t,x}^{\alpha}\partial_y^{\sigma}\delta u}{\bar{u}_y}
+ \frac{\partial_{t,x}^{\alpha}\partial_y^{\sigma}\delta u}{\bar{u}_y}\frac{\bar{u}_{yy}}{\bar{u}_y})
+ (\partial_y\frac{\partial_{t,x}^{\alpha}\partial_y^{\sigma}\delta u}{\bar{u}_y}
+ \frac{\partial_{t,x}^{\alpha}\partial_y^{\sigma}\delta u}{\bar{u}_y}\frac{\bar{u}_{yy}}{\bar{u}_y})
\frac{\bar{u}_{yy}}{\bar{u}_y}\big] \\[10pt]

= \bar{u}_y\partial_y\big[\frac{1}{\bar{u}_y}\partial_y \W_{\alpha,\sigma}
+ \frac{\partial_{t,x}^{\alpha}\partial_y^{\sigma}\delta u}{\bar{u}_y}\partial_y(\frac{\bar{u}_{yy}}{\bar{u}_y})]
+ \bar{u}_y\partial_y\big[\frac{\bar{u}_{yy}}{(\bar{u}_y)^2}\W_{\alpha,\sigma}
+ \frac{\partial_{t,x}^{\alpha}\partial_y^{\sigma}\delta u}{\bar{u}_y}(\frac{\bar{u}_{yy}}{\bar{u}_y})^2\big] \\[10pt]

= \bar{u}_y\partial_y\big[\frac{1}{\bar{u}_y}\partial_y \W_{\alpha,\sigma}
+ \frac{\bar{u}_{yy}}{(\bar{u}_y)^2}\W_{\alpha,\sigma}]
+ \bar{u}_y\partial_y\big[\frac{\partial_{t,x}^{\alpha}\partial_y^{\sigma}\delta u}{\bar{u}_y}\partial_y(\frac{\bar{u}_{yy}}{\bar{u}_y})
+ \frac{\partial_{t,x}^{\alpha}\partial_y^{\sigma}\delta u}{\bar{u}_y}(\frac{\bar{u}_{yy}}{\bar{u}_y})^2\big] \\[10pt]

= \partial_{yy} \W_{\alpha,\sigma}
+ (\frac{\bar{u}_{yyy}}{\bar{u}_y} - 2(\frac{\bar{u}_{yy}}{\bar{u}_y})^2)\W_{\alpha,\sigma}
+ \bar{u}_y\partial_y\big[\frac{\partial_{t,x}^{\alpha}\partial_y^{\sigma}\delta u}{\bar{u}_y} \frac{\bar{u}_{yyy}}{\bar{u}_y} \big],
\end{array}
\end{equation}

Similar to $(\ref{Appendix_1_Existence_VorticityEq_2_2})$, we have
\begin{equation}\label{Appendix_4_Stability_VorticityEq_3}
\begin{array}{ll}
\bar{u}_y\partial_y\frac{\partial_t \partial_{t,x}^{\alpha}\partial_y^{\sigma}\delta u}{\bar{u}_y}
= \bar{u}_y\partial_y[\partial_t(\frac{\partial_{t,x}^{\alpha}\partial_y^{\sigma}\delta u}{\bar{u}_y})]
+ \bar{u}_y\partial_y[\frac{\partial_{t,x}^{\alpha}\partial_y^{\sigma}\delta u}{\bar{u}_y}\frac{\bar{u}_{yt}}{\bar{u}_y}] \\[8pt]

=\partial_t \W_{\alpha,\sigma} -\frac{\bar{u}_{yt}}{\bar{u}_y}\W_{\alpha,\sigma}
+ \bar{u}_y\partial_y[\frac{\partial_{t,x}^{\alpha}\partial_y^{\sigma}\delta u}{\bar{u}_y}\frac{\bar{u}_{yt}}{\bar{u}_y}],
\end{array}
\end{equation}

Similar to $(\ref{Appendix_1_Existence_VorticityEq_2_3})$, we have
\begin{equation}\label{Appendix_4_Stability_VorticityEq_4}
\begin{array}{ll}
\bar{u}_y\partial_y\frac{\partial_x \partial_{t,x}^{\alpha}\partial_y^{\sigma}\delta u}{\bar{u}_y}
= \bar{u}_y\partial_y[\partial_x(\frac{\partial_{t,x}^{\alpha}\partial_y^{\sigma}\delta u}{\bar{u}_y})]
+ \bar{u}_y\partial_y[\frac{\partial_{t,x}^{\alpha}\partial_y^{\sigma}\delta u}{\bar{u}_y}\frac{\bar{u}_{yx}}{\bar{u}_y}] \\[8pt]

=\partial_x \W_{\alpha,\sigma} -\frac{\bar{u}_{yx}}{\bar{u}_y}\W_{\alpha,\sigma}
+ \bar{u}_y\partial_y[\frac{\partial_{t,x}^{\alpha}\partial_y^{\sigma}\delta u}{\bar{u}_y}\frac{\bar{u}_{yx}}{\bar{u}_y}],
\end{array}
\end{equation}

Moreover,
\begin{equation}\label{Appendix_4_Stability_VorticityEq_5}
\begin{array}{ll}
\bar{u}_y\partial_y\frac{\partial_y \partial_{t,x}^{\alpha}\partial_y^{\sigma}\delta u}{\bar{u}_y}
= \bar{u}_y\partial_y[\partial_y(\frac{\partial_{t,x}^{\alpha}\partial_y^{\sigma}\delta u}{\bar{u}_y})]
+ \bar{u}_y\partial_y[\frac{\partial_{t,x}^{\alpha}\partial_y^{\sigma}\delta u}{\bar{u}_y}\frac{\bar{u}_{yy}}{\bar{u}_y}] \\[8pt]

=\partial_y \W_{\alpha,\sigma} -\frac{\bar{u}_{yy}}{\bar{u}_y}\W_{\alpha,\sigma}
+ \bar{u}_y\partial_y[\frac{\partial_{t,x}^{\alpha}\partial_y^{\sigma}\delta u}{\bar{u}_y}\frac{\bar{u}_{yy}}{\bar{u}_y}],
\end{array}
\end{equation}

Plug $(\ref{Appendix_4_Stability_VorticityEq_2}), (\ref{Appendix_4_Stability_VorticityEq_3}), (\ref{Appendix_4_Stability_VorticityEq_4}),
(\ref{Appendix_4_Stability_VorticityEq_5})$
into the last equation of $(\ref{Appendix_4_Stability_VorticityEq_1})$, we get
\begin{equation}\label{Appendix_4_Stability_VorticityEq_6}
\begin{array}{ll}
\partial_t \W_{\alpha,\sigma} + \bar{u}\partial_x \W_{\alpha,\sigma} + \bar{v}\partial_y \W_{\alpha,\sigma} - \partial_{yy} \W_{\alpha,\sigma}
+ \W_{\alpha,\sigma}[-\frac{\bar{u}_{yt}+\bar{u}\bar{u}_{yx} +\bar{v}\bar{u}_{yy} + \bar{u}_{yyy}}{\bar{u}_y} \\[8pt]\quad
+ 2(\frac{\bar{u}_{yy}}{\bar{u}_y})^2]

+ \bar{u}_y\partial_y[\frac{\partial_{t,x}^{\alpha}\partial_y^{\sigma}\delta u}{\bar{u}_y}\frac{\bar{u}_{yt}
+ \bar{u}\bar{u}_{yx} + \bar{v}\bar{u}_{yy} - \bar{u}_{yyy}}{\bar{u}_y}]

= -\bar{u}_y\partial_y\big(\frac{[\partial_{t,x}^{\alpha}\partial_y^{\sigma},\, \bar{u} \partial_x]\delta u}{\bar{u}_y}\big) \\[8pt]\quad
-\bar{u}_y\partial_y\big(\frac{[\partial_{t,x}^{\alpha}\partial_y^{\sigma},\, \bar{u}_x]\delta u}{\bar{u}_y}\big)
-\bar{u}_y\partial_y\big(\frac{[\partial_{t,x}^{\alpha}\partial_y^{\sigma},\, \bar{v}\partial_y]\delta u}{\bar{u}_y}\big)
-\bar{u}_y\partial_y\big(\frac{[\partial_{t,x}^{\alpha}\partial_y^{\sigma},\, \bar{u}_y]\delta v}{\bar{u}_y}\big),
\end{array}
\end{equation}

Next we derive the boundary condition for $\W_{\alpha,\sigma}$. When $\sigma=0$, we denote $\W_{\alpha} = \W_{\alpha,0}$, then we firstly derive the boundary condition for $\W_{\alpha}$. Since $(\delta v)|_{y=0}=\bar{v}|_{y=0}=0$, the following equations hold on the boundary $\{y=0\}$:
\begin{equation}\label{Appendix_4_Stability0_BC_0}
\left\{\begin{array}{ll}
(\delta u)_t + \bar{u} (\delta u)_x + (\delta u) \bar{u}_x - (\delta u)_{yy} = 0,\\[8pt]
(\delta u)_y = \beta (\delta u).
\end{array}\right.
\end{equation}

Apply the tangential differential operator $\partial_{t,x}^{\alpha}$ to $(\ref{Appendix_4_Stability0_BC_0})_1$, we get
\begin{equation}\label{Appendix_4_Stability0_BC_1}
\begin{array}{ll}
\partial_t\partial_{t,x}^{\alpha} \delta u + \bar{u} \partial_x\partial_{t,x}^{\alpha} \delta u + \partial_{t,x}^{\alpha} \delta u \bar{u}_x
- \partial_{yy}\partial_{t,x}^{\alpha}\delta u = -[\partial_{t,x}^{\alpha},\,\bar{u} \partial_x]\delta u-[\partial_{t,x}^{\alpha},\,\bar{u}_x]\delta u.
\end{array}
\end{equation}

Apply the tangential differential operator $\partial_{t,x}^{\alpha}$ to $(\ref{Appendix_4_Stability0_BC_0})_2$, we get
\begin{equation}\label{Appendix_4_Stability0_BC_2}
\begin{array}{ll}
\partial_y\partial_{t,x}^{\alpha} \delta u = \beta\partial_{t,x}^{\alpha} \delta u.
\end{array}
\end{equation}

On the boundary $\{y=0\}$, $\W_{\alpha}|_{y=0}$ and $\partial_{t,x}^{\alpha} \delta u|_{y=0}$ have the relationship:
\begin{equation*}
\begin{array}{ll}
\W_{\alpha} = \bar{u}_y\partial_y(\frac{\partial_{t,x}^{\alpha} \delta u}{\bar{u}_y})
= \partial_y\partial_{t,x}^{\alpha} \delta u
-\partial_{t,x}^{\alpha} \delta u \frac{\bar{u}_{yy}}{\bar{u}_y}
= \partial_{t,x}^{\alpha}\delta u (\beta -\frac{\bar{u}_{yy}}{\bar{u}_y}).
\end{array}
\end{equation*}
Namely,
\begin{equation}\label{Appendix_4_Stability0_BC_3}
\begin{array}{ll}
\partial_{t,x}^{\alpha} \delta u|_{y=0} = \frac{\W_{\alpha}}{\beta-\frac{\bar{u}_{yy}}{\bar{u}_y} }\big|_{y=0}.
\end{array}
\end{equation}

Similar to $(\ref{Appendix_1_Existence_BC_5})$, we calculate $\partial_{yy} \partial_{t,x}^{\alpha}\delta u$:
\begin{equation}\label{Appendix_4_Stability0_BC_4}
\begin{array}{ll}
\partial_{yy} \partial_{t,x}^{\alpha}\delta u = \partial_y[\partial_y(\frac{\partial_{t,x}^{\alpha}\delta u}{\bar{u}_y}\bar{u}_y)]
= \partial_y[\partial_y(\frac{\partial_{t,x}^{\alpha}\delta u}{\bar{u}_y})\bar{u}_y
+ \frac{\partial_{t,x}^{\alpha}\delta u}{\bar{u}_y}\bar{u}_{yy}] \\[9pt]

= \partial_y \W_{\alpha} + \frac{\bar{u}_{yy}}{\bar{u}_y}\W_{\alpha} + \partial_{t,x}^{\alpha}\delta u\frac{\bar{u}_{yyy}}{\bar{u}_y}

= \partial_y \W_{\alpha} + \frac{\bar{u}_{yy}}{\bar{u}_y}\W_{\alpha} + \frac{\bar{u}_{yyy}}{\bar{u}_y}\cdot
\frac{1}{\beta - \frac{\bar{u}_{yy}}{\bar{u}_y}}\W_{\alpha},
\end{array}
\end{equation}

Similar to $(\ref{Appendix_1_Existence_BC_6})$, we calculate $\partial_t\partial_{t,x}^{\alpha} \delta u$:
\begin{equation}\label{Appendix_4_Stability0_BC_5}
\begin{array}{ll}
\partial_t\partial_{t,x}^{\alpha} \delta u
= \partial_t\Big[\frac{1}{\beta - \frac{\bar{u}_{yy}}{\bar{u}_y}}\W_{\alpha}\Big]

= \frac{1}{\beta - \frac{\bar{u}_{yy}}{\bar{u}_y}}\partial_t \W_{\alpha}
+ \frac{(\frac{\bar{u}_{yy}}{\bar{u}_y})_t}{(\beta - \frac{\bar{u}_{yy}}{\bar{u}_y})^2}\W_{\alpha}.
\end{array}
\end{equation}

Similar to $(\ref{Appendix_1_Existence_BC_7})$, we calculate $\partial_x\partial_{t,x}^{\alpha} \delta u$:
\begin{equation}\label{Appendix_4_Stability0_BC_6}
\begin{array}{ll}
\partial_x\partial_{t,x}^{\alpha} \delta u
= \partial_x\Big[\frac{1}{\beta - \frac{\bar{u}_{yy}}{\bar{u}_y}}\W_{\alpha}\Big]

= \frac{1}{\beta - \frac{\bar{u}_{yy}}{\bar{u}_y}}\partial_x \W_{\alpha}
+ \frac{(\frac{\bar{u}_{yy}}{\bar{u}_y})_x}{(\beta - \frac{\bar{u}_{yy}}{\bar{u}_y})^2}\W_{\alpha}.
\end{array}
\end{equation}

It follows from $(\ref{Appendix_4_Stability0_BC_1})$ that
\begin{equation}\label{Appendix_4_Stability0_BC_7}
\begin{array}{ll}
\frac{1}{\beta - \frac{\bar{u}_{yy}}{\bar{u}_y}}(\partial_t \W_{\alpha} + \bar{u}\partial_x \W_{\alpha})
- \partial_y \W_{\alpha} - \frac{\bar{u}_{yy}}{\bar{u}_y}\W_{\alpha} \\[9pt]\quad

+ \frac{1}{\beta - \frac{\bar{u}_{yy}}{\bar{u}_y}}\W_{\alpha}
\big[ \frac{(\frac{\bar{u}_{yy}}{\bar{u}_y})_t}{\beta - \frac{\bar{u}_{yy}}{\bar{u}_y}}
+ \frac{\bar{u}(\frac{\bar{u}_{yy}}{\bar{u}_y})_x}{\beta - \frac{\bar{u}_{yy}}{\bar{u}_y}}
- \frac{\bar{u}_{yyy}}{\bar{u}_y} + \bar{u}_x
\big] \\[12pt]
 = -\sum\limits_{\alpha_1>0}\partial_{t,x}^{\alpha_1}\bar{u}\partial_{t,x}^{\alpha_2}\partial_x\delta u
 -\sum\limits_{\alpha_1>0}\partial_{t,x}^{\alpha_1}\partial_x\bar{u} \partial_{t,x}^{\alpha_2}\delta u

 \\[12pt]
 = -\sum\limits_{\alpha_1>0}\partial_{t,x}^{\alpha_1}\bar{u}\cdot \frac{\W_{\alpha_2+(0,1)}}{\beta-\frac{\bar{u}_{yy}}{\bar{u}_y} }
 -\sum\limits_{\alpha_1>0}\partial_{t,x}^{\alpha_1}\partial_x\bar{u} \cdot \frac{\W_{\alpha_2}}{\beta-\frac{\bar{u}_{yy}}{\bar{u}_y} }.
\end{array}
\end{equation}

If $\beta=+\infty$, we have the equations on the boundary: $\partial_{yy} \partial_{t,x}^{\alpha} \delta u = 0$.
It follows from $(\ref{Appendix_4_Stability0_BC_4})$ that
\begin{equation}\label{Appendix_4_Stability0_Dirichlet_2}
\begin{array}{ll}
\partial_y \W_{\alpha} + 2\frac{\bar{u}_{yy}}{\bar{u}_y}\W_{\alpha} =0.
\end{array}
\end{equation}
Let $\beta\rto +\infty$ in $(\ref{Appendix_4_Stability0_BC_7})$, we get exactly $(\ref{Appendix_4_Stability0_Dirichlet_2})$.

When $\sigma\geq 1$, $\partial_{t,x}^{\alpha}\partial_y^{\sigma}((\delta u)_y -\beta (\delta u))|_{y=0}$ may not be zero.
Similar to $(\ref{Appendix_4_Stability0_BC_4})$, we calculate $\partial_{yy} \partial_{t,x}^{\alpha}\partial_y^{\sigma}\delta u$:
\begin{equation}\label{Appendix_4_Stability1_BC_1}
\begin{array}{ll}
\partial_{yy} \partial_{t,x}^{\alpha}\partial_y^{\sigma}\delta u = \partial_y[\partial_y(\frac{\partial_{t,x}^{\alpha}\partial_y^{\sigma}\delta u}
{\bar{u}_y}\bar{u}_y)]
= \partial_y[\partial_y(\frac{\partial_{t,x}^{\alpha}\partial_y^{\sigma} \delta u}{\bar{u}_y})\bar{u}_y
+ \frac{\partial_{t,x}^{\alpha}\partial_y^{\sigma}\delta u}{\bar{u}_y}\bar{u}_{yy}] \\[9pt]

= \partial_y \W_{\alpha,\sigma} + \frac{\bar{u}_{yy}}{\bar{u}_y}\W_{\alpha,\sigma}
+ \partial_{t,x}^{\alpha}\partial_y^{\sigma}\delta u\frac{\bar{u}_{yyy}}{\bar{u}_y},
\end{array}
\end{equation}

When $\sigma\geq 1$, we have the following transformation of the equations on the boundary:
\begin{equation}\label{Appendix_4_Stability1_BC_2}
\begin{array}{ll}
(\delta u)_t + \bar{u} (\delta u)_x + (\delta u) \bar{u}_x + \bar{v}(\delta u)_y + (\delta v) \bar{u}_y - \partial_{yy}(\delta u) = 0, \\[11pt]

\partial_t\partial_{t,x}^{\alpha}\partial_y^{\sigma} \delta u + \bar{u} \partial_x\partial_{t,x}^{\alpha}\partial_y^{\sigma} \delta u + \partial_{t,x}^{\alpha}\partial_y^{\sigma}\delta u \bar{u}_x
+ \bar{v}\partial_y\partial_{t,x}^{\alpha}\partial_y^{\sigma} \delta u + \partial_{t,x}^{\alpha}\partial_y^{\sigma}\delta v \bar{u}_y \\[8pt]\quad
- \partial_{yy}\partial_{t,x}^{\alpha}\partial_y^{\sigma} \delta u
= -[\partial_{t,x}^{\alpha}\partial_y^{\sigma},\, \bar{u} \partial_x]\delta u
-[\partial_{t,x}^{\alpha}\partial_y^{\sigma},\, \bar{u}_x]\delta u
-[\partial_{t,x}^{\alpha}\partial_y^{\sigma},\, \bar{v}\partial_y]\delta u \\[8pt]\quad
-[\partial_{t,x}^{\alpha}\partial_y^{\sigma},\, \bar{u}_y]\delta v,
\end{array}
\end{equation}

Since $\partial_{t,x}^{\alpha^{\prime}}\delta v|_{y=0} = \partial_{t,x}^{\alpha^{\prime}}\bar{v}|_{y=0} =0$,
\begin{equation*}
\begin{array}{ll}
\partial_t\partial_{t,x}^{\alpha}\partial_y^{\sigma} \delta u + \bar{u} \partial_x\partial_{t,x}^{\alpha}\partial_y^{\sigma} \delta u + \partial_{t,x}^{\alpha}\partial_y^{\sigma}\delta u \bar{u}_x
+ \partial_{t,x}^{\alpha}\partial_y^{\sigma}\delta v \bar{u}_y
- \partial_{yy}\partial_{t,x}^{\alpha}\partial_y^{\sigma} \delta u \\[9pt]

= - \sum\limits_{|\alpha_1|+\sigma_1>0}\partial_{t,x}^{\alpha_1}\partial_y^{\sigma_1}\bar{u} \partial_x\partial_{t,x}^{\alpha_2}\partial_y^{\sigma_2}\delta u
-\sum\limits_{|\alpha_1|+\sigma_1>0}\partial_{t,x}^{\alpha_1}\partial_y^{\sigma_1}\partial_x\bar{u}\partial_{t,x}^{\alpha_2}\partial_y^{\sigma_2}\delta u \\[12pt]\quad
-\sum\limits_{\sigma_1>0}\partial_{t,x}^{\alpha_1}\partial_y^{\sigma_1}\bar{v}\partial_{t,x}^{\alpha_2}\partial_y^{\sigma_2+1}\delta u
-\sum\limits_{|\alpha_1|+\sigma_1>0,\sigma_2>0}\partial_{t,x}^{\alpha_1}\partial_y^{\sigma_1+1}\bar{u}\partial_{t,x}^{\alpha_2}\partial_y^{\sigma_2}\delta v
\end{array}
\end{equation*}

\begin{equation}\label{Appendix_4_Stability1_BC_3}
\begin{array}{ll}
= - \sum\limits_{|\alpha_1|+\sigma_1>0}\partial_{t,x}^{\alpha_1}\partial_y^{\sigma_1}\bar{u} \partial_x\partial_{t,x}^{\alpha_2}\partial_y^{\sigma_2}\delta u
-\sum\limits_{|\alpha_1|+\sigma_1>0}\partial_{t,x}^{\alpha_1}\partial_y^{\sigma_1}\partial_x\bar{u}\partial_{t,x}^{\alpha_2}\partial_y^{\sigma_2}\delta u \\[12pt]\quad
+\sum\limits_{\sigma_1>0}\partial_{t,x}^{\alpha_1}\partial_y^{\sigma_1-1}\partial_x\bar{u}\partial_{t,x}^{\alpha_2}\partial_y^{\sigma_2+1}\delta u
+\sum\limits_{|\alpha_1|+\sigma_1>0,\sigma_2>0}\partial_{t,x}^{\alpha_1}\partial_y^{\sigma_1+1}\bar{u}\partial_{t,x}^{\alpha_2}\partial_y^{\sigma_2-1}
\partial_x\delta u \\[13pt]

= - \sum\limits_{|\alpha_1|+\sigma_1>0}\partial_{t,x}^{\alpha_1}\partial_y^{\sigma_1}\bar{u} \partial_x\partial_{t,x}^{\alpha_2}\partial_y^{\sigma_2}\delta u
-\sum\limits_{|\alpha_1|+\sigma_1>0}\partial_{t,x}^{\alpha_1}\partial_y^{\sigma_1}\partial_x\bar{u}\partial_{t,x}^{\alpha_2}\partial_y^{\sigma_2}\delta u \\[12pt]\quad
+\sum\limits_{\sigma_2\leq\sigma}\partial_{t,x}^{\alpha_1}\partial_y^{\sigma_1}\partial_x\bar{u}\partial_{t,x}^{\alpha_2}\partial_y^{\sigma_2}\delta u
+\sum\limits_{\sigma_2> 0}\partial_{t,x}^{\alpha_1}\partial_y^{\sigma_1+1}\bar{u}\partial_{t,x}^{\alpha_2}\partial_y^{\sigma_2-1}
\partial_x\delta u,
\end{array}
\end{equation}
then
\begin{equation}\label{Appendix_4_Stability1_BC_4}
\begin{array}{ll}
\partial_{yy}\partial_{t,x}^{\alpha}\partial_y^{\sigma} \delta u =
\partial_t\partial_{t,x}^{\alpha}\partial_y^{\sigma} \delta u + \bar{u} \partial_x\partial_{t,x}^{\alpha}\partial_y^{\sigma} \delta u + \partial_{t,x}^{\alpha}\partial_y^{\sigma}\delta u \bar{u}_x
+ \partial_{t,x}^{\alpha}\partial_y^{\sigma}\delta v \bar{u}_y \\[10pt]\quad

+ \sum\limits_{|\alpha_1|+\sigma_1>0}\partial_{t,x}^{\alpha_1}\partial_y^{\sigma_1}\bar{u} \partial_x\partial_{t,x}^{\alpha_2}\partial_y^{\sigma_2}\delta u
+ \sum\limits_{|\alpha_1|+\sigma_1>0}\partial_{t,x}^{\alpha_1}\partial_y^{\sigma_1}\partial_x\bar{u}\partial_{t,x}^{\alpha_2}\partial_y^{\sigma_2}\delta u \\[13pt]\quad
-\sum\limits_{\sigma_2\leq\sigma}\partial_{t,x}^{\alpha_1}\partial_y^{\sigma_1}\partial_x\bar{u}\partial_{t,x}^{\alpha_2}\partial_y^{\sigma_2}\delta u
-\sum\limits_{\sigma_2> 0}\partial_{t,x}^{\alpha_1}\partial_y^{\sigma_1+1}\bar{u}\partial_{t,x}^{\alpha_2}\partial_y^{\sigma_2-1}
\partial_x\delta u.
\end{array}
\end{equation}

By the induction method, we have
\begin{equation}\label{Appendix_4_Stability1_BC_5}
\begin{array}{ll}
\partial_t \partial_{t,x}^{\alpha}\partial_y^{\sigma} \delta u
= \partial_y (\frac{\partial_t \partial_{t,x}^{\alpha}\partial_y^{\sigma-1} \delta u}{\bar{u}_y}\bar{u}_y)
= \W_{\alpha+(1,0),\sigma-1} + \partial_t \partial_{t,x}^{\alpha}\partial_y^{\sigma-1} \delta u\frac{\bar{u}_{yy}}{\bar{u}_y} \\[9pt]

= \W_{\alpha+(1,0),\sigma-1} + \W_{\alpha+(1,0),\sigma-2} \frac{\bar{u}_{yy}}{\bar{u}_y}
 + \partial_t \partial_{t,x}^{\alpha}\partial_y^{\sigma-2} \delta u(\frac{\bar{u}_{yy}}{\bar{u}_y})^2 = \cdots \\[8pt]

= \sum\limits_{m=0}^{\sigma-1} \W_{\alpha+(1,0),\sigma-1-m} (\frac{\bar{u}_{yy}}{\bar{u}_y})^m
 + \partial_t \partial_{t,x}^{\alpha} \delta u(\frac{\bar{u}_{yy}}{\bar{u}_y})^{\sigma}.
\end{array}
\end{equation}

Similarly,
\begin{equation}\label{Appendix_4_Stability1_BC_6}
\begin{array}{ll}
\partial_x \partial_{t,x}^{\alpha}\partial_y^{\sigma} \delta u

= \sum\limits_{m=0}^{\sigma-1} \W_{\alpha+(0,1),\sigma-1-m} (\frac{\bar{u}_{yy}}{\bar{u}_y})^m
 + \partial_x \partial_{t,x}^{\alpha} \delta u(\frac{\bar{u}_{yy}}{\bar{u}_y})^{\sigma}.
\end{array}
\end{equation}

Plug $(\ref{Appendix_4_Stability1_BC_1})$, $(\ref{Appendix_4_Stability1_BC_5})$, $(\ref{Appendix_4_Stability1_BC_6})$ into $(\ref{Appendix_4_Stability1_BC_4})$, we have
\begin{equation}\label{Appendix_4_Stability1_BC_7}
\begin{array}{ll}
\partial_y \W_{\alpha,\sigma} + \frac{\bar{u}_{yy}}{\bar{u}_y}\W_{\alpha,\sigma}

= \sum\limits_{m=0}^{\sigma-1} (\W_{\alpha+(1,0),\sigma-1-m} + \bar{u} \W_{\alpha+(0,1),\sigma-1-m}) (\frac{\bar{u}_{yy}}{\bar{u}_y})^m \\[8pt]\quad
 + \partial_t \partial_{t,x}^{\alpha} \delta u(\frac{\bar{u}_{yy}}{\bar{u}_y})^{\sigma}
 + u\partial_x \partial_{t,x}^{\alpha} \delta u(\frac{\bar{u}_{yy}}{\bar{u}_y})^{\sigma}

 - \sum\limits_{m=0}^{\sigma-1} \W_{\alpha,\sigma-1-m} (\frac{\bar{u}_{yy}}{\bar{u}_y})^m (\frac{\bar{u}_{yyy}}{\bar{u}_y} - \bar{u}_x)\\[8pt]\quad
 - \partial_{t,x}^{\alpha} \delta u(\frac{\bar{u}_{yy}}{\bar{u}_y})^{\sigma} (\frac{\bar{u}_{yyy}}{\bar{u}_y} - \bar{u}_x)

 +\bar{u}_y \sum\limits_{m=0}^{\sigma-2}(\W_{\alpha+(1,0),\sigma-2-m} (\frac{\bar{u}_{yy}}{\bar{u}_y})^m \\[8pt]\quad
 + \partial_{t,x}^{\alpha}\partial_x \delta u(\frac{\bar{u}_{yy}}{\bar{u}_y})^{\sigma-1})

 - \bar{u}_y \partial_{t,x}^{\alpha}\partial_x\delta u
 +\mathcal{P}_3\Big(\sum\limits_{m=0}^{\sigma_1} (W_{\alpha_1+(0,1),\sigma_1-m}) (\frac{\bar{u}_{yy}}{\bar{u}_y})^m \\[8pt]\quad
 + \partial_x \partial_{t,x}^{\alpha_1} (\bar{u}-U)(\frac{\bar{u}_{yy}}{\bar{u}_y})^{\sigma_1},

 \sum\limits_{m=0}^{\sigma_2-1} (\W_{\alpha_2,\sigma_2-1-m} (\frac{\bar{u}_{yy}}{\bar{u}_y})^m
 + \partial_{t,x}^{\alpha_2} \delta u(\frac{\bar{u}_{yy}}{\bar{u}_y})^{\sigma_2} \Big),
\end{array}
\end{equation}
where $\mathcal{P}_3(\cdot,\cdot)$ is a quadratic polynomial, whose explicit form is determined by the last four terms of
$(\ref{Appendix_4_Stability1_BC_4})$. We rewrite $(\ref{Appendix_4_Stability1_BC_7})$ into the following form:
\begin{equation}\label{Appendix_4_Stability1_BC_8}
\begin{array}{ll}
-\partial_y \W_{\alpha,\sigma} - \frac{\bar{u}_{yy}}{\bar{u}_y}\W_{\alpha,\sigma} \\[6pt]

= \mathcal{P}_2\Big(\mathcal{L}+ \sum\limits_{|\alpha^{\prime}|\leq |\alpha|+1}\sum\limits_{m=0}^{\sigma} (W_{\alpha^{\prime},\sigma-m}) (\frac{\bar{u}_{yy}}{\bar{u}_y})^m
 + \partial_x \partial_{t,x}^{\alpha^{\prime}} (\bar{u}-U)(\frac{\bar{u}_{yy}}{\bar{u}_y})^{\sigma}, \\[10pt]\quad

\sum\limits_{|\alpha^{\prime}|\leq |\alpha|+1}\sum\limits_{m=0}^{\sigma-1} (\W_{\alpha^{\prime},\sigma-1-m}) (\frac{\bar{u}_{yy}}{\bar{u}_y})^m
 + \sum\limits_{|\alpha^{\prime}|\leq |\alpha|+1}\frac{1}{\beta - \frac{\bar{u}_{yy}}{\bar{u}_y}}\W_{\alpha^{\prime}}(\frac{\bar{u}_{yy}}{\bar{u}_y})^{\sigma}\Big),
\end{array}
\end{equation}
where $\mathcal{P}_2$ is a polynomial with lower degrees.

Thus, Lemma $\ref{Appendix_4_Lemma_1}$ is proved.
\end{proof}

%%% find 7
\section*{Acknowledgements}
%The author is grateful to anonymous referees for their many helpful suggestions.
The author is
%also
thankful for the support of the Center of Mathematical Sciences and Applications, Harvard University.
This paper is supported by the scholarship of Chinese Scholarship Council (No. 201500090074).

%\newpage
\bibliographystyle{siam}
\addcontentsline{toc}{section}{References}
\bibliography{FuzhouWu_PrandtlEq}

\end{document}